\documentclass{amsart}

\usepackage{amssymb,amsmath}
\usepackage{wasysym}
\usepackage{graphicx}
\usepackage{mathtools}
\usepackage[bottom]{footmisc}
\usepackage[colorlinks=true,linkcolor=Red,citecolor=Green]{hyperref}
\usepackage[a4paper, margin=1in]{geometry}
\usepackage[dvipsnames]{xcolor}
\usepackage[utf8]{inputenc}
\usepackage{siunitx}
\usepackage{dsfont,bookmark}
\usepackage{pgfplots}
\usepackage{csquotes}
\usepackage[caption=false]{subfig}
\usepackage{tikz-3dplot}
\pgfplotsset{compat=1.18}
\usetikzlibrary{decorations.markings}
\usepackage[all]{nowidow}
\usepackage[capitalise, noabbrev]{cleveref}
\usepackage[alphabetic,msc-links]{amsrefs}

\newcommand{\dd}{\mathrm{d}}
\newcommand{\A}{\mathcal{A}}

\newcommand{\DD}{\mathbb{D}}
\newcommand{\EE}{\mathbb{E}}
\newcommand{\HH}{\mathbb{H}}
\newcommand{\PP}{\mathbb{P}}
\newcommand{\RR}{\mathbb{R}}
\newcommand{\ZZ}{\mathbb{Z}}

\newcommand{\Euc}{\mathrm{Euc}}
\newcommand{\prin}{\mathrm{prin}}
\newcommand{\rem}{\mathrm{rem}}

\DeclareMathOperator{\vol}{dVol}
\DeclareMathOperator{\supp}{supp}

\newcommand{\inner}[2]{\left\langle{#1},{#2}\right\rangle}

\renewcommand{\epsilon}{\varepsilon}

\theoremstyle{plain}
\newtheorem{theorem}{Theorem}[section]
\newtheorem*{theorem*}{Theorem}

\theoremstyle{definition}
\newtheorem{definition}{Definition}[section]
\newtheorem{proposition}{Proposition}[section]
\newtheorem{conjecture}{Conjecture}[section]

\theoremstyle{remark}
\newtheorem{lemma}{Lemma}[section]
\newtheorem{remark}{Remark}[section]

\newtheorem*{example*}{Example}

\counterwithin{figure}{section}
\counterwithin{proposition}{section}
\counterwithin{lemma}{section}
\counterwithin{remark}{section}
\counterwithin{corollary}{section}
\newtheorem*{oldproof}{Proof}
\renewenvironment{proof}[1][{}]{\begin{oldproof}[#1]}{\qed\end{oldproof}}

\AddToHook{env/lemma/begin}{\crefalias{theorem}{lemma}}
\AddToHook{env/definition/begin}{\crefalias{theorem}{definition}}
\AddToHook{env/corollary/begin}{\crefalias{theorem}{corollary}}
\AddToHook{env/proposition/begin}{\crefalias{theorem}{proposition}}
\AddToHook{env/conjecture/begin}{\crefalias{theorem}{conjecture}}
\AddToHook{env/remark/begin}{\crefalias{theorem}{remark}}
\counterwithin{equation}{section}
\counterwithin{figure}{section}

\usepackage{algorithm}
\usepackage{algpseudocode}
\usepackage{listings}
\title[MFET on asymptotically hyperbolic surfaces]{
Mean first escape times of Brownian motion on asymptotically hyperbolic \& gas giant metric surfaces}
\date{3rd of December 2026}
\author[Gell-Redman]{Jesse Gell-Redman}
\address{School of Mathematics and Statistics,  The University of Melbourne}
\email{j.gell@unimelb.edu.au}
\author[Godfried]{Emanuel József Godfried}
\address{School of Mathematics and Statistics,  The University of Melbourne}
\email{egodfried@student.unimelb.edu.au}
\author[Tzou]{Justin Tzou}
\address{School of Mathematics \& Statistics, The University of New South Wales}
\email{tzou.justin@gmail.com}
\author[Tzou]{Leo Tzou}
\address{School of Mathematics and Statistics,  The University of Melbourne}
\email{leo.tzou@gmail.com}

\thanks{Corresponding author: Emanuel József Godfried egodfried@student.unimelb.edu.au.\\
J. Gell-Redman is  supported by the Australian Research Council under grant number: DP210103242.\\
E.J.\ Godfried is  supported by the Melbourne Research Scholarship and by Stichting dr Hendrik Muller's Vaderlandsch Fonds.\\
J. Tzou is  supported by the Australian Research Council under grant number: DP220101808.\\
L. Tzou is  supported by the Australian Research Council under grant numbers: DP220101808 and DP260103195.}

\begin{document}

\begin{abstract}
    This paper deals with the mean first escape time of Brownian motion on asymptotically hyperbolic and gas giant surfaces. We show that for a boundary defining function \(\rho\),
    the mean first escape time \(u_\epsilon(x)\) from the truncated Riemannian surface with an asymptotically hyperbolic metric \((M_\epsilon,\bar{g}/\rho^2) =  (\{x\in M:\rho(x)> \epsilon\},\bar{g}/\rho^2) \subset (M,\bar{g}/\rho^2)\) satisfies the asymptotic expansion \(u_\epsilon(x) = -\log \epsilon + \mathcal{O}(1)\) as \(\epsilon\to 0 \). 
    Furthermore, we show that in the case of a gas giant metric \(g = \bar{g}/\rho^\alpha\), where \(\alpha\in (0,2)\), the mean first escape time from the surface \((M_\epsilon,\bar{g}/\rho^\alpha)\) satisfies \(u_\epsilon(x) = \mathcal{O}(1)\) as \(\epsilon\to 0 \).
    Using techniques from the theory of polyhomogeneous conormal functions we explain this difference between in the mean first escape time on gas giant metric surfaces and asymptotically hyperbolic surfaces on the unit disc.
\end{abstract}
\maketitle
\section{Introduction}\label{sec:intro}
We study narrow escape problems for Brownian motion on asymptotically hyperbolic surfaces \(M\), equipped with a conformally compact metric \(g = \bar{g}/\rho^2\), where \(\rho\) is a boundary defining function and \(\bar{g}\) is a smooth, non-degenerate Riemannian metric. In this setting, the narrow escape problem concerns the first time \(\tau_\epsilon^x\) a Brownian particle \(X_t^x\), starting at a point \(x \in M\), escapes the region  
\[
    M_\epsilon = \{x \in M : \rho(x) > \epsilon\} \ .
\]
We then denote the mean first escape time (MFET) from \(M_\epsilon\) of the Brownian motion starting at \(x\in M_\epsilon \) by \(u_\epsilon(x) = \mathbb{E}[\tau_\epsilon^x ]\).

Narrow escape problems have been extensively studied in Euclidean and Riemannian settings, motivated in part by applications in cellular biology, such as in the modelling of {the timescales of} ion diffusion or protein transport (readers may refer to \cites{holcman2004escape,schuss2007NE,Cheviakov2010asymptotic,holcman2014nep} and references therein, with more detailed theory and applications found in textbooks such as \cites{Bressloff2021a,Bressloff2021b,holcman2015stoch}). For example, in a two-dimensional bounded Euclidean domain \(\Omega\), the expected escape time through a small absorbing trap of radius \(\epsilon\) behaves as \({u_\epsilon \sim -|\Omega| \log \epsilon + \mathcal{O}(1)}\), as \(\epsilon \to 0\) \cite{holcman2004escape}. This result and its generalisations have led to a rich literature on related problems, including narrow capture and escape on Riemannian manifolds \cites{singer2006NE, benichou2008narrow, Cheviakov2010asymptotic, Nursultanov2023}. 

More refined expansions also account for geometric features. For instance, in the case of a single circular trap in a Euclidean domain of \(\RR^3\), the expected escape time satisfies  
\[
    u_\epsilon \sim \frac{|\Omega|}{4\epsilon} \left(1 - \frac{\epsilon}{\pi} H \log \epsilon \right) + \mathcal{O}(1),
\]
where \(H\) is the mean curvature at the centre of the trap \cite{singer2008NE}. These expansions have also been extended to Riemannian 3-manifolds using microlocal techniques \cite{nursultanov2021mean}. {These models often assume uniform diffusion, and the timescales are computed as the target or trap decreases in size. However, in certain cellular environments, the diffusion process may be influenced by geometric factors or varying diffusion rates, leading to non-uniform diffusion patterns.}

{In this work, we extend the analysis of narrow escape to two-dimensional asymptotically hyperbolic and gas giant metric manifolds, which can model scenarios where the diffusion rate slows down near the boundary or cellular membrane. Specifically, in the case of an asymptotically hyperbolic manifold, the model assumes a  slowdown  of order \(1/\rho^2\), where \(\rho\) is the distance from the boundary.}
{The central question we address is:  
What is the mean {first} escape time \(u_\epsilon(x)\) of a Brownian motion starting at \(x\) from \(M_\epsilon\), as \(\epsilon \to 0\)? }

Prior work on Brownian motion on hyperbolic spaces has largely been probabilistic and restricted to upper half-space models \(\mathbb{H}^d\) with \cite{Gertsenshtein1959waveguides} being (one of) the first articles on hyperbolic Brownian motion, where the authors show  ``{\dots} how, with the help of Lobachevsky geometry one can obtain a precise solution of the  radio-engineering problem of a statistically inhomogeneous waveguide -- of a waveguide (transmission line) with random inhomogeneities.''
{Similarly, in \cite{comtet1996diffusion}, the authors show ``how the one-dimensional, classical diffusion of a particle in a quenched random potential \ldots is directly related to Brownian motion on the hyperbolic plane.''}
In \cite{Gruet1996semigroup}, the author calculates directly from the Laplace-Beltrami operator the transition probability density, which is then used in papers such as \cites{Matsumoto2010limiting,ovidio2011Bessel} to obtain  large time probability density distributions of the Brownian motion.

Using similar strategies, the hitting probabilities and  exponential decay laws of the hyperbolic Brownian motion on the upper half-plane \(\HH^d \) have been found \cites{Cammarota2013hittingprob,Cammarota2014asymptotic}, and in \cite{Shiozawa2017escape}, the author gives the escape rate of the hyperbolic Brownian motion.

Our approach on the other hand, builds on methods developed in \cite{AMMARI201266} and expanded in articles such as \cites{ nursultanov2021mean, Nursultanov2023}. This allows for robust computation of the mean escape time and higher moments of the associated distribution. Moreover, our techniques can be adapted to more general geometries. {For any compact Riemannian surface \((M,\bar{g})\) with boundary \(\partial M\), we can make \((M,g=\bar{g}/\rho^2)\) into an asymptotically hyperbolic surface. Of this surface, we can then compute the expansion of the  mean first escape time of the Brownian motion starting at \(x\) from the manifold \(M_\epsilon\).}

{Furthermore}, after proving the mean first escape time from an asymptotically hyperbolic manifold, we also consider the mean first escape time from a \emph{gas giant {metric} manifold}. These manifolds, recently introduced in the context of inverse problems by \cite{dehoop2024}, have metrics of the form \(g = \bar{g}/\rho^\alpha\) for \(\alpha \in (0,2)\), with the limit as \(\alpha \to 2 \) being  the asymptotically hyperbolic case. {This corresponds to a slowdown of order \(1/\rho^\alpha\) of the diffusion rate near the boundary, which compared to asymptotically hyperbolic case, is less extreme.} Such manifolds exhibit interesting geometric behaviour: geodesics reach the boundary in finite time, while the volume may be finite or infinite depending on the dimension and \(\alpha\). 
Due to the recent introduction of gas giant geometries, this paper will, to the authors' knowledge, be first treatment of Brownian motion and narrow escape problems in such geometries.

\subsection{Notational conventions}
We work on Riemannian manifolds \((M,g) \) with boundary \(\partial M \) and induced metric \(h= \iota_{\partial M}^* g \in \mathcal{S}^2(\partial M)\). The Laplace-Beltrami operator is given in local coordinates \(x^i\) by 
\begin{equation}
    \Delta_g = \frac{1}{\sqrt{|g|}}\partial_i \left( \sqrt{|g|}g ^{ij}\partial_j  \right)\ ,
\end{equation}
which has non-positive eigenvalues. 
The normal derivative \(\partial_\nu \) is outward pointing, so that Green's formula is given by 
\begin{equation}
    \int_{M}(\Delta_g u) v- u (\Delta_g v) \vol_{g} = \int_{\partial M} (\partial_\nu u) v -u (\partial_\nu v) \vol_h \ . 
\end{equation}

\subsection{The statement of the problem and theorem}
Let \((M,\partial M , \bar{g}) \) be a compact connected orientable smooth non-degenerate Riemannian manifold of dimension two. Let \(\rho(x) : M \to [0,\infty)\) be any boundary defining function. Then the manifold  \((M^\circ , g = \bar{g }/\rho^2 )\) is complete, and the sectional curvatures  approach \(-|d\rho(x)|_{\bar g}\) as  \(\rho(x) \to 0 \). We say that \((M^\circ,g)\) is an asymptotically hyperbolic manifold if the sectional curvatures approach \(-1 \) as  \(\rho (x) \to 0  \).
For \(\epsilon>0 \), let \(M_\epsilon\) be the Riemannian manifold with boundary defined by
\begin{eqnarray}\label{eq: def M epsilon}
 M_\epsilon := \{x\in M : \rho(x)>  \epsilon \} . 
\end{eqnarray} Set \((X_t^x,\PP^x) \) to be the Brownian motion on \((M,g)\) starting at \(x\), and let \(\tau_\epsilon^x \) denote the first time that the Brownian motion \(X_t^x \) is outside \(M_\epsilon\), i.e.\ 
\begin{equation*}
    \tau_\epsilon^x := \inf \{t\geq 0: X_t^x \in M\setminus M_\epsilon\} \ .
\end{equation*}
We wish to study the behaviour of the mean first escape time \(\EE(\tau_\epsilon^x)\) as \(\epsilon\to 0 \). In \cite{nursultanov2021mean}*{Appendix A} it is shown  that \(u_\epsilon(x):= \EE(\tau_\epsilon^x )\) satisfies the  boundary value problem
\begin{equation}
    \Delta_g u_\epsilon(x)|_{x\in M_\epsilon^\circ } = -1 , \quad u_\epsilon(x)|_{x\in \partial M_\epsilon} = 0 \quad \text{for \(u_\epsilon\in H^2(M_\epsilon)\cap H^1_0(M_\epsilon)\)} \ .
    \label{eq:bvp}
\end{equation}
Based on this boundary value problem, we state the main theorem.
\begin{theorem}\label{thm:hyperbolic}
    Assume that \((M^\circ,g) \) is an asymptotically hyperbolic surface, and  \(u_\epsilon(x) \) 
    satisfies the conditions of Equation \eqref{eq:bvp}, then there are functions \(\tilde U _\epsilon(x) \) and \(r_\epsilon(x) \) such that 
    \begin{equation}
        u_\epsilon(x) =-\log\epsilon + \tilde{U}_\epsilon(x)  + r_\epsilon(x)\ ,
    \end{equation}
    where for each fixed \(x\in M_\epsilon \) the function \(\tilde{U}_\epsilon(x) = \mathcal{O}(1)\) as \(\epsilon \to 0 \) and for each fixed \(x\in M_\epsilon \) there is a positive constant \(D_x \) such that  \(|r_\epsilon(x)|\leq -D_x \epsilon\log \epsilon \) for sufficiently small \(\epsilon\).
    Furthermore, if \(K \Subset M ^\circ \) is compact then, the function \(\tilde{U}_\epsilon(x) \) converges uniformly to some smooth bounded function \(\tilde{U}(x) \) on \(K \) and there is a positive constant \(D_K \) such that \(\sup_{x\in K }|r_\epsilon(x)| \leq -D_K \epsilon\log \epsilon \).
\end{theorem}
\subsection{Outline of the paper}
In \cref{sec:bdry_greens} we provide some geometric background for our situation of a two-dimensional compact connected manifold with boundary, give some information on the Green's function \(G(x;y) \in \mathcal{D}'(M\times M )\) and provide useful lemmas with references for the rest of the paper. We complete  the proof of \cref{thm:hyperbolic} in \cref{sec:proof-main} using the information given in \cref{sec:bdry_greens}. This proof is extended in  \cref{sec:ggg}, where we first provide some additional information on gas giant geometries, and then prove \cref{thm:ggg} alluded to in \cref{sec:intro} using techniques from \cref{sec:bdry_greens,sec:proof-main}. In \cref{sec:blowups} we show how the limiting case of \cref{thm:ggg} gives the same results as \cref{thm:hyperbolic}. Finally,  we supplement the theorems in this paper with Monte Carlo simulations in \cref{sec:numerics}.
\section{Boundary coordinates systems and Green's functions}\label{sec:bdry_greens}
\subsection{Boundary coordinate systems}
We begin this discussion on the compact Riemannian manifold with boundary \((M,\partial M , \bar g )\).
For \(\delta>0 \) small we define the manifold \(M_\delta \) by 
\begin{equation}
    M_\delta := \{x\in M : \rho(x)>\delta\} \ .
\end{equation}
On \(M\setminus M_\delta\) we have the coordinate system given by \(x\mapsto (\rho(x),\theta(x))\), where \(\theta(x) \in \partial M \) is the closest point on the boundary to \(x\in M\setminus M_\delta\). Observe that if \(v\in T_{(0,\theta(x))}M\) with \(|v|_g = 1 \) and \(v\perp_{\bar g }T_{(0,\theta(x))}\partial M \), then \(x = \exp_{(0,\theta(x))}(\rho(x) v)\). 
The metric in this coordinate systems is thus given by 
\begin{equation}\label{eq: bar g in coordinates}
    \bar g = \dd \rho^2 + \tilde \beta(\rho,\theta)\dd \theta^2
\end{equation}
for some smooth function \(\tilde \beta\). Since \(\dim M = 2 \), all Riemannian metrics are conformally related \cite{lee2018riemann}*{Chapter 7}. In particular, \(\bar g  \) is conformally related to a metric \(\bar g_0 \), such that \((M\setminus M_\delta,\bar g_0 )\) is isometric to the annulus \( (\DD\setminus \DD_{1-\bar \delta},g_\Euc)\) with inner radius \(1-\bar \delta \) and outer radius 1 for some \(0<\bar \delta <1\) under some isometry \(\Phi: M\setminus M_\delta \to \mathbb D \setminus \DD_{1- \bar \delta }\),
with \(\Phi(\partial M) = \partial \DD  \). 

In particular, since \(\bar g \) and \(\bar g_0 \) are conformally related, there is a smooth function \(b(x) \in C^\infty(M)\)  and  a positive smooth function \(\beta(x) = e^{b(x)}\in C^\infty(M)\)  such that 
\begin{equation}
    \bar g = \beta\cdot \bar g_0 = e^{-b}\cdot \bar g_0 \ .  
\end{equation}

\begin{definition}\label{def:rho0} 
    We define the boundary defining function  \(\rho_0 = \Phi^*(\varrho) \in C^\infty(M\setminus M_\delta)\) to  be the pullback under the isometric diffeomorphism \(\Phi \) of the boundary defining function \(\varrho\in C^\infty(\DD\setminus \DD_{1-\bar \delta})\) given by 
    \begin{equation}
        \varrho:\DD\setminus \DD_{1-\bar \delta}\to \RR_+: (r,\vartheta)\mapsto \frac{1}{2}(1-r^2) \ ,
    \end{equation}
    where \((r,\vartheta)\) are polar coordinates on \(\DD \).
\end{definition}
\begin{lemma}
    The boundary defining function \(\varrho \) vanishes on the unit circle \(\partial \DD = \{(r,\vartheta):r = 1 \}\). Furthermore, the Riemannian manifold \((\DD\setminus \DD_{1-\bar \delta}, g_{\Euc}/\varrho^2 )\) has constant scalar curvature of \(-2 \).
\end{lemma}
\begin{proof}
    The first statement of the lemma follows directly from the definition of \(\varrho \).

    For the second statement, we compute the scalar curvature of conformally related surfaces using \cite{lee2018riemann}*{Theorem 7.30}. We have 
    \begin{equation}
        R_{g_{\Euc}/\varrho^2} = e^{2\log \varrho} (R_{g_{\Euc}} +2 \Delta_{g_{\Euc}}\log \varrho ) = \varrho^2\left( 0-\frac{2}{\varrho^2} \right) = -2
    \end{equation}
    since the scalar curvature \(R_{g_\Euc} = 0 \).
\end{proof}
\begin{lemma}
    The metric \(g_{\Euc }\) on \(\DD\setminus \DD_{1-\bar \delta }\) can be expressed as
    \begin{equation}
        \begin{split}
            g_{\Euc} &= \dd x^2 + \dd y^2 = \dd r^2 +r ^2 \dd \vartheta^2  \\
            &= \frac{1}{1-2\varrho}\dd \varrho^2 + (1-2\varrho)\dd \vartheta^2 = \left( \sum_{n=0}^\infty (2\varrho)^n \right) \dd \varrho^2 + (1-2\varrho)\dd \vartheta^2 \ . 
        \end{split}
    \end{equation}
\end{lemma}
\begin{proof}
    The first equality  gives the Euclidean metric in polar coordinates. The second equality then follows by the chain rule for differentiation. The final equation follows by the geometric series, since the image of \(\varrho\) is contained in the interval \([0,\bar \delta-\frac{1}{2}\bar \delta^2]\subset [0,1-\delta')\) for some \(\delta'>0 \) for sufficiently small \(\bar \delta\).
\end{proof}
\begin{remark}
    The manifold \((\DD\setminus \DD_{1-\bar \delta }, g_{\Euc}/\varrho^2)\) will be our model manifold throughout this paper. We can explicitly compute on this manifold, and using the isometric diffeomorphism \(\Phi \), we can therefore compute on \(M\setminus M_\delta \).
\end{remark}
\begin{lemma}\label{lem: bar g0 in local coordinates}
    Near the boundary \(\bar g_0 \) has the form 
    \begin{equation}
        \bar g_0 =(1+2\rho_0 + \mathcal{O}(\rho_0^2)) \dd \rho_0^2 +(1-2\rho_0) \dd \theta^2  \ ,
    \end{equation}
    where \(\rho_0 \) is the boundary defining function from \cref{def:rho0}.
\end{lemma}
\begin{proof}
    Since \(\Phi(\partial M) = \partial \DD\), the pullback \(\Phi^*(\dd \vartheta^2 )\) of part of  the metric \(g_{\Euc }\) tangent to the  boundary \(\partial \DD \) remains tangent to the pullback \(\Phi^*\partial \DD = \partial M \) of the boundary under the isometric diffeomorphism \(\Phi\).
    Hence, we have 
    \begin{equation}
        \begin{split}
            \bar g_0 = \Phi^* g_{\Euc}& =  \Phi^*\left( \frac{1}{1-2\varrho}\dd \varrho^2 + (1-2\varrho)\dd \vartheta^2 \right) \\
            &= (1+2\rho_0 + \mathcal{O}(\rho_0^2)) \dd \rho_0^2 + (1-2\rho_0)\dd \theta^2 \ .
        \end{split}
    \end{equation}
\end{proof}
\begin{lemma}\label{lem: constants beta}
    There is a positive function \(c(\theta) \in C^\infty(\partial M) \) such that near the boundary \(\partial M \),  the boundary defining functions are related by
    \begin{equation}
        \rho_0 = c(\theta)\rho + \mathcal{O}(\rho^2) \ .
        \label{eq:rho relation}
    \end{equation} 
    Furthermore, the factor \(c^2(\theta)\beta(0,\theta)  = 1 \) for each point \(\theta \in \partial M \).
\end{lemma}
\begin{proof}
    From the local description of the metrics \(\bar g \) in \cref{eq: bar g in coordinates}, and \(\bar g_0 \) in \cref{lem: bar g0 in local coordinates} near the boundary \(\partial M \), it follows that the vector fields 
    \[
        \partial\rho |_{\rho=\rho_0=0}:= \frac{\partial}{\partial \rho}\bigg|_{\rho=\rho_0=0}\quad \text{and}\quad \partial\rho_0 |_{\rho=\rho_0=0}:=\frac{\partial}{\partial \rho_0}\bigg|_{\rho=\rho_0=0}
    \]
    are both perpendicular to the boundary \(\partial M \). Hence, since at every point the smooth vector fields are co-linear, it follows that there is a non-zero smooth function \(c(\theta) \in C^\infty (\partial M) \) such that 
    \begin{equation*}
        \partial\rho_0 = c(\theta)\partial \rho
    \end{equation*}
    for every point on the boundary \(\partial M \). Positivity follows from the fact that both smooth vector fields are inward pointing. Restricted to the boundary \(\partial M \), from the local descriptions we have 
    \begin{equation*}
        1 = |\partial \rho|_{\bar g} = |\partial \rho_0|_{\bar g_0}.
    \end{equation*}
    Thus, we now compute restricted to the boundary \(\partial M \),
    \begin{equation*}
        1 = |\partial \rho_0|_{\bar g_0} = |c(\theta) \partial \rho|_{\bar g_0} = c(\theta) \beta(0,\theta)^{1/2}|\partial \rho|_{\bar g} = c(\theta) \beta(0,\theta)^{1/2} \ ,
    \end{equation*}
    from which \(c^2(\theta)\beta(0,\theta)  = 1 \) follows.
\end{proof}
Hence, the boundary value problem \eqref{eq:bvp} simplifies to 
\begin{equation}
    \Delta_{\bar g_0 } u_\epsilon = -\frac{c^2\beta(\rho_0,\theta)}{\rho_0^2}  = -\frac{1 + \beta^{(1)}(\theta)\rho_0 + \beta^{(2)}(\rho_0,\theta)}{\rho_0^2}\ .
    \label{eq:bvp_simplified}
\end{equation}
\subsection{Green's functions}
In this subsection we provide, for the  metric \(\bar{g}_0  \), the Green's functions, and give a decomposition into a principal and  remainder parts near the boundary \(\partial M \).

There is a unique Green's Dirichlet function 
\(G_0(x;y) \in \mathcal{D}'(M\times M ) \) with the properties \cite{Aubin1982}*{Theorem 4.17}
\begin{equation}
    \begin{cases}
        \Delta_{\bar{g }_0,y} G_0(x;y) = \delta_x(y) \quad &\text{for \(x,y\in M^\circ \)}\\
        G_0(x;y)|_{y\in \partial M}  = 0\\ 
        G_0(x;y) \in \RR \quad &\text{for \(x,y\in \overline M \)}\\
        G_0(x;y) = G_0(y;x)\ .
        \label{eq:greens0_conditions}
    \end{cases} 
\end{equation}
Since away from the diagonal \({\rm diag}(M\times M )\) the Green's function \(G_0(x;y)\) is smooth, we have the following lemma.
\begin{lemma}\label{lem:G-0-epsilon}
    Assume that \(\delta>0  \) is chosen sufficiently small. For each fixed \(x_0\in M^\circ\), there is a positive constant \(C_{x_0} \) such that 
    \begin{equation}
        |G_0(x_0;y)| \leq C_{x_0} \rho_0(y)
    \end{equation}
    for \(y\in M^\circ\) such that \(\rho_0(y) <  \min\{\delta,\frac{1}{2}\rho_0(x_0)\}\). 
    Moreover, for any \(K\Subset M^\circ  \) compact, there is a positive  constant \(C_K \) such that 
    \begin{equation}
        |G_0(x;y)|\leq C_K \rho_0(y)
    \end{equation}
    for \( x\in K , y \in M ^\circ \) with \(\rho_0(y) < \min\left\{ \delta,\frac{1}{2}\min_{x\in K }\rho_0(x) \right\}\).
\end{lemma}
\begin{proof}
    For each fixed \(x_0\in M^\circ \) and \(y\in \partial M \) we have 
    \begin{equation}
        G_0(x_0,y)|_{y\in \partial M} = 0 \ .
    \end{equation}
    Now, because \(G_0(x_0;y) \) is smooth for each fixed \(x_0\in M ^\circ \) and  \(y\in \partial M  \), by applying Taylor's theorem in its integral from to the function \(G_0(x_0;\cdot)\), 
    it follows that at \(y = (\rho_0(y),\theta(y) )\)
    \[
        G_0(x_0;y) = \rho_0(y) \int_0^1 \frac{\partial G_0(x_0;t\rho_0(y),\theta(y))}{\partial \rho_0(y)}\dd t  \ .
    \]
    Hence, there is a positive constant \(C_{x_0}\) such that 
    \[
        |G_0(x_0;y)| \leq C_{x_0} \rho_0(y) 
    \]
    for  \(y\in M^\circ\) such that \(\rho_0(y) <  \min\{\delta,\frac{1}{2}\rho_0(x_0)\}\). 
    Since the Green's function is smooth away from the diagonal, it follows that the function 
    \[
        C_x = \int_0^1 \frac{\partial G_0(x;t\rho_0(y),\theta(y))}{\partial \rho_0(y)} \dd t 
    \]
    is continuous for \(x \in K \Subset M ^\circ \). Hence,
    if \(K \Subset M^\circ  \) is compact, we can take the maximum \(C_K = \max_{x\in K } C_x \), then 
    \[
        |G_0(x;y)|\leq C_K \rho_0(y)
    \]
    for \( x\in K , y \in M ^\circ \) with \(\rho_0(y) < \min\left\{ \delta,\frac{1}{2}\min_{x\in K }\rho_0(x) \right\}\).
\end{proof}
We will now give a decomposition of the Green's function \(G_0(x;y) \) into a sum of a \emph{principal Green's function} \(G_{0,\prin}(x;y)\) and \emph{remainder Green's function} \(G_{0,\rem}(x;y)\).
\begin{definition}\label{def:Greens_principal}
    Assume without loss of generality that \((M\setminus M_\delta,\bar g_0) \) can be extended to some manifold \((\widetilde M\setminus M_\delta ,\bar g_0 )\) isometric to \((\RR^2\setminus \DD_{1-\bar \delta},g_{\Euc})\), and let \(\Phi:\widetilde M\setminus M_\delta \to \RR^2\setminus \DD_{1-\bar \delta}\) be this isometry. 

    For \(x,y\in M\setminus M_\delta\) denote 
    \begin{equation}
        |x|: = |\Phi(x)| , \  |x-y|:= |\Phi(x)-\Phi(y)|, \  \text{and}\quad   \frac{x}{|x|^2} := \Phi^{-1}\left(\frac{\Phi(x)}{|\Phi(x)|^2}\right)\in \widetilde M\setminus M_\delta \ ,
    \end{equation}
    then define the distribution \(G_{0,\prin}(x;y)\in \mathcal{D}'(M\times M) \) by 
    \begin{equation}
        G_{0,\prin}(x;y) = \frac{\kappa(x;y)}{2\pi}\left[ \log |x-y| - \log \left||x| \cdot \left( y-\frac{x }{|x|^2} \right)\right| \right] \ ,
    \end{equation}
    where \(\kappa(x;y)= \kappa_1(x)\cdot \kappa_1(y)\cdot \kappa_0(x;y)\) and \(\kappa_1\in C^\infty(M),\kappa_0(x;y)\in C^\infty(M\times M) \) are smooth cutoff functions on \(M \) with the properties that 
    \begin{equation}\label{eq:kappa-properties}
        \begin{cases}
            \kappa_1(x) \equiv 1 \quad &\text{if \(\rho_0(x)\leq \delta/4\)}\\
            \kappa_1(x) \equiv 0 &\text{if \(\rho_0(x) \geq \delta/3\)}
        \end{cases}
        \quad \text{and}\quad 
        \begin{cases}
            \kappa_0(x;y) \equiv 1 \quad &\text{if \(d_{\bar g_0}(x;y)<\delta/4\)}\\
            \kappa_0(x;y) \equiv 0 &\text{if \(d_{\bar g_0}(x;y)>\delta/3\)}\\
            \kappa_0(x;y)\equiv \kappa_0(y;x)&\text{for all \(x,y\in M \).}
        \end{cases} 
    \end{equation}
\end{definition} 
\begin{remark}\label{rem:Greens disc}
    The distribution 
    \begin{equation}
        G_{0,\prin,\A}(z_1;z_2) = \frac{1}{2\pi}\left[ \log |z_1-z_2| - \log \left| |z_1|\cdot \left( z_2-\frac{z_1}{|z_1|^2} \right)\right| \right]
    \end{equation}
    is the symmetric Green's function on the unit disc, with the properties of \eqref{eq:greens0_conditions} for the Euclidean metric \(g_{\Euc}\)~\cite{evans2010pde}*{Section 2.2 Equation (41)}. We therefore have 
    \begin{align}
        \begin{cases}
            \Delta_{g_\Euc,z_2} G_{0,\prin,\A} (z_1,z_2) &= \delta_{z_1}(z_2)\ ,\\
            G_{0,\prin,\A}(z_1,z_2) &= G_{0,\prin,\A}(z_2,z_1)\ ,\\
            G_{0,\prin,\A}(z_1,z_2)|_{|z_2| = 1} &= 0 \ .
        \end{cases}
    \end{align}
\end{remark}
\begin{lemma}\label{lem:G0prin-pullback}
    We can  write 
    \begin{equation}
        G_{0,\prin}(x;y) =\kappa(x;y) (\Phi^*G_{0,\prin,\A})(x;y) = \kappa(x;y) G_{0,\prin,\A}(\Phi(x);\Phi(y)) \ .
    \end{equation}
\end{lemma}
\begin{proof}
    In \cref{def:Greens_principal}, \(G_{0,\prin}\) is exactly defined using  the distribution \(G_{0,\prin,\A}\).
\end{proof}
\begin{lemma}
    The distribution \(G_{0,\prin}(x;y) \) is symmetric, \(G_{0,\prin}(y;x) = G_{0,\prin}(x;y) \). 
\end{lemma}
\begin{proof}
    The smooth cut-off function \(\kappa(x;y) \) is symmetric by definition. On the support \(\supp \kappa(x;y) \), we therefore have 
    \begin{equation}
        \begin{split}
            G_{0,\prin}(y;x) &=  {\kappa(y;x)} G_{0,\prin,\A}(\Phi(y);\Phi(x))  \\
            &= {\kappa(x;y)} G_{0,\prin,\A}(\Phi(x);\Phi(y)) = G_{0,\prin}(x;y) \ .
        \end{split}
    \end{equation}
\end{proof}
\begin{proposition}\label{prop:greens decomposition}
    We have the following decomposition of the Green's function \(G_0(x;y) \)
    \begin{equation}
        G_0(x;y) = G_{0,\prin }(x;y) + G_{0,\rem }(x;y) \ ,
    \end{equation}
    where \(G_{0,\rem }(x;y) \) satisfies 
    \begin{enumerate}
        \item \label{prop_part:greens decomp 1} \(G_{0,\rem }(x;y) \) is a smooth symmetric function for \(x,y\in M\setminus M_{\delta/4}\), with \(G_{0,\rem}(x;y)= 0 \) whenever \(y\in \partial M \), and 
        \item \label{prop_part:greens decomp 2} if \(0<\epsilon\ll \delta \), then  there there are constants \(C_0,C_1,C_2\) such that 
        \begin{equation*}
            \sup_{z\in \partial M_{\epsilon}}\left|\partial_{\nu_z} G_{0,\rem}(z;y)\right| \leq C_0 \cdot \rho_0(y) + C_1 (\epsilon + \rho(y))^2+ C_2 (\epsilon + \rho(y))^3 
        \end{equation*}
        holds for all \(y \in M\setminus M_{\delta/4}\) and \(\epsilon>0\).
    \end{enumerate}
\end{proposition}
\begin{proof}
    We have 
    \begin{equation}
        \begin{split}
            \Delta_{\bar g_0,y}G_{0,\prin}(x;y) &= \frac{1}{2\pi}\bigg[
                \kappa(x;y) \Delta_{\bar g_0,y}\left( \log |x-y| - \log \left||x| \cdot \left( y-\frac{x }{|x|^2} \right) \right|\right)+ F(x;y)
            \bigg]\ .
            \label{eq:G0prin-laplace}
        \end{split}
    \end{equation}
    Where
    \begin{equation}
        \begin{split}
            F(x;y) &:= \Delta_{\bar g_0,y}\kappa(x;y)\left[  \log |x-y| - \log \left||x| \cdot \left(y-\frac{x }{|x|^2}\right)\right|  \right] \\
            &\qquad + 2\inner{\nabla_y\kappa(x;y)}{\nabla_y\left( \log |x-y| - \log \left||x| \cdot \left(y-\frac{x }{|x|^2}\right)\right| \right)}_{\bar g_0},
        \end{split}
    \end{equation}
    and \(\nabla_y \) denotes the gradient in \(y \).
    By \cref{lem:G0prin-pullback}, it follows that 
    \begin{equation}
        \begin{split}
            \frac{\kappa(x;y)}{2\pi}&\Delta_{\bar g_0,y}\left(\log |x-y| -\log \left||x| \cdot \left( y-\frac{x }{|x|^2} \right) \right|\right) \\
            &= \kappa(x;y)\Delta_{\Phi^*g_{\Euc}}(\Phi^*G_{0,\prin,\A})(x;y) \\
            &=\kappa(x;y)\Phi^*( \Delta_{g_{\Euc}}G_{0,\prin,\A})(x;y) = \kappa(x;y)\delta_{x}(y) \ ,
        \end{split}
    \end{equation}
    since we assumed that \(\bar g_0 \) is isometric to \(g_{\Euc} \). So \eqref{eq:G0prin-laplace} becomes 
    \begin{eqnarray}\label{eq: G0prin-laplace simplified}  
        \Delta_{\bar g_0,y}G_{0,\prin}(x;y) = \kappa(x;y)\delta_x(y) + F(x;y).
    \end{eqnarray}
    By construction, we have that the supports of \(d_y \kappa(x;y) \) and \(\Delta_{\bar g_0} \kappa(x;y)\) are contained in the union
    \begin{equation}
        \begin{split}
            \supp d_{y}\kappa(x;y), &\supp \Delta_{\bar g_0,y} \kappa(x;y) \\
            &\subset \{(M_{\delta/4}\setminus M_{\delta/2})\times (M_{\delta/4}\setminus M_{\delta/2}) \} \cup \{\delta/4 <d_{\bar g_0}(x;y) <\delta/2\} \ .
        \end{split}
    \end{equation}
    So we have that \(F(x;y)\) is smooth for \(x,y\in M\setminus M_{\delta/4}\).
    If we now let 
    \begin{equation}
        G_{0,\rem} (x;y) := G_{0}(x;y)- G_{0,\prin}(x;y)
    \end{equation}
    then for \(x,y\in M\setminus M_{\delta/4 }\) we have that \eqref{eq: G0prin-laplace simplified} becomes 
    \begin{equation}
        \Delta_{\bar g_0,y} G_{0,\rem}(x;y) = \delta_x(y) - \kappa(x;y) \delta_x(y) - F(x;y) = -F(x;y)\ ,
    \end{equation}
    since \(\kappa(x;y) \equiv 1 \) for \(x,y\in M \setminus M_{\delta/4 }\).
    By elliptic regularity \cite{taylor2011pde2}, it therefore follows that \(G_{0,\rem}(x;y) \) is smooth on \(M\setminus M_{\delta/4}\). Furthermore, since both \(G_0(x;y) \) and \(G_{0,\prin}(x;y) \) are symmetric and zero on the boundary \(\partial M \), it follows that \(G_{0,\rem}(y;x)=G_{0,\rem}(x;y)\) and \(G(x;y)|_{y\in \partial M}=0\). This concludes \cref{prop_part:greens decomp 1}.

    Finally, since \(G_{0,\rem} (x;y)\) is jointly smooth near the boundary
    we can take a joint Taylor polynomial around \(\rho_0(x) = \rho_0(y) = 0 \), such that  there are smooth functions \(A_{ij }(\theta(x),\theta(y))\) such that 
    \begin{equation}
        \begin{split}
            G_{0,\rem}(x;y) &=  A_{00}(\theta(x),\theta(y))+ A_{10}(\theta(x),\theta(y))\rho_0(x)+ A_{01}(\theta(x),\theta(y))\rho_0(y) \\
            &\quad + A_{11}(\theta(x),\theta(y))\rho_0(x)\rho_0(y) + A_{20}(\theta(x),\theta(y) ) \rho_0(x)^2 \\
            &\quad + A_{02}(\theta(x),\theta(y) ) \rho_0(y)^2 +  R_3(x;y)
        \end{split}
    \end{equation}
    with the remainder \(R_3 (x;y) \) given by
    \begin{equation}\label{eq: remainder in taylor}
        \begin{split}
            R_3(x;y) &= \sum_{i+j = 3} \frac{3(\rho_0(x))^i (\rho_0(y))^j}{i!j!}\cdot \bigg[\\
            &\hspace{1cm}\int_0^1 (1-t)^{2}\frac{\partial^3 G_{0,\rem}(t\rho_0(x),\theta(x),t\rho_0(y),\theta(y))}{(\partial \rho_0(x))^i (\partial \rho_0(y))^j}\dd t   \bigg]
        \end{split}
    \end{equation}

    Since \(G_{0,\rem}(x;y) |_{y\in \partial M}\equiv 0 \), it follows that the functions \(A_{00}, A_{10} \) are identically 0. Furthermore, also because \(G_{0,\rem }(x;y)|_{y\in \partial M }\equiv 0 \) for all \(x \in M^\circ \), we find that the normal derivative \(\partial_{\nu_z}G_{0,\rem }(z;y)|_{\substack{y\in \partial M\\ z\in \partial M_\epsilon}}\equiv 0\), and that the functions \( A_{10}= A_{20}  \) are identically 0.

    Then, since the normal derivative \(\partial_\nu \) at the boundary \(\partial M_\epsilon\) in the local coordinates \((\rho_0,\theta) \) satisfies the relation \(\partial_\nu = -\partial_{\rho_0}|_{\rho_0 = \epsilon }\), we find that 
    \begin{equation}\label{eq:del G_0 rem with error term }
        \partial_{\nu_z}G_{0,\rem}(z;y)|_{z\in \partial M_\epsilon} = -A_{11}(\theta(z),\theta(y)) \rho_0(y) + \partial_{\nu_z}R_3(z;y)|_{z\in \partial M_\epsilon}\ .
    \end{equation}
    We can now bound the function \(A_{11}(\theta(x),\theta(y))\) uniformly on \(\partial M \times \partial M \) by a constant  \(C_0>0  \) such that  
    \[
        \sup_{\theta(x),\theta(y)\in \partial M }|A_{11}(\theta(x),\theta(y))|\leq C_0 \ . 
    \] 
    Furthermore,
    \begin{equation*}
        \begin{split}
            \bigg|\partial_{\nu_z} R_3(z;y)|_{z\in \partial M_\epsilon}\bigg| \leq 
            & \sum_{i+j= 3}\bigg|\frac{3 i (\epsilon)^{i-1} (\rho_0(y))^j}{i!j!}\int_0^1 (1-t)^{2}\frac{\partial^3 G_{0,\rem}(t\epsilon,\theta(z),t\rho_0(y),\theta(y))}{(\partial \rho_0(x))^i (\partial \rho_0(y))^j}\dd t     \bigg|\\
            &+\sum_{i+j=3}\bigg|\frac{3(\epsilon)^i (\rho_0(y))^j}{i!j!}\int_0^1 (1-t)^{2}\frac{\partial^4 G_{0,\rem}(t\epsilon,\theta(z),t\rho_0(y),\theta(y))}{(\partial \rho_0(x))^{i+1} (\partial \rho_0(y))^j}\dd t     \bigg|\\
            &\leq C_1 (\epsilon + \rho(y))^2 + C_2 (\epsilon + \rho(y))^3
        \end{split}
    \end{equation*}
    {by bounding the derivatives uniformly on the right-hand side over \(M\setminus M_{\delta/4}\times M\setminus M_{\delta/4}\) due to the joint smoothness of the function \(G_{0,\rem }(x;y)\).} 
    Then, it follows that  there are constants \(C_0,C_1,C_2\) such that 
    \begin{equation*}
        \sup_{z\in \partial M_{\epsilon}}\left|\partial_{\nu_z} G_{0,\rem}(z;y)\right| \leq C_0 \cdot \rho_0(y) + C_1 (\epsilon + \rho(y))^2+ C_2 (\epsilon + \rho(y))^3 
    \end{equation*}
    holds for all \(y \in M\setminus M_{\delta/4}\) and \(\epsilon>0\).
    This concludes \cref{prop_part:greens decomp 2}. 
\end{proof}
In \cref{sec:proof-main} we will use these decompositions to explicitly compute some integrals near the boundary.
\subsection{Sobolev spaces}
In this subsection we give definitions of the Sobolev spaces of interest in this paper. 
\begin{definition}\label{def:sobolev}
    For \(k\in \ZZ^+,(\Omega,\bar g)\subset ( \overline M ,\bar g)\) an open domain and \(\bar g \) a smooth metric on \(\overline \Omega\), we define the Sobolev spaces as the completions with respect to the \(H^k(\Omega, \bar g)\)-norm
    \begin{equation}
        \begin{split}
            H^{k}(\Omega,\bar g) &:= \overline{\{u\in C^\infty(\Omega)\}}^{\|\cdot\|_{H^k(\Omega, \bar g )}}\\
            H^{k,p}_0(\Omega,\bar g) &:= \overline{\{u\in C^\infty_0(\Omega)\}}^{\|\cdot\|_{H^k(\Omega, \bar g )}} \ ,
        \end{split}
    \end{equation}
    where the \(H^k(\Omega, \bar g )\) norm is given by 
    \begin{equation*}
        \|f\|_{H^k(\Omega, \bar g )} = \left( \sum_{i=0}^k \|D^i f \|_{L^2(\Omega, \bar g )}^2 \right)^{1/2}
    \end{equation*}
    where we interpret the \(D^i f \) as the \(i\)\textsuperscript{th} weak derivative of \(f\in C^\infty (\Omega)\). 
    
    We define the localised Sobolev space 
    \begin{equation}
        H^{k}_{\mathrm{loc}}(\Omega,\bar g):= \{u \in \mathcal{D}'(\overline{\Omega}): \varphi u \in H^{k}(\Omega, \bar g ) \quad \text{for all \(\varphi\in C^\infty_0(\Omega)\)}\} \ .
    \end{equation}
    For \(k\in \ZZ^+ \), the space \(H^{-k}(\Omega, \bar g )\) is defined as the formal dual to \(H^{k}_0(\Omega, \bar g  )\).

    For non-integer values \(l>0 \), we define the Sobolev spaces \(H^{l}(\Omega , \bar g )\) using the complex interpolation method \cite{Bergh1976interpolation}*{Theorem 6.4.5 part (7)}.
\end{definition}
This definition depends  implicitly on the metric (and therefore also on the \(L^2 \)-norms). 
By \cite{hebey1999sobolev}*{Proposition 2.2}, it follows that if \(\bar g,\bar h \) are two smooth metrics for \(\overline{\Omega} \), then the Sobolev spaces \(H^{l}(\Omega,\bar g) = H^{l}(\Omega,\bar h) \), since the metrics \(\bar g, \) and \(\bar h  \) provide equivalent norms on \(\Omega\). 

Next, the Poincaré inequality also provides an equivalent norm for \(H^{1}_0(\Omega, \bar g )\). Indeed, we see for \(\varphi\in H^{1}_0(\Omega, \bar g )\) that 
\begin{equation}
    \label{eq:poincare-inequality}
    \|D \varphi \|_{L^2(\Omega, \bar g )} \leq \left( \|\varphi\|_{L^2(\Omega, \bar g )}^2 + \|D \varphi\|_{L^2(\Omega, \bar g )}^2 \right)^{1/2} 
    \leq C(\Omega) \|D\varphi\|_{L^2(\Omega, \bar g )} \ .
\end{equation}
The second inequality is the Poincaré inequality. Since, by definition,   \(\|\varphi\|_{H^1(\Omega,\bar g)}:= (\|\varphi\|_{L^2(\Omega, \bar g )}^2 + \|D\varphi\|^2_{L^2(\Omega, \bar g )})^{1/2}\) the conclusion follows.  

The following lemmas show that (weak) solutions to elliptic boundary value problems have the ``expected'' regularity.
\begin{proposition}\label{prop:elliptic_regularity_p-type_sobolev}
    Let  \(\Omega \subset \overline M \) a  bounded domain with  smooth boundary and let  \(\bar g \) be a smooth metric on \(\overline{\Omega}\), then 
    if \(f \in H^{-1}(\Omega,\bar g)\), the boundary value problem
    \begin{equation}
        \begin{cases}
            \Delta_{\bar g } u = f \\
            u |_{\partial \Omega} = 0 
        \end{cases}
    \end{equation}
    has a unique solution \(u \in H^{1}_0(\Omega,\bar g)\) and there is a constant \(C> 0 \) such that 
    \begin{equation}
        \| u \|_{H^{1}(\Omega,\bar g)} \leq C \|f\|_{H^{-1}(\Omega,\bar g)}
    \end{equation}
\end{proposition}
This is a consequence of the boundary elliptic regularity \cite{mclean2000strongly}*{Theorem 4.10} together with the maximum principle for the Laplacian.

\begin{proposition}\label{prop:local-regularity}
    Let  \(\Omega \subset \overline M \) a bounded domain with  smooth boundary, then for   and \(s\in \RR \),  if \(f\in H^{s-2}_{\mathrm{loc }}(\Omega, \bar g )\), and if \(u \) satisfies \(\Delta_{\bar g} u = f\) for \(x\in \Omega \), then \(u \in H^{s}_{\mathrm{loc }}(\Omega, \bar g )\).
\end{proposition}
This is \cite{folland1995pde}*{Theorem 6.33} for smooth manifolds.

\begin{proposition}[Sobolev embedding theorem]\label{lem:sobolev embedding}
    Let \(\overline M \) be a compact smooth Riemannian manifold of dimension \(n \) with smooth boundary \(\partial M \), then 
    the embedding of \(H^{l}(M, \bar g ) \hookrightarrow C^s(\overline M )\) is continuous if \(l-n/2 > s\geq 0 \)  for integer \(s\).
\end{proposition}
For a proof see e.g.\ \cite{Aubin1982}*{Theorem 2.30}.

We finish this section with Hardy's inequality. For a discussion of the best constants see \cite{marcus1998hardy}.
\begin{proposition}\label{prop:hardy-inequality}
    Let \(\Omega \subset M \) be an open domain and  let  \(\bar g \) be a smooth metric on \(\overline{\Omega}\), then there exists a constant \(c \) such that for \(u \in H^{1}_0(\Omega, \bar g )\) the inequality
    \begin{equation}
        \int_{\Omega} \left( \frac{|u(x)|}{\mathrm{dist}_{\bar g} (x,\partial \Omega) } \right)^2 \vol_{\bar g }( x) \leq c \int_{\Omega}|D u |^2 \vol_{\bar g}( x) 
    \end{equation}
    holds.
\end{proposition}
\section{Proof of  \texorpdfstring{\cref{thm:hyperbolic}}{Theorem \ref{thm:hyperbolic}}}\label{sec:proof-main}
Now to continue, assume that \(\epsilon \) is sufficiently small. Fix some \(x_0\in M_\epsilon^\circ \) and let \(\delta >\epsilon>0 \) such that \(\rho_0(x_0) > \delta \).
Similarly to before, set \(M_\delta = \{x\in M : \rho_0(x) > \delta \}\). We further assume that the   estimates from \cref{lem:G-0-epsilon} hold for \(y\in M\setminus M_\delta\). This can be achieved by choosing \(\delta \) sufficiently small.
On the neighbourhood \(M_\epsilon\setminus M_\delta \) we can define the  function
\begin{equation}
    \label{eq:v-epsilon-def}
    v_\epsilon(x) = \log \rho_0(x) -\log \epsilon \ .
\end{equation}
On the model asymptotically hyperbolic surface \((\DD\setminus \DD_{1-\overline \delta},g_{\Euc}/\varrho^2 )\), the function \(v_\epsilon(x) \) solves the boundary value problem \eqref{eq:bvp} exactly.
By multiplying \(v_\epsilon \) with a smooth cut-off function \(\chi\in C^\infty(\overline{M}) \), depending only on \(\rho_0(x) \), which has the properties \(\chi|_{\rho_0(x)<\delta/2 }\equiv 1 \) and \(\supp \chi\subset \overline{M}\setminus M_\delta\), we see that the function \(\chi v_\epsilon\in C^\infty(M_\epsilon) \).
\begin{proposition}\label{lem:w_eps_Laplacian}
    For \(x\in M_\epsilon\), let \(w_\epsilon:M_\epsilon \to \RR \) with \(w_\epsilon (x)=u_\epsilon(x) - \chi(x)  v_\epsilon(x)\) where \(u_\epsilon\) is the solution to \eqref{eq:bvp}. Then 
    \begin{enumerate}
        \item \label{prop_part:w_eps_laplacian_1} There are smooth functions \(f_1,f_2\in C^\infty(M_\epsilon)\) such that 
    \begin{equation}
        \Delta_{\bar{g}_0}w_\epsilon(x) = f_1(x) \log  \epsilon + f_2(x) \ {\text{on }} M_\epsilon \ .
    \end{equation}
    \item \label{prop_part:w_eps_laplacian_2} The function \(f_1 \) extends to a smooth function on \(\overline{M } \) which is 0 on \(\overline{M}\setminus M_\epsilon\). The integral
    \begin{align}\label{eq: f1}
        I_1(x) = \int_{M_\epsilon} G_0(x;y) f_1(y) \vol_{\bar{g}_0}(y) &= \chi(x) -1.
    \end{align}
    \item \label{prop_part:w_eps_laplacian_3} For each fixed  \(x_0 = (\rho_0(x_0),\theta(x_0))\in M^\circ \), the function 
    \begin{eqnarray}\label{eq: f2}     
    f_2(\rho_0,\theta) = -\frac{(1-\chi) + \beta^{(1)}(\theta)\rho_0 + \beta^{(2)}(\rho_0,\theta)}{\rho_0^2} -\Delta_{\bar{g}_0}\chi\cdot \log \rho_0- 2 \inner{d\chi}{d v_\epsilon}_{\bar{g}_0}
    \end{eqnarray}
    satisfies
    \begin{eqnarray}\label{eq:I2}  
        I_2(x_0)=\int_{M_\epsilon}G_0(x_0;y) f_2(y) \vol_{\bar{g}_0}(y) &=  \mathcal{O}(1)\ ,
    \end{eqnarray}
    as \(\epsilon \to 0 \). 
    \item \label{prop_part:w_eps_laplacian_4} Finally, if \(K\Subset M^\circ\) is compact, the function \(\tilde u_\epsilon(x) := I_2(x)\) converges uniformly on \(K\) to some  bounded smooth function \(\tilde{u}(x)\in C^\infty(K)\) given by 
    \begin{equation}\label{eq: tilde u def}
        \tilde u(x) := \int_{\overline{M}}G_0(x;y) f_2(y) \vol_{\bar{g}_0}(y)  \ .
    \end{equation}
\end{enumerate}
\end{proposition}
\begin{proof}
    We simply calculate the Laplacian applied to \(w_\epsilon\) and use expression \eqref{eq:v-epsilon-def} for \(v_\epsilon\) to get 
    \begin{equation}\label{eq: Delta of w epsilon}
        \begin{split}
            \Delta_{\bar{g}_0} w_\epsilon(x) &= -\frac{1-\chi(x) + \beta^{(1)}(\theta)\rho_0(x) + \beta^{(2)}(\rho_0,\theta)}{\rho_0(x)^2 }\\
            &\qquad  -\Delta_{\bar{g}_0}(\chi(x)) v_\epsilon(x) -2\inner{d\chi(x)}{d v_\epsilon(x)}_{\bar{g}_0} \\
            &= f_1(x)\log\epsilon + f_2(x) \ ,
        \end{split}
    \end{equation}
    where \(f_1 = \Delta_{\bar{g}_0}\chi\) and 
    \[
        f_2 = -\frac{(1-\chi) + \beta^{(1)}(\theta)\rho_0 + \beta^{(2)}(\rho_0,\theta)}{\rho_0^2} -\Delta_{\bar{g}_0}\chi\cdot \log \rho_0- 2 \inner{d\chi}{d v_\epsilon}_{\bar{g}_0}\ .
    \] 
    Then by construction \(f_1,f_2 \in C^\infty(M_\epsilon) \) are smooth. This gives part 1.

    Now, we note that \(f_1 \) is identically 0 on a neighbourhood of the  boundary \(\partial M_\epsilon\), so that we can extend 
    \(f_1 \) to a smooth function on \(\overline{M }\) with \(f_1|_{\overline{M}\setminus M_\epsilon} \equiv 0 \).

    Given these expressions we can compute 
    \begin{equation}
        \begin{split}
            I_1(x) &= \int_{M_\epsilon}G_0(x;y)f_1(y) \vol_{\bar{g}_0}(x) =\int_{M_\epsilon}G_0(x;y)\Delta_{\bar{g}_0}\chi (y) \vol_{\bar{g}_0}(y) \\
            &=    \chi(x) + \int_{\partial M_\epsilon}G_0(x;y)\partial_{\nu} \chi(y) - \partial_{\nu_y} G_0(x;y) \chi(y) \vol_h(y) = \chi(x)-1 \ ,
        \end{split}
    \end{equation}
    due to fact that \(\int_{\partial M_\epsilon}\partial_\nu G_0(x;y) \vol_h(y) = 1 \) by integration by parts, 
    and hence \(I_1(x) \in C^\infty(\overline{M })\).
    This gives part 2.

    We now show that for each \(x_0\in M^\circ \) the integral \(I_2(x_0) \) is \(\mathcal{O}(1)\)  as \(\epsilon\to 0 \). 
    The functions 
    \[
        -\frac{1-\chi}{\rho_0^2}\ ,\quad \Delta_{\bar{g}_0}\chi \cdot \log \rho_0\ ,\quad \inner{d \chi}{d v_\epsilon}_{\bar{g}_0}
    \]
    are all identically 0 in a collar neighbourhood containing \(\partial M_\epsilon \) of the boundary \(\partial M \). 
    Similarly, the function \(\beta^{(2)}(\rho_0,\theta)/\rho_0^2\) is a bounded smooth function on this collar neighbourhood.  Hence, for each \(x_0\in M_\epsilon \) the function 
    \[
        G_0(x_0;y) \left[ -\frac{1-\chi+ \beta^{(2)}(\rho_0,\theta)}{\rho_0^2}-\Delta_{\bar{g}_0}\chi \cdot \log \rho_0-\inner{d \chi}{d v_\epsilon}_{\bar{g}_0} \right]
    \]
    is absolutely integrable over \(\overline{M }\).
    Finally, by  \cref{lem:G-0-epsilon} we notice that for each \(x_0\in M^\circ \) the function \(G_0(x_0;y)/\rho_0(y) \) is also absolutely integrable over \(\overline{M }\). Then, the integral 
    \begin{equation*}
        \int_{M_\epsilon}G_0(x;y) f_2(y) \vol_{\bar g_0}(y)  = \int_{\overline M}G_0(x;y) f_2(y)\mathbf{1}_{\rho_0>\epsilon} (y)\vol_{\bar g_0}(y)  
    \end{equation*}
    converges to 
    \begin{equation*}
        \int_{\overline M}G_0(x;y) f_2(y) \vol_{\bar g_0}(y) = \mathcal O  (1)\ ,
    \end{equation*}
    by the dominated convergence theorem.
    This gives part 3. 

    We can extend the smooth function \(f_2  \)  to a distribution in the  space \(H^{-1}(M,\bar g_0)\)  by setting 
    for each \(\varphi\in H^{1}_0(M,\bar g_0)\), the pairing 
    \begin{equation}
        \label{eq:f2 in H-1p}
        \begin{split}
            \inner{\varphi}{f_2 } = \int_{\overline M} \overline{\varphi(x) } \bigg(&-\frac{(1-\chi) + \beta^{(1)}(\theta)\rho_0 + \beta^{(2)}(\rho_0,\theta)}{\rho_0^2} \\
            &\qquad \quad-\Delta_{\bar{g}_0}\chi\cdot \log \rho_0- 2 \inner{d\chi}{d v_\epsilon}_{\bar{g}_0}\bigg) \vol_{\bar g_0}(x) .
        \end{split}
    \end{equation}
    Note that \(\varphi\mapsto \langle \varphi, f_2\rangle\) is a bounded linear function on \(H^1_0(M, \bar g_0)\) thanks to Hardy's inequality, \cref{prop:hardy-inequality}. Indeed, near the boundary the function \[f_2(x) =\frac{\beta^{(1)}(\theta)}{\rho_0(x)} + \mathcal{O}(1) = \frac{\beta^{(1)}(\theta)}{d_{\bar g}(\partial M , x)}+ \mathcal{O}(1)\ ,\]
    such that by the Cauchy-Schwarz and Hardy inequalities, there are constants \(\hat{c}, \hat{C}\)
    \[
        \begin{split}
            |\!\inner{\varphi}{f_2 }\!|&\leq \left\|\frac{\varphi}{d_{\bar g}(\partial M , x)}\right\|_{L^2({M}, \bar g )}\cdot \left\|{f_2}\cdot {d_{\bar g}(\partial M , x)}\right\|_{L^2({M}, \bar g )}\leq \hat{c}\|D\varphi\|_{L^2({M}, \overline{g})}\leq \hat{C}\|\varphi\|_{H^1(M,\bar{g})} \ .
        \end{split}
    \]
    
    Given this, we observe then, that the function \(\tilde u_\epsilon(x): M \to \RR \) defined by 
    \begin{equation}
        \tilde u_\epsilon(x) =I_2(x)= \int_{M_\epsilon}G_0(x;y) f_2(y) \vol_{\bar g_0}(y)  = \int_M G_0(x;y) \mathbf{1}_{\rho_0>\epsilon}(y)f_2\vol_{\bar g_0}(y)
    \end{equation}
    solves the partial differential equation 
    \begin{equation}
        \begin{cases}
            \Delta_{\bar g_0} \tilde u_\epsilon(x) =\mathbf{1}_{\rho_0>\epsilon}(x) f_2(x)  \quad &\text{for \(x\in M^\circ\) }\\
            \tilde u_\epsilon|_{\partial M} = 0  
        \end{cases}  \ .
    \end{equation}
    For any \(\epsilon>0 \), \cref{prop:elliptic_regularity_p-type_sobolev} implies that 
    the function  \(\tilde u_\epsilon(x) \in H^{1}_0(M,\bar g_0)\), since the right-hand side \(\mathbf{1}_{\rho_0>\epsilon}f_2 \in H^{-1}(M,\bar g_0)\).

    Notice further that \(\tilde u (x) \) given by \eqref{eq: tilde u def} solves the distributional partial differential equation 
    \begin{equation}\label{eq: tilde u solves}
        \begin{cases}
            \Delta_{\bar g_0} \tilde u (x) = f_2(x) \quad &\text{for \(x\in M^\circ\) }\\
            \tilde u|_{\partial M} = 0 
        \end{cases}
    \end{equation}
    We now show that \(\tilde u_\epsilon(x) \) converges to \(\tilde u(x) \) in \(H^{1}_0(M,\bar g_0)\).
    By \cref{prop:elliptic_regularity_p-type_sobolev} it follows that solutions to \cref{eq: tilde u solves} are unique. This shows that \(\tilde u (x)\in H^{1}_0(M, \bar g_0) \) is the unique  solution to \cref{eq: tilde u solves}.

    Now we have by the Cauchy-Schwarz inequality and by \cref{prop:hardy-inequality}, for each \(\epsilon>0 \) and 
    all  \(\varphi\in H^{1}_0(M,\bar g_0)\)  that there are constants \(c,d,d'>0 \) such that
    \begin{equation}
        \begin{split}
            \int_{M}&\left|(\mathbf{1}_{\rho_0>\epsilon}(x) - 1) f_2(x) \varphi(x) \right|\vol_{\bar g_0} x \\
            &\quad\leq c \left( \int_{M}\left|\frac{\varphi}{\rho_0}\right|^{2} \vol_{\bar g_0}(x)\right)^{1/2} \|\mathbf{1}_{\rho_0>\epsilon}-1\|_{L^{2}(M,\bar g_0 )}\\
            &\quad\leq d\left( \int_M |D\varphi(x)|^{2} \vol_{\bar g_0} x  \right)^{1/2}\cdot \|\mathbf{1}_{\rho_0>\epsilon}-1\|_{L^{2}(M,\bar g_0)} \\
            &\quad= d\|D\varphi\|_{L^{2}(M)} \cdot \|\mathbf{1}_{\rho_0>\epsilon}-1\|_{L^{2}(M,\bar g_0)}\\
            &\quad \leq d' \|\varphi\|_{H^{1}(M,\bar g_0)}\cdot \|\mathbf{1}_{\rho_0>\epsilon}-1\|_{L^{2}(M,\bar g_0)} \ .
        \end{split}
    \end{equation}
    Hence, we have that in the dual space topology on \(H^{-1}(M,\bar g_0)\), that 
    \begin{equation}
        \begin{split}
            \|\Delta_{\bar g_0} (\tilde u_\epsilon - \tilde u)\|_{H^{-1}(M,\bar g_0)} &= \sup_{\substack{\varphi\in H^{1}_0(M,\bar g_0)\\ 
            \|\varphi\|_{H^{1}(M,\bar g_0)}=1}}|\inner{\Delta_{\bar g_0} (\tilde u_\epsilon - \tilde u)}{\varphi}|\\
            &= \sup_{\substack{\varphi\in H^{1}_0(M,\bar g_0)\\ \|\varphi\|_{H^{1}(M,\bar g_0)}=1}}|\inner{(\mathbf{1}_{\rho_0>\epsilon}(x) - 1) f_2(x)}{\varphi}|\\
            &\leq d' \|\mathbf{1}_{\rho_0>\epsilon}-1\|_{L^{2}(M,\bar g_0)} \xrightarrow{\epsilon\to 0 } 0  \ . 
        \end{split}
    \end{equation}
    Thus, by \cref{prop:elliptic_regularity_p-type_sobolev}, it follows that, there is a constant \(C\) such that 
    \begin{equation}
        \|\tilde u_\epsilon - \tilde u\|_{H^{1}(M,\bar g_0)} \leq C \|\mathbf{1}_{\rho_0>\epsilon}-1\|_{L^2(M,\bar g_0)} \xrightarrow{\epsilon\to 0 }0 \ ,
    \end{equation}
    hence \(\tilde u_\epsilon(x) \) converges to \(\tilde u(x) \) in \(H^{1}_0\).

    Now for any compact set \(K \Subset M^\circ \), we can then take \(\epsilon>0 \) sufficiently small, such that the set \(K \subset \overline{\{\rho_0 > 2 \epsilon\}}\). 
    Because \(f_2 \in H^{k-2}_{\mathrm{loc }}(M,\bar g_0)\) for any \(k\in \ZZ^+ \), it follows that \(\tilde u_\epsilon, \tilde u \in H^{k}_{\mathrm{loc }}(M,\bar g_0)\) for any \(k\in \ZZ^+ \)  by \cref{prop:local-regularity}.
    Furthermore, for any \(\varphi\in C^\infty_0(M) \) such that \(\supp \varphi \subset K \) we have 
    \begin{equation}
        \varphi \Delta_{\bar g_0 } \tilde u_\epsilon  = \varphi f_2 = \varphi \Delta_{\bar g_0} \tilde u  \ ,\quad 
        \text{which implies that}\quad 
        \varphi \Delta_{\bar g_0} (\tilde u_\epsilon - \tilde u) = 0  \ .
    \end{equation}
    This means that \(\tilde u_\epsilon - \tilde u \in H^{k}(K,\bar g_0) \) for any \(k \in \ZZ^+\). Now by the complex interpolation method \cite{Bergh1976interpolation}*{Theorem 6.4.5 part (7)}, we find that for any \(l  = (1-\theta')k +\theta', \theta'\in (0,1 ) \) 
    \begin{equation}
        \begin{split}
            \|\tilde u_\epsilon - \tilde u\|_{H^{l}(K,\bar g_0)} &\leq C(k) \| \tilde u_\epsilon - \tilde u \|^{1-\theta'}_{H^{k}(K,\bar g_0)} \cdot \|\tilde u_\epsilon - \tilde u\|_{H^{1}(K,\bar g_0)}^{\theta'}\\
            &\leq C(k) (\|\tilde u_\epsilon\|^{1-\theta'}_{H^{k}(K,\bar g_0)} + \| \tilde u \|^{1-\theta'}_{H^{k}(K,\bar g_0)} ) \cdot\|\tilde u_\epsilon - \tilde u\|_{H^{1}(K,\bar g_0)}^{\theta'} \xrightarrow{\epsilon \to 0 }0  \ .
        \end{split}
    \end{equation}
    Hence, by  \cref{lem:sobolev embedding} it follows that for any integer \(s \) we can find \(l\) sufficiently large such that 
    \begin{equation}
        \|\tilde u_\epsilon - \tilde u\|_{C^s(K)} \leq C(s,k')\|\tilde u_\epsilon - \tilde u \|_{H^{l}(M,\bar g_0)} \xrightarrow{\epsilon\to 0 }0  \ .
    \end{equation}
    This shows that for any  compact subsets \(K \Subset M^\circ \), the function \(\tilde u \in C^\infty (K) \) and \(\tilde u_\epsilon \) converges uniformly to \(\tilde u \) on \(K \). This finishes the proof of the proposition.
\end{proof}
We now give one more estimate on the functions \(f_1 \) and \(f_2 \) in the following proposition. 
\begin{proposition}\label{lem:normal_integral}
    Let \(f_1,f_2 \in C^\infty({M_\epsilon })\) be the functions from  \cref{lem:w_eps_Laplacian}. Then, for \(z\in \partial M_\epsilon\), there are functions \(g_1, g_2:\partial M_\epsilon\to \RR  \) such that 
    \begin{equation}\label{eq: int f1 f2}
        \int_{M_\epsilon} \partial_{\nu_z}G_0(z;y)(f_1(y)\log \epsilon + f_2(y))\vol_{\bar{g}_0}(y) = g_1(z)\log\epsilon + g_2(z) \ ,
    \end{equation}
    with 
    \(\sup_{z\in \partial M_\epsilon}| g_i (z)|= \mathcal{O}(1)\) as \(\epsilon\to 0 \).
\end{proposition}
Define the  family of functions \(\tilde g_1, \tilde g_2: \partial M_\epsilon \to \RR \) by 
\begin{equation}\label{eq:del chi - 1}
    \tilde g_1(z) := \int_{M_\epsilon}\partial_{\nu_z}G_0(z;y)f_1(y)   \vol_{\bar{g}_0}(y)  
\end{equation}
and 
\begin{equation}\label{eq: tilde g2}
    \tilde g_2 (z):=\int_{M_\epsilon}\partial_{\nu_z}G_0(z;y)\left[ \frac{1-\chi}{\rho_0^2}+\Delta_{\bar{g}_0}\chi \cdot \log \rho_0+\inner{d\chi}{d v_\epsilon}_{\bar{g}_0} \right]\vol_{\bar{g}_0}(y) \ .
\end{equation}
\begin{lemma}\label{lem:normal_integral_no_singularities}
    For \(\epsilon \) small enough, \(\tilde g_1: \partial M_\epsilon \to \RR \) is constant \(0\).
        Furthermore, the function \(\tilde g_2(z)\) satisfies the estimate \(\sup_{z\in \partial M_\epsilon}|\tilde g_2(z)| = \mathcal O(1)  \) as \(\epsilon \to 0 \).
\end{lemma}
The proof to \cref{lem:normal_integral_no_singularities} is given in \cref{app:prop3.2}. We continue with the proof of \cref{lem:normal_integral}.
\begin{proof}[Proof of \cref{lem:normal_integral}]
    We recall that \(f_2 \) is given by 
    \[
        f_2 (y) = -\frac{(1-\chi) + \beta^{(1)}(\theta)\rho_0 + \beta^{(2)}(\rho_0,\theta)}{\rho_0^2} - \Delta_{\bar{g}_0}\chi \cdot \log \rho_0 - 2\inner{d\chi}{dv_{\epsilon}}_{\bar{g}_0}\ . 
    \]
    Hence, by the definition of \(\tilde{g}_2(z) \) and  \cref{lem:normal_integral_no_singularities} we may rewrite the left-hand side of \eqref{eq: int f1 f2}, for \(z\in \partial M_\epsilon\), to 
    \begin{equation}
        \label{eq: int f1 f2'}
        \begin{split}
            \int_{M_\epsilon} \partial_{\nu_z}&G_0(z;y)(f_1(y)\log \epsilon + f_2(y))\vol_{\bar{g}_0}(y) \\ 
            &=\tilde g_2(z) -  \int_{M_\epsilon}\partial_{\nu_z}G_0(z;y) \frac{\beta^{(1)}(\theta(y))\rho_0(y)+ \beta^{(2)}(\rho_0,\theta)}{\rho_0(y)^2}\vol_{\bar{g}_0}(y) 
        \end{split}
    \end{equation}
    with \(\tilde g_2(z)\) uniformly bounded  on \(\partial M_\epsilon \) as \(\epsilon \to 0 \) as shown in Appendix \ref{app:prop3.2}.

    To complete the proof, it remains to analyse the integral
    \begin{equation}
        \int_{M_\epsilon}\partial_{\nu_z}G_0(z;y) \frac{\beta^{(1)}(\theta(y))\rho_0(y)+ \beta^{(2)}(\rho_0,\theta)}{\rho_0(y)^2}\vol_{\bar{g}_0}(y) \ .
        \label{eq:normal_error}
    \end{equation}
    We shall show that the absolute integral is bounded uniformly by a function which is a multiple of \(\log \epsilon\)  for \(\epsilon>0 \). In \cref{app:prop3.2} we show that there are families of  functions  \(h_1,h_2: \partial M_\epsilon \to \RR \) such that 
    \[
        \int_{M_\epsilon}\partial_{\nu_z}G_0(z;y) \frac{\beta^{(1)}(\theta(y))\rho_0(y)+ \beta^{(2)}(\rho_0,\theta)}{\rho_0(y)^2}\vol_{\bar{g}_0}(y) = h_1(z) + h_2(z) 
    \]
    where \(\sup_{z\in \partial M_\epsilon}|h_1(z) | = \mathcal{O}(-\log \epsilon)\), and \(\sup_{z\in \partial M_\epsilon}|h_2(z)|= \mathcal O(1)\) as \(\epsilon \to 0 \), see \cref{prop:claim_3.2}.

    Hence, it follows that the functions
    \begin{equation}
        g_1(z) = \frac{h_1(z)}{\log \epsilon }\quad \text{and}\quad g_2(z) = \tilde g_2 (z) + h_2(z) 
    \end{equation}
    satisfy the estimates \(\sup_{z\in \partial M_\epsilon}|g_i(z)| =\mathcal{O}(1)\) as \(\epsilon\to 0 \). 
    Combining these facts, we conclude that 
    \[ 
        \int_{M_\epsilon} \partial_{\nu_z}G_0(z;y)(f_1(y)\log \epsilon + f_2(y))\vol_{\bar{g}_0}(y) = g_1(z)\log\epsilon + g_2(z) 
    \]
    with 
    \(\sup_{z\in \partial M_\epsilon}| g_i (z)|= \mathcal{O}(1)\) as \(\epsilon\to 0 \).
\end{proof}
\begin{proposition}\label{prop:normal_w_eps}
    Let \(w_\epsilon(x)= u_\epsilon(x) -\chi(x) v_\epsilon(x) \) from \cref{lem:w_eps_Laplacian}.
    The following estimate holds 
    \begin{equation}
        \|\partial_\nu w_\epsilon(z) \|_{L^2(\partial M_\epsilon)} = \mathcal{O}( -\log \epsilon) \ , \quad \text{as \(\epsilon\to 0 \).}
    \end{equation}
\end{proposition}
We prove this proposition using layer potential theory \cite{taylor2011pde2}*{Chapter 7, Section 11}.
\begin{definition}\label{def:N-sharp}
    Let \(N_\epsilon^\# \) be the double-layer potential 
    \begin{equation}\label{eq:N-sharp-def}
        (N^\#_\epsilon f)(z) = 2\int_{\partial M_\epsilon}\partial_{\nu_z} G_0(z;y)f(y)\vol_{h}(y) \ .
    \end{equation}
    Then \(N_\epsilon^\#\) is a pseudodifferential operator of order \(-1\) on \(\partial M_\epsilon\)~\cite{taylor2011pde2}*{Chapter 7, Proposition 11.3}.
    
    For \(f \in C^\infty(\partial M_\epsilon )\) we write
    \begin{equation}
        (N_\epsilon^\# f)(z) = (N_{\epsilon,\mathrm{prin}}^\# f)(z) + (N^\#_{\epsilon,\mathrm{rem}} f)(z)\ ,
    \end{equation}
    with 
    \begin{align}
        (N_{\epsilon,\mathrm{prin}}^\# f)(z) &= 2\int_{\partial M_\epsilon} \partial_{\nu_z} G_{0,\mathrm{prin}}(z;y) f(y)\vol_h(y)
        \ ,
        \intertext{and}
        (N_{\epsilon,\mathrm{rem}}^\# f)(z) &= 2\int_{\partial M_\epsilon} \partial_{\nu_z} G_{0,\mathrm{rem}}(z;y) f(y)\vol_h(y) \ ,
    \end{align}
    where \(G_{0,\mathrm{prin}}\) and \(G_{0,\mathrm{rem}}\) are the principal and remainder parts of the Green's function from \cref{def:Greens_principal}.
\end{definition}
\begin{definition}\label{def:fourier}
    For \(f\in L^2(\partial M_\epsilon)\) we define the Fourier transform \(\hat{f}\in L^2(\ZZ)\) by the isometry \(\Phi \) from \cref{def:rho0}. In particular:
    \begin{equation}
        \hat{f} = \mathcal{F} \circ (\Phi^{-1})^* f \ ,
    \end{equation}
    where \(\mathcal F \) is the Fourier transform on the circle of radius \(\sqrt{1-2\epsilon}\) given by 
    \begin{equation*}
        \hat{g}(n) = \frac{1}{2\pi \sqrt{1-2\epsilon}} \int_{-\pi}^\pi e^{-in \vartheta} g(\vartheta) \sqrt{1-2\epsilon}\dd \vartheta
    \end{equation*}
    for \(g(\vartheta)\in L^2(\partial\DD_{\sqrt{1-2\epsilon}})\).
\end{definition}
\begin{proposition}\label{prop:double-layer-N-sharp}
    \begin{enumerate}
        \item \label{prop_part:double-layer-N-sharp-prin} For \(f\in L^2(\partial M_\epsilon)\) we have 
        \begin{equation}
            (N_{\epsilon,\mathrm{prin}}^\# f)^\wedge (n) = (1-2\epsilon)^{|n|}\cdot \hat{f}(n) \ .
        \end{equation}        
        \item \label{prop_part:double-layer-N-sharp-rem}  The estimate 
        \[
            \|N_{\epsilon,\mathrm{rem}}^\#\|_{L^2(\partial M_\epsilon)\to L^2(\partial M_\epsilon)} = \mathcal{O}(\epsilon)
        \]
        on the operator  \(N_{\epsilon,\rem}^\#\) holds as \(\epsilon \to 0 \).
    \end{enumerate}
\end{proposition}
The proof to  \cref{prop:double-layer-N-sharp} is given in  \cref{app:double-layer-N-sharp-pf-general}. 
We continue with the proof of  \cref{prop:normal_w_eps}.
\begin{proof}[Proof of  \cref{prop:normal_w_eps}]
    We begin with Green's formula 
    \begin{equation}
        w_\epsilon(x) + \int_{\partial M_\epsilon}G_0(x;y) \partial_{\nu_y} w_\epsilon(y) \vol_h(y)  = \int_{M_\epsilon}{G_0(x;y)}(\Delta_{\bar{g}_0}w_\epsilon)(y) \vol_{\bar{g}_0}(y)
    \end{equation}
    take  \(x \in M_\epsilon \) to some point \(z \in \partial M_\epsilon \) and apply the outward normal derivative \(\partial_{\nu }\) at that \(z\in \partial M_\epsilon\).
    Using \cite{taylor2011pde2}*{Chapter 7, Proposition 11.3} 
    we find with \cref{eq: Delta of w epsilon} that  
    \begin{equation}
        \begin{split}
            \partial_{\nu} w_\epsilon(z) + \frac{1}{2}( -\partial_{\nu}w_\epsilon(z) &+ N_\epsilon^\# (\partial_{\nu}w)(z)) \\
            &= \partial_{\nu}\int_{M_\epsilon}{G_0(z;y)}(f_1(y)\log \epsilon + f_2(y))\vol_{\bar{g }_0}(y) \ .
        \end{split}
    \end{equation}
    On the right-hand side we pull the normal derivative into the integral, as the integral is absolutely convergent, and then apply  \cref{lem:normal_integral}. The conclusion is that for \(\epsilon \) sufficiently small
    \begin{equation}
    \label{eq: integral eq with layer potential}
        \frac{1}{2}\left(\partial_{\nu} w_\epsilon(z) +  N_\epsilon^\# (\partial_\nu w_\epsilon)(z)\right) = g_1(z) \log \epsilon + g_2(z)\ .
    \end{equation}
    
    Now, since we know by \cref{prop_part:double-layer-N-sharp-prin} of  \cref{prop:double-layer-N-sharp} that \(({N}_{\epsilon,\mathrm{prin}}^\# f)^\wedge (n) = (1-2\epsilon)^{|n|}\hat{f}(n)\), it follows that the inverse of the operator \(A_\epsilon := \frac{1}{2}(I + N_{\epsilon,\mathrm{prin}}^\#)\) satisfies the bound 
    \begin{equation}
        \begin{split}
            \|A_\epsilon^{-1} f \|_{L^2(\partial M_\epsilon)}^2&= \left\|\sum_{k\in \ZZ}\frac{2e^{in\theta}}{1+\hat{N}_{\epsilon,\mathrm{prin}}^\#(n)} \hat{f}(n)\right\|_{L^2(\partial M_\epsilon)}^2\\
            &\leq 
            |\partial M_\epsilon|_h \sum_{k\in \ZZ}\frac{4|\hat{f}(n)|^2}{(1+(1-2\epsilon)^{|n|})^2}\leq 4|\partial M_\epsilon|_h \sum_{k\in \ZZ}|\hat{f}(n)|^2\\
            &\quad  = 4 |\partial M_\epsilon|_h \|\hat{f}\|_{L^2(\ZZ)}^2 = 4 \|f\|_{L^2(\partial M_\epsilon)}^2 \ , 
        \end{split}
    \end{equation}
    by Parseval's identity: \(\|f \|_{L^2(\partial M_\epsilon)}^2 =|\partial M_\epsilon|_h \|\hat{f }\|_{L^2(\ZZ)}^2  \).
    Thus, by the Neumann series for operators we have 
    \begin{equation}
        \left(\frac{1}{2}(I + N_\epsilon^\# )\right)^{-1} = \left(A_\epsilon + \frac{1}{2} N_{\epsilon,\rem}^\#\right)^{-1} = A_\epsilon^{-1}\sum_{k=0}^\infty \left(-\frac{1}{2}N_{\epsilon,\rem}^\# A_\epsilon^{-1}\right)^k \ ,
    \end{equation}
    with 
    \begin{equation}
        \begin{split}
            \left\| \left(\frac{1}{2}(I + N_\epsilon^\# )\right)^{-1} \right\|_{L^2\to L^2} &\leq \|A_\epsilon^{-1}\|_{L^2\to L^2} \cdot \left\|\sum_{k=0}^\infty \left(-\frac{1}{2}N_{\epsilon,\rem}^\# A_\epsilon^{-1}\right)^k \right\|_{L^2\to L^2}\\
            &\leq 2 +   \mathcal{O}(\epsilon)  \leq 3
        \end{split}\ ,
    \end{equation}
    for sufficiently small \(\epsilon \), 
    because \(\|N_{\epsilon,\rem}^\#\|_{L^2\to L^2} = \mathcal{O}(\epsilon )\) as \(\epsilon\to 0\) by \cref{prop_part:double-layer-N-sharp-rem} of \cref{prop:double-layer-N-sharp}. 

    To find asymptotics of \(\partial_\nu w_\epsilon \), we apply the inverse \(\left( \frac{1}{2} I + N_\epsilon^\#  \right)^{-1} \) to both sides of  \eqref{eq: integral eq with layer potential}. We find that 
    \begin{equation}
        \begin{split}
            \|\partial_\nu w_\epsilon(z)\|_{L^2(\partial M_\epsilon)} 
            &=  \left\|\left(\frac{1}{2}(I+N_\epsilon^\#)\right)^{-1}(g_1(z)\log\epsilon + g_2(z))\right\|_{L^2(\partial M_\epsilon)}\\
            &\leq \left\| \left(\frac{1}{2}(I + N_\epsilon^\# )\right)^{-1} \right\|_{L^2\to L^2}\cdot \left\|g_1(z)\log\epsilon + g_2(z)\right\|_{L^2(\partial M_\epsilon)}\\
            &\leq 3 \| g_1(z) \log \epsilon + g_2 (z) \|_{L^2(\partial M_\epsilon )} 
            = \mathcal{O}(-\log \epsilon ) \quad \text{as \(\epsilon \to 0 \),}
        \end{split}
    \end{equation}
    because \(\|g_1\log \epsilon + g_2\|_{L^2(\partial M_\epsilon)}\leq |\partial M_\epsilon|_h\|g_1\log \epsilon + g_2\|_{L^\infty(\partial M_\epsilon)} = \mathcal O(-\log \epsilon)  \) as \(\epsilon\to 0 \) by \cref{lem:normal_integral}.
\end{proof}
We now only need one more proposition to prove  \cref{thm:hyperbolic}.
\begin{proposition}\label{prop:int-G0-delta-w}
    Let \(\tilde u_\epsilon(x) \) be the function from \cref{prop_part:w_eps_laplacian_4} of \cref{lem:w_eps_Laplacian}. For \(x\in  M^\circ  \), there is a function \(r_\epsilon(x) \) such that  
    \begin{equation}
        w_\epsilon(x) =  (\chi(x)-1) \log \epsilon + \tilde u_\epsilon(x) + r_\epsilon(x)  \ ,
    \end{equation}
    with  a positive constant \(D_x \) such that \(|r_\epsilon(x)|\leq -D_x \epsilon \log \epsilon\) for sufficiently small \(\epsilon\).
    If \(K \Subset M^\circ  \) is compact then there is  positive constant \(D_K \) such that \(\sup_{x\in K }|r_\epsilon(x)| \leq -D_K \epsilon\log \epsilon\).
\end{proposition}
\begin{proof}
    By Green's formula we have
    \begin{equation}
        \begin{split}
            \label{eq:w_eps-greens-formula}
            w_\epsilon(x) 
            = \int_{M_\epsilon}G_0(x;y)\Delta_{\bar g_0 ,y}w_\epsilon(y)\vol_{\bar g_0}(y)-\int_{\partial M_\epsilon}G_0(x;y) \partial_\nu w_\epsilon(y) \vol_h(y) \ .
        \end{split}
    \end{equation}
    By \cref{prop_part:w_eps_laplacian_1} of \cref{lem:w_eps_Laplacian} we have that 
    \begin{equation}\label{eq: w-eps-laplacian}
        \int_{M_\epsilon} G_0(x;y) \Delta_{\bar g_0, y }w_\epsilon(y) \vol_{\bar g_0} (y)  = \int_{M_\epsilon} G_0(x;y) (f_1(y) \log\epsilon + f_2(y)) \vol_{\bar g_0}(y)  \ .
    \end{equation}
    By \cref{prop_part:w_eps_laplacian_2} of \cref{lem:w_eps_Laplacian} we have 
    \begin{equation}\label{eq: f1-chi-1}
         \int_{M_\epsilon} G_0(x;y) f_1(y)\log \epsilon \vol_{\bar{g}_0}(y) = \left( \chi(x) -1     \right)\log \epsilon \ .
    \end{equation}
    By \cref{prop_part:w_eps_laplacian_3} of \cref{lem:w_eps_Laplacian} we have that for each \(x\in M^\circ \)  the function 
    \begin{equation}\label{eq: f2 asymptotics}
        \tilde u_\epsilon(x) = \int_{M_\epsilon} G_0(x;y) f_2(y) \vol_{\bar g_0 }(y) = \mathcal{O}(1) \quad \text{as \(\epsilon\to 0 \).}
    \end{equation}
    If \(K\Subset M^\circ \), then \(\tilde u _\epsilon(x) \) converges uniformly on \(K \) to \(\tilde u(x) \in C^\infty(K) \) by \cref{prop_part:w_eps_laplacian_4} of \cref{lem:w_eps_Laplacian}.

    Setting \(r_\epsilon(x) \) to be 
    \begin{equation}\label{eq:r-eps-def}
        r_\epsilon(x):= - \int_{\partial M_\epsilon} G_0(x;y) \partial_\nu w_\epsilon(y) \vol_h(y) \ ,
    \end{equation}
    we find by  the Cauchy-Schwarz inequality  that
    \begin{equation}\label{eq:r-eps-cs}
        \begin{split}
            |r_\epsilon(x) | = \left|\int_{\partial M_\epsilon}G_0(x;y)\partial_\nu w_\epsilon(y) \vol_h(y) \right|&\leq \|G_0(x;\cdot)\|_{L^2(\partial M_\epsilon)} \cdot \| \partial w_\epsilon\|_{L^2(\partial M_\epsilon)} \ .
        \end{split}
    \end{equation}
    Because 
    \(\|\partial_\nu w_\epsilon(z)\|_{L^2(\partial M_\epsilon)} = \mathcal{O}(-\log \epsilon)\) as \(\epsilon\to 0 \) and because  \(\sup_{y\in \partial M_\epsilon}|G_0(x;y)| \leq C_x \cdot \epsilon\) as \(\epsilon\to 0 \) by  \cref{lem:G-0-epsilon}, it follows that there is a positive constant \(D_x \) such that 
    \begin{equation}\label{eq:r-eps-asymptotics}
        |r_\epsilon(x) | \leq -D_x \epsilon\log \epsilon \quad \text{for sufficiently small \(\epsilon\).}
    \end{equation}
    If \(K \Subset M^\circ \) is compact, then by the same argument as in the proof of \cref{lem:G-0-epsilon}, it follows that \(D_x \) is continuous for \(x \in K \Subset M^\circ \) and \(y\in \partial M_\epsilon \) if \(\epsilon< \frac{1}{2}\min_{x\in K}\rho_0(x)\). Thus taking    \( D_K = \max_{x\in K }D_x \) we conclude that 
    \(\sup_{x\in K } |r_\epsilon(x) | \leq -D_K \epsilon \log \epsilon \).

    Plugging \cref{eq: w-eps-laplacian,eq: f1-chi-1,eq: f2 asymptotics} into the first integral on the right-hand side of \cref{eq:w_eps-greens-formula}, and \cref{,eq:r-eps-def} into the second integral in \cref{eq:w_eps-greens-formula}, we have 
    \begin{equation}
        \begin{split}
            w_\epsilon(x) &=  \int_{M_\epsilon}G_0(x;y) (f_1(y) \log \epsilon+f_2(y))\vol_{\bar{g}_0}(y) \\
            &\qquad \qquad -\int_{\partial M_\epsilon} G_0(x;y) \partial_\nu w_\epsilon(y) \vol_h(y)\\
            & = (\chi(x)-1) \log \epsilon + \tilde{u}_\epsilon(x) +r_\epsilon(x)   \ ,
        \end{split}
    \end{equation}
    with \(|r_\epsilon(x) |\leq -D_x \epsilon\log \epsilon \) by \cref{eq:r-eps-asymptotics}. Furthermore, for \(K \Subset M ^\circ \), there is a constant \(D_K \) such that \(|r_\epsilon(x)|\leq -D_K \epsilon\log \epsilon\), and we have that  \(\tilde{u}_\epsilon\) converges  uniformly to some smooth bounded function \(\tilde{u}(x) \) on \(K \).
\end{proof}
\begin{proof}[Proof of \cref{thm:hyperbolic}]
    We have by the definition of \(w_\epsilon(x) \) in \cref{lem:w_eps_Laplacian} that 
    \begin{equation}
        u_\epsilon(x)  = \chi(x) v_\epsilon(x) + w_\epsilon(x)  \ .
    \end{equation}
    Thus, by \cref{eq:v-epsilon-def,prop:int-G0-delta-w} we find for each \(x\in M^\circ\) that 
    \begin{equation}
        u_\epsilon(x) = -\log \epsilon +\chi(x) \log \rho_0(x)+ \tilde{u}_\epsilon(x)+ r_\epsilon(x)
    \end{equation}
    with \(|r_\epsilon(x)|\leq -D_x \epsilon \log \epsilon   \) and \(\tilde u_\epsilon(x) = \mathcal O (1) \)    as \(\epsilon\to 0 \) by \cref{prop:int-G0-delta-w}.

    Finally, 
    if \(K \Subset M^\circ \) is  compact, then there is a constant \(D_K \) such that \(\sup_{x\in K } |r_\epsilon(x)| \leq -D_K \epsilon\log \epsilon \) and 
    the function
    \begin{equation}
        \tilde U_\epsilon : K \to \RR: x \mapsto \chi(x) \log \rho_0(x) + \tilde u_\epsilon (x)
    \end{equation} 
    converges uniformly 
    to the smooth and  bounded function
    \begin{equation}
        \tilde{U}: K \to \RR: x\mapsto \chi(x) \log \rho_0(x) + \tilde u(x) 
    \end{equation}
    by applying \cref{prop_part:w_eps_laplacian_4} of \cref{lem:w_eps_Laplacian} to \(\tilde u_\epsilon \) and noting that \(\chi(x)\log \rho_0(x) \) is bounded and smooth on \(K \).
\end{proof}
\section{Gas giant geometries}\label{sec:ggg}
In the following sections we turn to gas giant geometries introduced in \cite{dehoop2024}. A Riemannian manifold \((M,\partial M , g )\) has a gas giant metric \(g \) of order \(\alpha\in (0,2) \) if it can be written in the form 
\begin{equation}
    g = \frac{\bar{g }}{\rho^\alpha}
\end{equation}
for some boundary defining function \(\rho \) and some smooth non-degenerate metric \(\bar{g }\). For \(\epsilon>0 \) we once again set \(M_\epsilon \) to be the Riemannian manifold with boundary \(\{x\in M : \rho(x) > \epsilon \}\). We once again wish to study the behaviour of the mean first escape time \(\EE(\tau_\epsilon^x)\) of the Brownian motion \((X_t^x,\PP^x_t)\) on \((M,g)\) starting at \(x\in M \) as \(\epsilon\to 0 \). Once again, by \cite{nursultanov2021mean}*{Appendix A}, \(u_\epsilon(x)= \EE(\tau_\epsilon^x )\) satisfies the boundary value problem \eqref{eq:bvp} 
\begin{equation*}
    \Delta_g u_\epsilon(x)|_{x\in M_\epsilon^\circ }= -1 , \quad u_\epsilon(x)|_{x\in \partial M_\epsilon} = 0 \quad \text{for \(u\in H^2(M_\epsilon)\cap H^1_0(M_\epsilon)\)} \ .
\end{equation*}
The following theorem shows that the behaviour of Brownian motion on gas giant surfaces differs significantly from Brownian motion on asymptotically hyperbolic surfaces. We show that the mean first escape time on a gas giant surface \(u_{\epsilon}(x) = \mathcal O(1) \)  as \(\epsilon\to 0 \).
\begin{theorem}
    \label{thm:ggg}
    Assume that \((M,g) \) is a gas giant surface of order \(\alpha\), and that  \(u_\epsilon(x) \) satisfies the boundary value problem \eqref{eq:bvp}. Then there are functions \(\tilde U_{\epsilon,\alpha }(x),r_{\epsilon,\alpha}(x) \) such that 
    \begin{equation}
        u_\epsilon(x) =\tilde U_{\epsilon,\alpha}(x) + r_{\epsilon,\alpha}(x) \ ,
    \end{equation}
    where for each fixed \(x\in M^\circ \) the function \(\tilde U_{\epsilon,\alpha}(x) = \mathcal O(1)\) as \(\epsilon\to 0 \) and there is a constant \(D_{x}\) such that \(|r_{\epsilon,\alpha}(x)|\leq D_x \epsilon \) for sufficiently small \(\epsilon\). Finally, if \(K \Subset M^\circ \) is compact, then there is  a smooth function \(\tilde U_\alpha \in C^\infty(K)\) and a constant \(D_K \)  such that \(\tilde U_{\epsilon,\alpha}(x)\) converges uniformly on \(K \) to \(\tilde U_\alpha \)  and \(\sup_{x\in K }|r_\epsilon(x)| \leq D_K \epsilon\).
\end{theorem}
The proof of \cref{thm:ggg} follows the same steps as the proof of \cref{thm:hyperbolic}.  Notice first that by \cref{lem: constants beta},  we can once again reduce the boundary value problem \eqref{eq:bvp} to 
\begin{equation}
    \Delta_{\bar{g}_0}u_\epsilon(x) = -\frac{1+ \beta^{(1)}(\theta)\rho_0(x) + \beta^{(2)}(\rho_0,\theta)}{\rho_0^\alpha(x)} \ .
\end{equation}
We define the functions \(v_{\epsilon,\alpha }\) by 
\begin{equation}
    v_{\epsilon,\alpha}(x):= \begin{cases}
        \begin{aligned}
            \frac{1}{(\alpha -1) (2-\alpha )}\bigg[\rho_0 ^{2-\alpha } \, _2F_1(1,2-\alpha ;3-\alpha ;2 \rho_0 )\qquad \qquad \qquad \\
            -{\epsilon }^{2-\alpha } \, _2F_1(1,2-\alpha ;3-\alpha ;2 \epsilon )\bigg]\qquad\quad  \\
            +\frac{2^{\alpha -2} (\log (1-2 \rho_0 )-\log (1-2 \epsilon ))}{\alpha -1}
        \end{aligned}&\text{if \(\alpha\neq 1 \)}\\
        \frac{1}{2} \left(\text{Li}_2(1-2 \epsilon )-\text{Li}_2\left(1-2\rho_0\right)\right)\quad &\text{if \(\alpha = 1 \)}
    \end{cases}
    \label{eq:v-epsilon-alpha-def}
\end{equation}
where \(\, _2F_1\) is the Hypergeometric function \cite{Abramowitz1964Handbook}*{Chapter 15},  
and \(\text{Li}_2\) is the Dilogarithm function \cite{Abramowitz1964Handbook}*{Section 27.7}. 
On the model gas giant surface \((\DD\setminus \DD_{1- \bar\delta},g_{\Euc}/\varrho^\alpha )\), the functions \(v_{\epsilon,\alpha}(x) \) solve the boundary value problem \eqref{eq:bvp} exactly.

Using the power series expansion of the Hypergeometric function 
\[
    \, _2F_1(a,b,c;z) = \sum_{n=0}^\infty \frac{(a)_n (b)_n }{(c)_n } \frac{z^n }{n!} \quad \text{where} \quad (d)_n = d \cdot (d+1) \cdots (d+n-1)
\]
around \(0 \) \cite{Abramowitz1964Handbook}*{15.1.1}, and the Dilogarithm
\[
    \text{Li}_2(z) = \sum_{k=1}^\infty (-1)^k \frac{(z-1)^k }{k^2}
\]
around 1~\cite{Abramowitz1964Handbook}*{27.7.2}, we find that the  functions \(v_{\epsilon,\alpha}\) satisfy the expansions around \(\epsilon,\rho_0 = 0 \)
\begin{align}
    \label{eq:v-alpha-epsilon-asymptotics}
    v_{\epsilon,\alpha}(x)=
        \frac{1}{\alpha-1}\bigg(\frac{1}{2-\alpha}\left( \rho_0^{2-\alpha}-\epsilon^{2-\alpha} \right)+{2^{\alpha -1} \epsilon }-{2^{\alpha -1} \rho_0}\bigg) + \mathcal{O}(\rho_0^2) + \mathcal{O}(\epsilon^2)
        \intertext{if \(\alpha\neq 1\), and}
        v_{\epsilon,1} =
        \rho_0 (-\log \rho_0+1-\log 2)+\epsilon  (\log \epsilon -1+\log 2) + \mathcal{O}(\rho_0^2\log \rho_0 )+ \mathcal{O}(\epsilon^2 \log \epsilon)  \ .
\end{align}
Now if \(\chi(x) \in C^\infty(\overline M )\)  is a  cut-off function depending only on \(\rho_0(x) \), with the properties that \(\chi|_{\rho_0(x) < \delta/2 }\equiv 1 \) and \(\supp \chi \subset \overline M \setminus M_\delta\), 
then \(\chi v_{\epsilon,\alpha} \in C^\infty(M_\epsilon)\) and  we have the following  lemma, similar to \cref{lem:w_eps_Laplacian}.
\begin{lemma}\label{lem:w-eps-alpha-laplacian}
    Let \(w_{\epsilon,\alpha}(x) = u_{\epsilon,\alpha}(x)-\chi(x)\cdot v_{\epsilon,\alpha}(x)\). Then there are smooth functions \(f_{1,\alpha},f_{2,\alpha}\in C^\infty(M_\epsilon) \) such that 
    \begin{equation}\label{eq:w-eps-alpha-laplacian}
        \Delta_{\bar{g}_0}w_{\epsilon,\alpha}(x) = \begin{cases}
            \frac{1}{(\alpha-1)(2-\alpha)}f_{1,\alpha}(x) \epsilon^{2-\alpha} +f_{2,\alpha}(x) \quad &\text{if \(\alpha> 1 \)}\\
            -f_{1,1}(x)\epsilon(\log \epsilon-1+\log 2) +f_{2,1}(x) \quad &\text{if \(\alpha = 1 \)}\\
            \frac{-2^{\alpha-1}}{\alpha-1}f_{1,\alpha} (x)\epsilon + f_{2,\alpha} (x)&\text{if \(\alpha< 1 \)}
        \end{cases}
    \end{equation}
    Furthermore, \(f_{1,\alpha} \)  extends to a smooth function which is 0 on \(\overline{M}\setminus M_\epsilon\). The integral
    \begin{align}
        \label{eq:w-eps-alpha-I1}
        I_{1,\alpha}(x)=\int_{M_\epsilon} G_0(x;y) f_{1,\alpha}(y) \vol_{\bar{g}_0}(y) &= \chi(x) -1\ ,
        \intertext{and for each \(x\in M^\circ\), the integral satisfies the estimate}
        \label{eq:w-eps-alpha-I2}
        I_{2,\alpha}(x)=\int_{M_\epsilon}G_0(x;y) f_{2,\alpha}(y) \vol_{\bar{g}_0}(y) &=  \mathcal{O}(1)\ ,
    \end{align}
    as \(\epsilon\to 0 \).

    Finally, if \(K\Subset M^\circ \) is a compact set, then the function \(\tilde u_{\epsilon,\alpha}:= I_{2,\alpha}(x) \) converges uniformly on \(K \) to some bounded smooth function \(\tilde u_\alpha(x) \in C^\infty(K)\) given by 
    \begin{equation*}
        \tilde u_{\alpha}(x) : = \int_{\overline M }G_0(x;y) f_{2,\alpha}(y) \vol_{\bar g_0}(y) \ .
    \end{equation*} 
\end{lemma}
\begin{proof}
    The proof is identical to the proof of \cref{lem:w_eps_Laplacian}, by replacing \(v_\epsilon \) from Equation \eqref{eq:v-epsilon-def} by \(v_{\epsilon,\alpha }\) from Equation \eqref{eq:v-epsilon-alpha-def}, 
    since the integral
    \begin{equation}
        \int_{M_\epsilon}G_0(x;y) \rho_0^{1-\alpha}(y) \vol_{\bar g_0} (y )
    \end{equation}
    converges to the integral
    \begin{equation}
        \int_{\overline{M}} G_0(x;y) \rho_0^{1-\alpha}(y) \vol_{\bar g_0}( y)  
    \end{equation}
    as \(\epsilon\to 0 \), as for every fixed \(x \in M^\circ \), we may bound \(|G_0(x;y)\rho_0(y)^{1-\alpha}| \leq C_x \rho_0(y)^{2-\alpha}\) for some constant \(C_x \) by \cref{lem:G-0-epsilon}.
\end{proof}
\Cref{lem:normal_integral} can also be modified for the gas giant geometries. 
\begin{lemma}\label{lem: g i alpha}
    Let \(f_{1,\alpha },f_{2,\alpha}\in C^\infty(\overline{M })\) be the functions from \cref{lem:w-eps-alpha-laplacian}. Then,  there are functions \(g_{i,\alpha}:\partial M_\epsilon \to \RR\) such that for \(z\in \partial M_\epsilon , \alpha \in (0,2)\), the following estimates hold. 
    If \(\alpha> 1 \), then 
    \begin{align}
            \int_{M_\epsilon}\partial_{\nu_z}G_0(z;y)(f_{1,\alpha}(y)\epsilon^{2-\alpha}+f_{2,\alpha}(y))\vol_{\bar{g}_0}(y) = g_{1,\alpha}(z) \epsilon^{2-\alpha} + g_{2,\alpha} (z)
    \end{align}
    if \(\alpha = 1\), then
    \begin{align}
            \begin{split}
                \int_{M_\epsilon}\partial_{\nu_x}G_0(z;y)(-f_{1,1}(y)\epsilon(\log \epsilon-1+\log 2 )+f_{2,1}(y))&\vol_{\bar{g}_0}(y) \\
                &= g_{1,1}(z) \epsilon + g_{2,1}(z)
            \end{split}
    \end{align}
    and if \(\alpha< 1 \), then
    \begin{equation}
            \int_{M_\epsilon}\partial_{\nu_z} G_0(z;y)(f_{1,\alpha}(y)\epsilon + f_{2,\alpha})\vol_{\bar g_0}(y) = g_{1,\alpha}(z) \epsilon^{2-\alpha} + g_{2,\alpha}(z)
    \end{equation}
    with 
    \(\sup_{z\in \partial M_\epsilon}|g_{i,\alpha}(z)| = \mathcal{O}(1) \), as \(\epsilon\to 0 \).
\end{lemma}
\begin{proof}
    The proof is identical to the proof of \cref{lem:normal_integral} with the change that the integral 
    \[
        \sup_{z\in \partial M_\epsilon}\left|\int_{M_\epsilon}\partial_{\nu_z}G_0(z;y) \beta^{(1)}(\theta(y)) \rho_0^{1-\alpha}(y)\vol_{\bar{g}_0}(y)\right| = \mathcal{O}(1)
    \]
    as \(\epsilon \to 0 \), 
\end{proof}
A consequence of these lemmas is the following proposition, similar to \cref{prop:normal_w_eps}.
\begin{proposition}\label{prop:w-eps-alpha-normal-deriv}
    The following estimate holds 
    \begin{equation}
        \|\partial_{\nu} w_{\epsilon,\alpha}(z)\|_{L^2(\partial M_\epsilon)} = \mathcal{O}(1)
        \quad \text{as \(\epsilon\to 0 \). }
    \end{equation}
\end{proposition}
\begin{proof}
    The proof is once again identical to the proof of \cref{prop:normal_w_eps} with the change that since the limit \( \lim_{\epsilon \to 0 } \epsilon^{2-\alpha} = 0\), we have 
    \begin{equation}
        \begin{split}
            \|\partial_{\nu}w_{\epsilon,\alpha}(z)\|_{L^2(\partial M_\epsilon)}
            &\leq  \left\|\left( \frac{1}{2} (I+N_\epsilon^\#)\right)^{-1}\right\|_{L^2\to L^2}\cdot 
            \|g_{1,\alpha}(z)\epsilon^{2-\alpha} + g_{2,\alpha}(z)\|_{L^2(\partial M_\epsilon)}
            \\
            &= \mathcal{O}(1)
        \end{split}
    \end{equation}
    by \cref{prop:double-layer-N-sharp,lem: g i alpha}.
\end{proof}
To finish the proof of \cref{thm:ggg} we modify \cref{prop:int-G0-delta-w}.
\begin{proposition}\label{prop:int-G0-delta-w-eps-alpha}
    Let \(\tilde u_{\epsilon,\alpha}(x) \) be the function defined in \cref{lem:w-eps-alpha-laplacian}. For \(x\in M^\circ \), there is a function \(r_{\epsilon,\alpha}(x) \) such that 
    \begin{equation}
        w_{\epsilon,\alpha}(x) = \begin{cases}
            \tilde u_{\epsilon,\alpha}(x)+\frac{1}{(\alpha-1)(2-\alpha)}(\chi(x)-1)  \epsilon^{2-\alpha} + r_{\epsilon,\alpha}(x)\quad &\text{if \(\alpha> 1 \)}\\
            \tilde u_{\epsilon,1}(x)+(\chi(x)-1) \epsilon(\log\epsilon-1+\log 2) + r_{\epsilon,\alpha}(x)& \text{if \(\alpha = 1\)}\\
            \tilde u_{\epsilon,\alpha}(x) -\frac{2^{\alpha-1}}{\alpha-1} (\chi(x)-1) \epsilon  + r_{\epsilon,\alpha}(x)&\text{if \(\alpha< 1 \)}
        \end{cases}
    \end{equation}
    with a positive constant \(D_{x,\alpha} \) depending on \( x \) such that \(|r_{\epsilon,\alpha}(x)|\leq D_{x,\alpha} \epsilon\) for sufficiently small \(\epsilon\). Furthermore, if \(K\Subset M^\circ \) is compact, then there is a positive constant \(D_{K,\alpha}\) such that \(\sup_{x\in K }|r_{\epsilon,\alpha}(x)|\leq D_{K,\alpha}\epsilon\).
\end{proposition}
\begin{proof}
    By Green's formula we have
    \begin{equation}
        \begin{split}
            \label{eq:w_eps_alpha-greens-formula}
            w_{\epsilon,\alpha}(x)
            &= \int_{M_\epsilon}G_0(x;y)\Delta_{\bar g_0 ,y}w_{\epsilon,\alpha}(y)\vol_{\bar g_0}(y)-\int_{\partial M_\epsilon}G_0(x;y) \partial_\nu w_{\epsilon,\alpha}(y) \vol_h(y) \ .
        \end{split}
    \end{equation}
    By \cref{eq:w-eps-alpha-laplacian}, it follows that 
    \begin{equation*}
        \Delta_{\bar{g}_0}w_{\epsilon,\alpha}(x) = \begin{cases}
            \frac{1}{(\alpha-1)(2-\alpha)}f_{1,\alpha}(x) \epsilon^{2-\alpha} +f_{2,\alpha}(x) \quad &\text{if \(\alpha> 1 \)}\\
            f_{1,1}(x)\epsilon(\log\epsilon-1+\log 2)  +f_{2,1}(x) \quad &\text{if \(\alpha = 1 \)}\\
            \frac{-2^{\alpha-1}}{\alpha-1}f_{1,\alpha} (x)\epsilon + f_{2,\alpha} (x)&\text{if \(\alpha< 1 \)}
        \end{cases}
    \end{equation*}
    Furthermore by \cref{eq:w-eps-alpha-I1}, we also have 
    \begin{equation}
        \int_{M_\epsilon} G_0(x;y) f_{1,\alpha}(y) \vol_{\bar g_0}(y) = \chi(x) -1  \ ,
    \end{equation}
    and for each \(x\in M^\circ  \) we have by \cref{eq:w-eps-alpha-I2}
    \begin{equation}
        \tilde u_{\epsilon,\alpha}(x) = \int_{M_\epsilon} G_0(x;y) f_{2,\alpha}(y) \vol_{\bar g_0 }(y) = \mathcal{O}(1) \quad \text{as \(\epsilon\to 0 \).}
    \end{equation}
    Finally, if \(K\Subset M^\circ \) is compact, then there is a smooth function \(\tilde u_\alpha(x) \in C^\infty(K)\) such that \(\tilde u_{\epsilon,\alpha }\) converges uniformly to \(\tilde u_\epsilon \) on \(K \).

    We set 
    \begin{equation}\label{eq:r-epsilon-alpha}
        r_{\epsilon,\alpha}(x): = -\int_{\partial M_\epsilon} G_0(x;y) \partial_\nu w_{\epsilon,\alpha}(y) \vol_h(y) \ .
    \end{equation}
    By \cref{prop:w-eps-alpha-normal-deriv}, it follows that \(\|\partial_\nu w_\epsilon(z)\|_{L^2(\partial M_\epsilon)} = \mathcal{O}(1)\) as \(\epsilon\to 0 \). By, \cref{lem:G-0-epsilon}, it follows that there is a constant \(C_x \) such that  \(\sup_{y\in \partial M_\epsilon}|G_0(x;y)| \leq C_x \epsilon \).
    Thus, there is a constant \(D_{x,\alpha} \) such that 
    \begin{equation}
        \begin{split}
            |r_{\epsilon,\alpha}(x)| &= \left|\int_{\partial M_\epsilon}G_0(x;y)\partial_\nu w_\epsilon(y) \vol_h(y) \right| \\
            &\leq \|G_0(x;\cdot)\|_{L^2(\partial M_\epsilon)}\cdot \|\partial_\nu w_\epsilon\|_{L^2(\partial M_\epsilon)} \leq D_{x,\alpha} \epsilon \ .
        \end{split}
    \end{equation}
    If \(K\Subset M^\circ  \) is compact, then by taking \(D_{K,\alpha} = \sup_{x\in K }D_{x,\alpha}\) we can conclude that \(\sup_{x\in K}|r_{\epsilon,\alpha}(x)| \leq D_{K,\alpha}\epsilon\).
    Combining these facts, if \(\alpha> 1\) then by \cref{eq:w-eps-alpha-laplacian,eq:w-eps-alpha-I1,eq:w-eps-alpha-I2,eq:r-epsilon-alpha}
    \begin{equation}
        \begin{split}
            w_{\epsilon,\alpha}(x)&= \int_{M_\epsilon}G_0(x;y) \Delta_{\bar g_0} w_{\epsilon,\alpha}(x)\vol_{\bar{g}_0}(y) -\int_{\partial M_\epsilon} G_0(x;y) \partial_\nu w_{\epsilon,\alpha}(y) \vol_h(y)\\
            &=\int_{M_\epsilon}G_0(x;y) \left(\frac{\epsilon^{2-\alpha}}{(\alpha-1)(2-\alpha)}f_{1,\alpha}(y) +f_{2,\alpha}(y)\right)\vol_{\bar{g}_0}(y) \\
            &\qquad \qquad \qquad \qquad \qquad\qquad -\int_{\partial M_\epsilon} G_0(x;y) \partial_\nu w_{\epsilon,\alpha}(y) \vol_h(y)\\
            &= \frac{1}{(\alpha-1)(2-\alpha)}(\chi(x)-1) \epsilon^{2-\alpha} + \tilde u_{\epsilon,\alpha}(x) + r_{\epsilon,\alpha}(x) \ ,
        \end{split}
    \end{equation}
    if \(\alpha =1 \), then 
    \begin{equation}
        \begin{split}
            w_{\epsilon,1} (x)&=  \int_{M_\epsilon}G_0(x;y) \Delta_{\bar g_0} w_{\epsilon,1}(x)\vol_{\bar{g}_0}(y) -\int_{\partial M_\epsilon} G_0(x;y) \partial_\nu w_{\epsilon,\alpha}(y) \vol_h(y)\\
            &=  \int_{M_\epsilon}G_0(x;y) (f_{1,1}(y) \epsilon(\log\epsilon-1+\log 2) +f_{2,1}(y))\vol_{\bar{g}_0}(y) \\
            &\qquad \qquad \qquad \qquad \qquad\qquad -\int_{\partial M_\epsilon} G_0(x;y) \partial_\nu w_{\epsilon,1}(y) \vol_h(y)\\
            & = (\chi(x)-1) \epsilon(\log\epsilon-1+\log 2)  + \tilde u_{\epsilon,1}(x) + r_{\epsilon,\alpha}(x) \ ,
       \end{split}
    \end{equation}
    if \(\alpha< 1 \), then 
    \begin{equation}
        \begin{split}
            w_{\epsilon,\alpha} (x)&= \int_{M_\epsilon}G_0(x;y) \Delta_{\bar g_0} w_{\epsilon,\alpha}(x)\vol_{\bar{g}_0}(y) -\int_{\partial M_\epsilon} G_0(x;y) \partial_\nu w_{\epsilon,\alpha}(y) \vol_h(y)\\
            &= \int_{M_\epsilon}G_0(x;y) \left(\frac{-2^{\alpha-1}}{\alpha-1}f_{1,\alpha}(y) \epsilon+f_{2,\alpha}(y)\right)\vol_{\bar{g}_0}(y) \\
            &\qquad \qquad \qquad \qquad \qquad\qquad-\int_{\partial M_\epsilon} G_0(x;y) \partial_\nu w_{\epsilon,\alpha}(y) \vol_h(y)\\
            &= -\frac{2^{\alpha-1}}{\alpha-1}(\chi(x)-1) \epsilon + \tilde u_{\epsilon,\alpha}(x) + r_{\epsilon,\alpha}(x)
        \end{split}
    \end{equation}
    with \(\tilde u_{\epsilon,\alpha} = \mathcal O(1) \) as \(\epsilon\to 0 \), and  a constant \(D_{x,\alpha }\) such that \(|r_{\epsilon,\alpha}(x)| \leq D_{x,\alpha}\epsilon \) for sufficiently small \(\epsilon\). If \(K\Subset M^\circ \) then there is a  function \(\tilde u_{\alpha }\) such that \(\tilde u_{\epsilon,\alpha}\) converges uniformly to \(\tilde u_\alpha \) on \(K \)  and a constant \(D_{K,\alpha }\) such that \(\sup_{x\in K }|r_{\epsilon,\alpha}(x)| \leq D_{K,\alpha}\epsilon\).
\end{proof}
\begin{proof}[Proof of \cref{thm:ggg}]
    Let \(\tilde u_{\epsilon,\alpha }\) be the function defined in \cref{lem:w-eps-alpha-laplacian}, and set 
    \[
        \hat{U}_{\epsilon,\alpha}(x) = \begin{cases}
            \tilde{u}_{\epsilon,\alpha}(x) + \chi(x) \left(v_{\epsilon,\alpha}(x) -  \frac{\epsilon^{2-\alpha}}{(\alpha-1)(2-\alpha)}\right) &\text{if \(\alpha > 1\)}\\
            \tilde u_{\epsilon,1}(x) + \chi(x) \left( v_{\epsilon,1}(x)  - \epsilon(\log \epsilon-1 + \log 2 ) \right)\quad &\text{if \(\alpha = 1\)}\\
            \tilde u_{\epsilon,\alpha}(x) + \chi(x) \left( v_{\epsilon,\alpha}(x) - \frac{2^{\alpha-1}}{\alpha-1}\epsilon \right) & \text{if \(\alpha< 1 \)}
        \end{cases} \ .
    \]
    We have by the definition of \(w_{\epsilon,\alpha }\) in  \cref{lem:w-eps-alpha-laplacian} that 
    \begin{equation}
        u_{\epsilon,\alpha}(x) =w_{\epsilon,\alpha}(x)+\chi(x)v_{\epsilon,\alpha}(x) 
    \end{equation}
    Thus by the definition of \(v_{\epsilon,\alpha }\) in  \eqref{eq:v-epsilon-alpha-def}, their  expansions in \cref{eq:v-alpha-epsilon-asymptotics}, and the expansions of \(w_{\epsilon,\alpha }\) in  \cref{prop:int-G0-delta-w-eps-alpha} we find that if \(\alpha>1 \) then 
    \begin{equation}
        \begin{split}
            u_{\epsilon,\alpha}(x) &= \frac{1}{(\alpha-1)(2-\alpha)}(\chi(x)-1) \epsilon^{2-\alpha} + \tilde u_{\epsilon,\alpha}(x) + r_{\epsilon,\alpha}(x) + \chi(x) v_{\epsilon,\alpha}(x)\\
            &= -\frac{1}{(\alpha-1)(2-\alpha)} \epsilon^{2-\alpha} + \hat U_{\epsilon,\alpha}(x) + r_{\epsilon,\alpha}(x) = \tilde U_{\epsilon,\alpha} (x) + r_{\epsilon,\alpha}(x) \ ,
        \end{split}
    \end{equation}
    if \(\alpha = 1 \), then 
    \begin{equation}
        \begin{split}
            u_{\epsilon,1}(x) &= -(\chi(x)-1)\epsilon(\log\epsilon-1+\log 2) + \tilde u_{\epsilon,1}(x) + r_{\epsilon,1}(x) + \chi(x) v_{\epsilon,1}(x)\\
            &= \epsilon(\log\epsilon-1+\log 2) + \hat U_{\epsilon,1}(x) + r_{\epsilon,1}(x) = \tilde U_{\epsilon,1}(x) + r_{\epsilon,1}(x) \ ,
        \end{split}
    \end{equation}
    and if \(\alpha<1 \), then 
    \begin{equation}
        \begin{split}
            u_{\epsilon,\alpha}(x) &= -\frac{2^{\alpha-1}}{\alpha-1}(\chi(x)-1) \epsilon + \tilde u_{\epsilon,\alpha}(x) + r_{\epsilon,\alpha}(x) + \chi(x) v_{\epsilon,\alpha}(x)\\
            &= \frac{2^{\alpha-1}}{\alpha-1} \epsilon + \hat U_{\epsilon,\alpha}(x) + r_{\epsilon,\alpha}(x) =\tilde U_{\epsilon,\alpha}(x) + r_{\epsilon,\alpha}(x)\  ,
        \end{split}
    \end{equation}
    where  for each fixed \(x\in M^\circ \) the function \(\tilde U_{\epsilon,\alpha}(x) = \mathcal O(1)\) as \(\epsilon\to 0 \) and for each fixed \(x\in M^\circ \) there is a constant \(D_{x}\) such that \(|r(x)|\leq D_x \epsilon \). Finally, if \(K \Subset M^\circ \) is compact, then there is a smooth function \(\tilde U_\alpha \in C^\infty(K)\) and a constant \(D_K \)  such that \(\tilde U_{\epsilon,\alpha}(x)\) converges uniformly on \(K \) to \(\tilde U_\alpha \)  and \(\sup_{x\in K }|r_\epsilon(x)| \leq D_K \epsilon\).
\end{proof}
\section{Limiting behaviour of MFET on gas giant surfaces to asymptotically hyperbolic surfaces}\label{sec:blowups}
In this section, we examine the behaviour of the mean first escape times of Brownian motion on the unit disc \(\DD  \) with a family of gas giant metrics 
\begin{equation}\label{eq:gas-giant-metric-disc}
    g_\alpha = \frac{2^\alpha (\dd x _1^2 + \dd x_2^2 )}{(1-x_1^2-x_2^2 )^\alpha}
\end{equation}
as \(\alpha \to 2\), i.e.\ as the metric \(g_\alpha\) tends from gas giant to the  Poincaré metric 
\begin{equation*}
    g_{PC} = \frac{4 (\dd x_1^2 + \dd x_2 ^2 )}{(1-x_1^2 -x_2^2)^2} \ .
\end{equation*}

We proved in \cref{thm:ggg} that if \(\alpha\in (0,2)\), then the mean first escape time of the Brownian motion starting at \(x \in M^\circ \) satisfies \(u_{\epsilon,\alpha}(x)= \mathcal{O}(1)\) as \(\epsilon\to 0 \), whereas in \cref{thm:hyperbolic}, we proved that the  mean first escape time of the Brownian motion starting at \(x \in M^\circ \) satisfies \(u_{\epsilon,2}(x)= \mathcal{O}(-\log \epsilon)\) as \(\epsilon\to 0 \). In this section we examine for \(x = (x_1,x_2) \in \DD \) the behaviour of \(u_{\epsilon,\alpha}(x)\), when \(\epsilon\to 0 \), \(\alpha \to 2 \) and \(\rho_0(x) \to \epsilon\) simultaneously. We show below that this behaviour can be described completely by localising the three variables \(\rho_0(x), \epsilon,\) and \(\alpha\) into various regimes, such that the function \(u_{\epsilon,\alpha}\) is nice with respect to  the natural coordinate systems defining said regimes.  By ``nice'', we mean that it admits an asymptotic expansion in those coordinates, or, in more technical (and accurate) language, that it is polyhomogeneous conormal.  We describe this in detail below after doing a few calculations.

In  \cref{eq:v-epsilon-alpha-def} the solutions \(u_{\epsilon,\alpha }(x)\) to the boundary value problem \eqref{eq:bvp} on the unit disc \(\DD  \) with the family of metrics \(g_\alpha\) from  \cref{eq:gas-giant-metric-disc} are given by
\begin{equation}
    \begin{split}
        u_{\epsilon,\alpha}(x) &= 
        \frac{\rho_0 ^{2-\alpha } \, _2F_1(1,2-\alpha ;3-\alpha ;2 \rho_0 ) - {\epsilon }^{2-\alpha } \, _2F_1(1,2-\alpha ;3-\alpha ;2 \epsilon )}{(\alpha -1) (2-\alpha )}\\ 
        &\hspace{5cm}  +\frac{2^{\alpha -2} (\log (1-2 \rho_0 )-\log (1-2 \epsilon ))}{\alpha -1}   \ ,
    \end{split}
\end{equation}
where \(1 < \alpha < 2\), and \(\rho_0 =\rho_0(x) =  \frac{1}{2}(1-x_1^2-x_2^2)\). We can rewrite this as 
\begin{equation}
    \label{eq:v-epsilon-alpha-asymptotics-easy}
    \begin{split}
        u_{\epsilon,\alpha}(x)&= \frac{\rho_0^{2-\alpha}}{(2-\alpha)(\alpha-1)}\left( 1-\left( \frac{\epsilon}{\rho_0} \right)^{2-\alpha} \right)\\
        &\quad+\frac{\rho_0^{2-\alpha}}{(\alpha-1)}\left( f_{\alpha}(2\rho_0) - \left( \frac{\epsilon}{\rho_0} \right)^{2-\alpha} f_{\alpha}\left(2\frac{\epsilon}{\rho_0}\rho_0\right) \right) \\
        &\quad+ \frac{2^{\alpha -2} \left(\log (1-2 \rho_0 )-\log \left(1-2 \frac{\epsilon}{\rho_0}\rho_0   \right)\right)}{\alpha -1} \ ,
    \end{split}
\end{equation}
where \[f_{\alpha}(z) = \frac{\, _2F_1(1,2-\alpha ;3-\alpha ;z) -1 }{(2-\alpha)} \in C^{\infty}\left( \Big[0,\frac 12\Big)_{z}\times (1,2]_\alpha\right) \ ,\] such that \(f_\alpha(0) = 0 \) and \(\sup_{z\in [0,1/2) }|f_{\alpha}(z)| \leq C\)  for some \(C>0 \) uniformly as \(\alpha\to  2 \). 
There are now three limits we wish to understand:
\begin{enumerate}
    \item the limit as \(\epsilon\to 0 \),
    \item the limit as  \(\rho_0 \to 0 \), note that this also requires \(\epsilon \to 0 \) simultaneously, and 
    \item the limit as \(\alpha\to 2^{-}\).
\end{enumerate}
We intend to describe this limiting behaviour using the theory of polyhomogeneous conormal distributions on manifolds with corners.  Despite the technical sounding jargon, this means that we will find coordinate patches, valid in various regimes of \((x, \epsilon, \alpha)\) space, such that \(u_{\epsilon, \alpha}(x)\) admits an asymptotic expansion when written in those coordinates. However, both to discover these coordinates and to prove the corresponding regularity statements, we use the theory of radial blow-ups developed by Melrose. For an introduction to the subject and a comprehensive overview we refer the reader to \cites{grieser2001basics,MelroseGreen,MelroseDAOMWC,schulze1991psdo}. In general, given some  manifold with corners \(M \) and a submanifold \(N \subset M \) (satisfying some conditions) we construct a new manifold with corners \(M' :=[M,N] \) together with a surjective map \(\pi: [M,N] \to M \) called the blow-down map, such that \(\pi: [M,N]^\circ\to M^\circ \) is a diffeomorphism, by replacing the submanifold \(N \) by a local product of a sphere \(S^{n-k-1}\) of the right dimension and \(\RR^k \) (in the language of the b-calculus such a map is a b-map, see \cite{grieser2001basics}*{Section 2.3.2} for more details).  \emph{At every step, the blow-up produces new systems of coordinates valid on different subsets of the space; below we specify at each step what these new valid coordinates are.}  Thus, the reader can take away simply that the function is nice in the sense that nice coordinates can be found according to a breakdown of the space into coordinate patches produced by our blow-up process.  All of this will be visualised using the blow-up picture in \cref{fig:triple-space}.

    First we begin by considering the manifold with corners 
    \[
        X_0 = [0,1)_y \times \left[0,\frac{1}{2}\right)_{\rho_0}\times \left[0,1\right)_\beta \ ,
    \]
    where, for $\rho_0 \neq 0$, we set 
    \[
        y = \epsilon /\rho_0, \quad \rho_0 = \rho_0, \quad \beta = 2-\alpha\ .
    \]  
    The functions \(\left\{ y, \rho_0, \beta  \right\}\) still provide \emph{global} boundary defining functions for the space we are analysing. This is because \(0\leq y = \epsilon/\rho_0 \leq 1 \) is bounded. On this manifold with corners we analyse the function
    \begin{equation}
        U(y,\rho_0,\beta) := u_{y\rho_0,2-\beta}(\rho_0) \ . 
    \end{equation}
    It turns out that \(X_0 \) together with the map 
    \[
        \varpi_0: X_0 \to \Big[0,\frac 12\Big)_\epsilon \times \Big[\epsilon,\frac 12\Big)_{\rho_0} \times (1,2]_\alpha :(y,\rho_0,\beta)\mapsto (y\rho_0,\rho_0,2-\beta )
    \]  
    is a blow-up space of the original space \([0,\frac 12)_{\epsilon}\times [\epsilon,\frac 12)_{\rho_0}\times (1,2]_\alpha\).

    In these coordinates we now find that 
    \begin{equation}\label{eq:U-X0}
        \begin{split}
            U(y,\rho_0,\beta)= \frac{\rho_0^\beta\left( 1-y^\beta \right)}{\beta(1-\beta)} &+ \frac{\rho_0^{\beta}\left( \overline f(2\rho_0,\beta) -  y^{\beta}\overline f
            \left( 2y \rho_0,\beta\right) \right) }{(1-\beta)} \\
            &+ \frac{2^\beta \left( \log (1-2\rho_0)- \log (1-2y\rho_0) \right)}{1-\beta}\  ,
        \end{split}
    \end{equation}
    where \(\overline f (z,\beta) = f_{2-\beta}(z) \in C^\infty([0,\frac 12)_z\times [0,1)_\beta)\).
    By \cite{grieser2001basics}*{Remark 2.4}, we note that this function is not a polyhomogeneous conormal function; indeed, such functions cannot have coordinate-dependent exponents. To start to resolve this, we introduce  the coordinate changes  
    \begin{equation}\label{eq:ilog-def}
        \eta = -1/\log (y)\quad\text{and}\quad {\xi = -1/\log(\rho_0)} \ .
    \end{equation}
    This amounts to changing the smooth structure of \(X_0\), and is equivalent to performing   \emph{logarithmic blow-ups} of the boundary hypersurfaces \(B_0^1 = \left\{ y = 0  \right\}\subset X_0\) and \(B_0^2 = \left\{ \rho_0 = 0  \right\} \subset X_0 \) denoted by \([X_0, B_0^1]_{\log} \) and \([X_0,B_0^2]_{\log}\) with the property that the iterated logarithmic blow-up commutes  \cite{MelroseDAOMWC}*{p.\ V.28}, i.e.\ the space \(X_1 \) 
    \[
        X_1 = [[X_0,B_0^1]_{\log},B_0^2]_{\log}= [[X_0,B_0^2]_{\log},B_0^1]_{\log} \ ,
    \]
    is a well-defined blow-up of \(X_0 \) with a surjective map  \(\varpi_1: X_1 \to X_0 \) satisfying the properties of a blow-down map, and with a  \(C^\infty\) structure of a manifold with corners
    (see  \cite{MelroseDAOMWC}*{p.\ V.28} for more information on this coordinate change -- one of the characteristics of this coordinate change is that finite order vanishing turns into infinite order vanishing). 

    On the space \(X_1 \) with global boundary defining functions \(\left\{ \eta,\xi,\beta \right\}\), the function \(\varpi^*_1 U\) is given by 
    \begin{equation}\label{eq:ilogs}
        \begin{split}
            (\varpi_1^* U)(\eta,\xi,\beta) &= \frac{e^{-\beta/\xi}}{\beta(1-\beta)}\left( 1-e^{-\beta/\eta} \right) \\
            &\quad+ \frac{e^{-\beta/\xi}}{(1-\beta)}\left( \overline f(2e^{-1/\xi},\beta) + e^{-\beta/\eta}\overline f (2e^{-1/\xi\eta},\beta) \right) \\
            &\quad + \frac{2^{\beta}}{1-\beta}{\left( \log (1-2e^{-1/\xi}) - \log (1-2e^{-1/\xi\eta}) \right)}\ .
        \end{split} 
    \end{equation}
    We will show that there is a blow-up space \(X \) of \(X_1 \) in which \(\varpi^*_2 (\varpi_1 ^*U)\) is polyhomogeneous conormal.

    We do this by first considering the third term in the sum of \cref{eq:ilogs}. The function
    \begin{equation*}
        \frac{2^\beta \left( \log (1-2e^{-1/\xi})- \log (1-2e^{-1/\xi\eta}) \right)}{1-\beta}
    \end{equation*}
    is smooth on \([0,1)_\eta\times [0,1)_\xi\times [0,1)_{\beta}\) and infinitely vanishing at the boundary hypersurfaces given by \(B_1^1=\left\{ \eta = 0 \right\} \subset X_1\) and \(B_1^2= \left\{ \xi = 0  \right\}\subset X_1\).
    
    Next we consider the function 
    \[
        \overline{U}(\eta,\xi,\beta) = \frac{e^{-\beta/\xi}}{\beta(1-\beta)}(1-e^{-\beta/\eta}) \ .
    \]
    We first analyse the limits of this function to the three different coordinate axes away from the origin. In \cref{fig:triple-space}, these regions are denoted by A, B, and C.
    \begin{itemize}
        \item[A:] \(\beta\gg 0 \) fixed: The function \(\overline U(\eta,\xi,\beta)\) is already smooth around here. This is clear from \eqref{eq:ilogs}, as \(e^{-1/\eta}\) and \(e^{-1/\xi}\) are smooth functions on \([0,1)_\eta\) and \([0,1)_\xi\) respectively.  We further see infinite order vanishing as \(\xi \to 0 \), and a 0\textsuperscript{th} order vanishing as \(\eta \to 0 \).
        \item[B:] \(\xi\gg 0 \) fixed: Whenever \(\xi \gg 0 \), we can analyse the function 
        \begin{equation*}
            \overline U(\eta,\xi,\beta) = \frac{e^{-\beta/\xi}}{\beta(1-\beta)}\left( 1-e^{-\beta/\eta} \right) = \frac{e^{-\eta(\beta/\eta\xi)}}{\eta(\beta /\eta)(1-\eta (\beta/\eta))}\left( 1-e^{-\beta/\eta} \right) 
        \end{equation*}
        around the \(\xi \) axis. We do a standard blow-up of this axis away from \(\xi = 0 \).
        We recognise that there are two regimes:
        \begin{enumerate}
            \item where \(\eta > \beta \) and  \[\{\eta,\quad  s := \beta/\eta\}\] provide boundary defining functions 
            \item where \(\beta > \eta \) and \[\{\beta, \quad t := \eta/\beta \}\] provide boundary defining functions. 
        \end{enumerate}
        For case 1, we recognise that 
        \begin{equation}\label{eq:gives back AH}
            \overline U (\eta,\xi,s) = \frac{e^{-\eta s /\xi}}{\eta(1-\eta s )}\left( \frac{1-e^{-s}}{s }   \right)
        \end{equation}
        is polyhomogeneous conormal on \([0,1)_\eta\times [0,1)_s\), since \((1-e^{-s})/s \in C^\infty[0,1)_s \). Notice further that \eqref{eq:gives back AH} returns the asymptotics from \cref{thm:hyperbolic}. By the definition of the derivative, we find that the limit  \(\lim_{s\to 0}(1-e^{-s})/s =1\), so that \( \overline U(\eta,\xi,0)= 1/\eta = \log \rho_0/\epsilon \), which is what was found in \cref{thm:hyperbolic}.
    
        For case 2, we find that 
        \begin{equation*}
            \overline U(t,\xi,\beta)=  \frac{e^{-\beta/\xi}}{\beta(1-\beta)} \left( 1-e^{-1/t} \right)
        \end{equation*}
        is polyhomogeneous conormal on \([0,1)_{\beta} \times [0,1)_t\), since \(e^{-1/t }\in C^\infty[0,1)_t \) is smooth all the way up to the boundary and rapidly decaying. 
        \item[C:] \(\eta \gg 0 \) fixed: We once again rewrite the function \(\overline U(\eta,\xi,\beta)\)
        \begin{equation*}
            \overline U(\eta,\xi,\beta)=  \frac{e^{-\beta/\xi}}{\beta(1-\beta)}\left( 1-e^{-\beta/\eta} \right)  = \frac{e^{-\beta/\xi}}{\xi (\beta/\xi)(1-\xi(\beta/\xi))}\left( 1-e^{-\xi(\beta/\xi)/\eta} \right)  \ .
        \end{equation*}
        We now do a standard blow-up of the \(\eta\)-axis away from \(\eta = 0 \). We once again get two regimes:
        \begin{enumerate}
            \item where \(\xi>\beta \) and \[\{\xi,\quad q: = \beta/\xi\}\] provide boundary defining functions
            \item where \(\beta > \xi \) and \[\{\beta, \quad r := \xi /\beta \}\] provide boundary defining functions.
        \end{enumerate}
        For case 1, we see that 
        \begin{equation*}
            \begin{split}
                \overline U(\eta,\xi,q) &= \frac{e^{- q}}{\xi q(1-\xi q )}\left( 1-e^{-\xi q/\eta} \right) \\
                &= \frac{e^{-q}}{\eta(1-\xi q )}\sum_{k=0}^\infty \frac{1}{(k+1)!}\left( \frac{-\xi q}{\eta} \right)^k \in C^\infty[0,1)_\xi\times [0,1)_q \ .
            \end{split}
        \end{equation*}
        We can also once again recover \(\overline U(\eta,\xi,0) = 1/\eta = \log \rho_0/\epsilon \) here by noting that \(q = 0 \) if and only if \(\beta = 0\).
        For case 2, we get that 
        \begin{equation*}
            \begin{split}
                \overline U(\eta,r,\beta)&= \frac{e^{-1/r}}{\beta(1-\beta)}\left( 1-e^{-\beta/\eta} \right) \\
                &=  \frac{e^{-1/r}}{\eta(1-\beta)}\sum_{k=0}^\infty \frac{1}{(k+1)!} \left( \frac{-\beta}{\eta} \right)^k \in C^\infty[0,1)_\beta \times [0,1)_r    
            \end{split}
        \end{equation*}
        from which we see that \(\overline U(\nu,\xi,\mu)\) is polyhomogeneous conormal. We still need this blow-up as the limits from the two directions are different.
    \end{itemize}
    
    Finally, we have the origin to analyse. We consider three regimes, and will see that the blow-ups of the coordinate axes as above are required. The different regimes are drawn in \cref{fig:triple-space}.
    \begin{enumerate}
        \item[I:] \(\beta > \xi,\eta\), where the coordinates \[\{\gamma:=\eta/\beta,\quad\vartheta:=\xi/\beta ,\quad \beta \}\] provide boundary defining functions. Here
        \begin{equation*}
            \overline U(\gamma,\vartheta,\beta) = \frac{e^{-1/\vartheta }}{\beta(1-\beta)}\left( 1-e^{-1/\gamma } \right) \ ,
        \end{equation*}
        which is polyhomogeneous conormal on \([0,1)_\gamma\times [0,1)_\vartheta \times [0,1)_\beta \).
        \item[II:] \(\xi> \beta,\eta\), where the coordinates \[\{\nu:=\eta/\xi,\quad\xi,\quad\mu:=\beta/\xi \}\] provide boundary defining functions. Here 
        \begin{equation*}
            \overline U(\nu,\xi,\mu) = \frac{e^{-\mu}}{ \xi\mu(1-\xi \mu)}\left( 1-e^{-\mu/\nu} \right) = \frac{1}{\xi}\left( \frac{e^{-\mu}}{\mu(1-\xi \mu)}\left( 1-e^{-\mu/\nu} \right) \right) \ ,
        \end{equation*}
        which is not polyhomogeneous conormal, but we recognise that there are two regimes 
        \begin{enumerate}
            \item[IIa:] \(\mu>\nu\), where \[\{\sigma := \nu/\mu= \eta/\beta,\quad \xi,\quad \mu= \beta/\xi  \}\] provide boundary defining coordinates. Here:
            \begin{equation*}
                \overline U(\sigma,\xi,\mu) = \frac{1}{\xi}\left( \frac{e^{-\mu}}{\mu(1-\xi \mu)}\left( 1-e^{-1/\sigma} \right) \right)\ ,
            \end{equation*}
            which is polyhomogeneous conormal on \([0,1)_\sigma\times [0,1)_\xi\times [0,1)_\mu\).
            \item[IIb:] \(\nu>\mu\), where \[\{\nu = \eta/\xi,\quad\xi,\quad \tau:=\mu/\nu = \beta/\eta \}\] provide boundary defining functions. Here 
            \begin{equation*}
                \overline U(\nu,\xi,\tau) = \frac{1}{\xi\nu}\left( \frac{e^{-\nu \tau}}{\tau(1-\xi\nu\tau)}\left( 1-e^{-\tau} \right) \right)\ ,
            \end{equation*}
            which is polyhomogeneous conormal on \([0,1)_\nu\times [0,1)_\xi\times [0,1)_\tau\). Once again we recover \cref{thm:hyperbolic} here, because \(\xi \nu = \xi \eta/\xi = \eta\). Because \(\tau = 0 \) if and only if \(\beta= 0 \), we find that  \(\overline U(\nu,\xi,0) = 1/(\xi\nu) = \log \rho_0/\epsilon\) in this regime.
        \end{enumerate}
        \item[III:] \(\eta > \beta,\xi\), where \[\{\eta , \quad \chi:=\xi/\eta,\quad \phi:=\beta/\eta \}\] provide boundary defining functions. Here 
        \begin{equation*}
            \overline U(\eta,\chi,\phi) = \frac{1}{\eta}\left( \frac{e^{-\phi/\chi}}{\phi(1-\eta\phi)}(1-e^{-\phi}) \right)\ ,
        \end{equation*}
        which is not polyhomogeneous conormal on \([0,1)_\eta\times [0,1)_\phi\times [0,1)_\chi \), but we once again recognise that there are two regimes
        \begin{enumerate}
            \item[IIIa:] \(\phi >\chi\), where \[\{\eta,\quad \psi:=\chi/\phi=\xi/\beta,\quad \phi = \beta/\eta\}\] provide boundary defining functions. Here:
            \begin{equation*}
                \overline U(\eta,\psi,\phi) = \frac{1}{\eta}\left( \frac{e^{-1/\psi}}{\phi(1-\eta\phi)}\left( 1-e^{-\phi} \right) \right)\ ,
            \end{equation*}
            which is polyhomogeneous conormal on \([0,1)_\eta\times [0,1)_\psi\times [0,1)_\phi\).
            \item[IIIb:] \(\chi>\phi \), where \[\{\eta,\quad \chi=\xi/\eta,\omega:=\phi/\chi =\beta/\xi\}\] provide boundary defining functions. Here:
            \begin{equation*}
                \overline U(\eta,\chi,\omega)= \frac{1}{\eta}\left( \frac{e^{-\omega}}{\chi\omega(1-\eta\chi\omega)}\left( 1-e^{\chi\omega} \right) \right)\ ,
            \end{equation*}
            which is polyhomogeneous conormal on \([0,1)_\eta\times [0,1)_\chi\times [0,1)_\omega\). Finally, we also recover \cref{thm:hyperbolic} here. Since \(\omega = 0 \) if and only if \(\beta = 0 \), we find that \(\overline u_{\epsilon,0} = 1/\eta = \log \rho_0/\epsilon\) in this regime. 
        \end{enumerate}
    \end{enumerate}
    We conclude that the function \(\overline U(\eta,\xi,\beta)\) resolves to a polyhomogeneous conormal function on the iterated blow-up space 
    \begin{equation}\label{eq:triple-space}
        X = 
        \left[ \left[ X_1 ; \left\{ 0 \right\} \right];\left\{ 0 \right\}_\eta\times[0,1)_\xi\times \{0\}_\beta\cup \left[ 0,1 \right)_\eta\times\{0\}_\xi\times \{0\}_\beta \right] \ ,
    \end{equation}
    by blowing up the origin of the space \(X_1 \), and then separately blowing up the coordinate axes, given by the sets \({\{0\}_\eta\times [0,1)_\xi\times \{0\}_\beta}\) and \({[0,1)_\eta\times \{0\}_\xi\times \{0\}_\beta}\).
    \Cref{fig:triple-space} gives an impression of the blow-up space \(X \) as a blow-up from \(X_1 \). By \cite{grieser2001basics}*{Remark 3.13}, it follows that this space comes with a surjective map \(\varpi_{2}:X\to X_1 \) satisfying the properties of a blow-down map. Now \(\varpi = \varpi_2\circ\varpi_1 : X\to X_0\) is a well-defined blow-down map.

    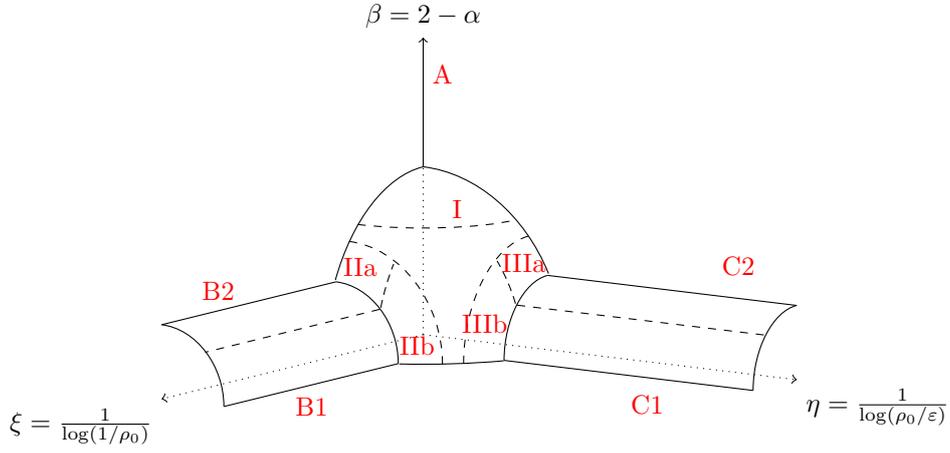
\begin{figure}[tbp]
        \centering
        
        \begin{tikzpicture}[scale=2]
            \tdplotsetmaincoords{80}{125}
            \begin{scope}[tdplot_main_coords]
                
                \draw[->,dotted] (0,0,0) -- (3,0,0) node[anchor=north east]{\(\xi = \frac{1}{\log(1/\rho_0)}\)};
                
                \draw[->,dotted] (0,0,0) -- (0,3,0) node[anchor=north west]{\(\eta = \frac{1}{\log(\rho_0 / \epsilon)}\)};
                
                \draw[dotted] (0,0,0) -- (0,0,1.13) ;
                \draw[->] (0,0,1.13) -- (0,0,2) node[anchor=south]{\(\beta = 2-\alpha\)};
                \node[red,anchor = west] at (0,0,1.75) {A};
                
                \tdplotdrawarc{(0,0,0)}{1.13}{27}{63}{anchor= south east}{\color{red}{IIb} \ } ;
                \tdplotdrawarc[dashed]{(0,0,0.85)}{0.75}{0}{90}{anchor=south west}{{\color{red} I}} ;
                \tdplotsetrotatedcoords{0}{90}{0} 
                \tdplotdrawarc[tdplot_rotated_coords]{(0,0,0)}{1.13}{117}{180}{anchor=north}{} ;
                \tdplotdrawarc[tdplot_rotated_coords,dashed]{(0,0,0.85)}{0.75}{90}{180}{anchor=north}{} ;
                \tdplotdrawarc[tdplot_rotated_coords,densely dashed]{(0,0,0.35)}{1.05}{110}{127}{anchor=south }{\; \; {\color{red}  IIIa}} ;
                
                \tdplotdrawarc[tdplot_rotated_coords]{(0,0,1)}{0.5}{90}{180}{anchor=north}{} ;
                \tdplotdrawarc[tdplot_rotated_coords]{(0,0,3)}{0.5}{90}{180}{anchor=north}{} ;
                \draw[tdplot_rotated_coords] (0,0.5,1) -- (0,0.5,3);
                \draw[tdplot_rotated_coords] (-0.5,0,1) -- (-0.5,0,3);
                \draw[tdplot_rotated_coords,dashed] (-0.35,0.35,1) -- (-0.35,0.35,3);
                \node[red,anchor = west] at (3,0.25,.75) {B2};
                \node[red,anchor = west] at (3,1,0.05) {B1};

                \tdplotsetrotatedcoords{90}{-90}{-90} 
                \tdplotdrawarc[tdplot_rotated_coords]{(0,0,0)}{1.13}{27}{90}{anchor=north}{} ;
                \tdplotdrawarc[tdplot_rotated_coords,dashed]{(0,0,-0.85)}{0.75}{0}{90}{anchor=east}{} ;
                \tdplotdrawarc[tdplot_rotated_coords,densely dashed]{(0,0,-0.35)}{1.05}{20}{37}{anchor= south east}{{\;  \color{red} IIa}} ;

                \tdplotdrawarc[tdplot_rotated_coords]{(0,0,-1)}{0.5}{0}{90}{anchor=north east}{\color{red}{IIIb}} ;
                \tdplotdrawarc[tdplot_rotated_coords]{(0,0,-3)}{0.5}{0}{90}{anchor=north}{} ;
                \draw[tdplot_rotated_coords] (0,0.5,-1) -- (0,0.5,-3);
                \draw[tdplot_rotated_coords] (0.5,0,-1) -- (0.5,0,-3);
                \draw[tdplot_rotated_coords,dashed] (0.35,0.35,-1) -- (0.35,0.35,-3);
                
                \node[red,anchor = west] at (0.25,2.5,.75) {C2};
                \node[red,anchor = north] at (1,2.5,0.05) {C1};
            
            \end{scope}
        \end{tikzpicture}
        \caption{Required blow-up space to resolve the function \( U(y,\rho_0,\beta)\) from \cref{eq:U-X0}.}
        \label{fig:triple-space}
    \end{figure}
    
    Since the functions \(\overline f(2y,\beta) \) and \(\overline f(2\rho_0,\beta)\) from \cref{eq:U-X0} are smooth on \(X_0 \), they remain smooth after blow-ups \cite{grieser2001basics} and thus by following the same procedure as before, one can show that the function 
    \begin{equation}\label{eq:infinite-order-vanishing}
        \frac{e^{-\beta/\xi}}{(1-\beta)}\left( \overline f(2e^{-1/\xi},\beta)+e^{-\beta/\eta} \overline f(2e^{1/\xi\eta},\beta) \right)
    \end{equation}
    resolves to a smooth function on the same space \(X \) given by  \cref{eq:triple-space}. Since the function \(\overline f (z,\beta)\) vanishes for \(z = 0 \), by the discussion on logarithmic blow-ups it follows that the order of vanishing of the functions \(\overline f(2e^{-1/\xi},\beta)\) and \(\overline f(2e^{-1/\xi\eta},\beta) \) in  \cref{eq:infinite-order-vanishing} are infinite.
    We can therefore conclude that there is a manifold with corners \(X \) and a surjective map \(\varpi: X\to X_0 \), which is a diffeomorphism on the interior,  such that \(\varpi^* U \) is polyhomogeneous conormal on \(X\).

Given the discussion above about the mean first escape times of Brownian motion on a family of gas giant metrics on the unit disc, we conjecture the following. 
\begin{conjecture}\label{con:blow-up}
    Let \((M,g_\alpha )\) be a two-dimensional  smooth manifold with a family of gas giant metrics \(g_\alpha \). Assume that the function \(\beta^{(1)}(\theta)\) from  \cref{eq:bvp_simplified} is identically equal to 0, and let  \(U(x,\epsilon,\alpha):= u_{\epsilon,\alpha}(x) \) be the solutions to the boundary value problems \cref{eq:bvp} on the gas giant surfaces \((M,g_\alpha)\). 
    Then there exists a manifold with corners \(Z \) obtained by the blow-up
    \begin{equation}
        Z = \bigg[\Big[[[Z_0,B_0^1]_{\log},B_0^2]_{\log},\{0\}\Big],\left\{ 0 \right\}_\eta\times[0,1)_\xi\times \{0\}_\beta\cup \left[ 0,1 \right)_\eta\times\{0\}_\xi\times \{0\}_\beta\bigg] \ ,
    \end{equation}
    where 
    \[
        Z_0 =  \left[0,\frac 12\right)_{y}\times \partial  M \times\left[0,\frac 12\right)_{\rho_0} \times [0,1)_\beta \ ,\qquad y =\frac{\epsilon}{\rho_0}, \quad \beta = 2-\alpha \ ,
    \] 
    the hypersurfaces \(B_0^1 = \{y = 0\}, B_0^2 = \{\rho_0 = 0 \}\subset Z_0\) and the functions 
    \[
        \eta = -1/\log (y) \quad \text{and}\quad \xi = -1/\log (\rho_0 ) \ ,
    \]
    together with a surjective map 
    \[
        \varpi:Z \to \overline{M_\epsilon}\times \left[0,\frac{1}{2}\right)_\epsilon\times (1,2]_\alpha
    \]
    such that \(\varpi \) is a blow-down map and such that \(\varpi^* U\) is polyhomogeneous conormal on \(Z \). 
\end{conjecture}
The manifold \(Z_0  \) is  the natural extension of the manifold \(X_0 \) to the general case, by considering 
\begin{equation*}
    \overline M_\epsilon \setminus M_\delta\times \left[ 0,\frac{1}{2} \right)_\epsilon \cong \partial M \times \left[ \epsilon,\frac 12 \right)_{\rho_0} \times \left[ 0,\frac{1}{2} \right)_\epsilon \ .
\end{equation*}
Then \(Z \) is obtained by performing the same blow-ups as before. The condition of the function \(\beta^{(1)}(\theta)\) being identically 0, is  due to  \cref{lem:normal_error_principal}. As a consequence of \cref{lem:normal_error_principal}, then  we get estimates for \(\partial_\nu w_\epsilon = \mathcal O (\log \epsilon)\), and consequently together with \cref{lem:G-0-epsilon} only we get pointwise estimates for \(u_{\epsilon}(x)\) as \(\epsilon \to 0 \), whereas we require uniform estimates for \cref{con:blow-up}.

\section{Numerical results}\label{sec:numerics}
In this section we will compare the analytical results from the previous sections with simulations of the Brownian motion, and with numerical PDE solvers on the unit disc.

To generate the Brownian motion on the Riemannian manifolds, we use the algorithms described in \cite{Schwarz2023}. Up to a rescaling of constants in the generator of the Brownian motion, Algorithms 1 and 2 provide an efficient method of generating Brownian motion on Riemannian manifolds. The resulting generator is a time rescaled version of the generator used in this paper. To account for this fact we multiply the analytically found solutions by a factor of 2.

\begin{figure}[tp]
    \centering
    \subfloat[\(\alpha = 0.1\)]{\label{fig:0.1}\includegraphics[scale=0.9]{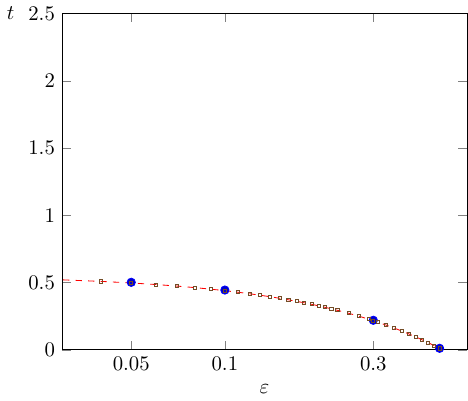}}
    \hfill
    \subfloat[\(\alpha = 0.5\)]{\label{fig:0.5}\includegraphics[scale=0.9]{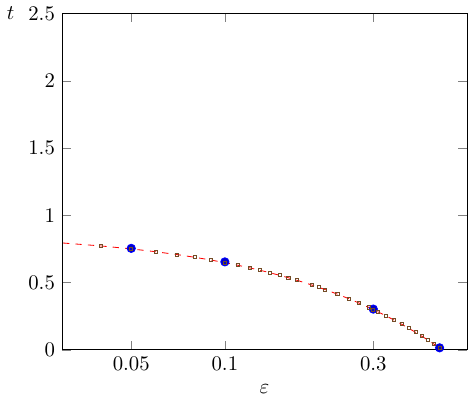}}\\
    \subfloat[\(\alpha = 1.0\)]{\label{fig:1.0}\includegraphics[scale=0.9]{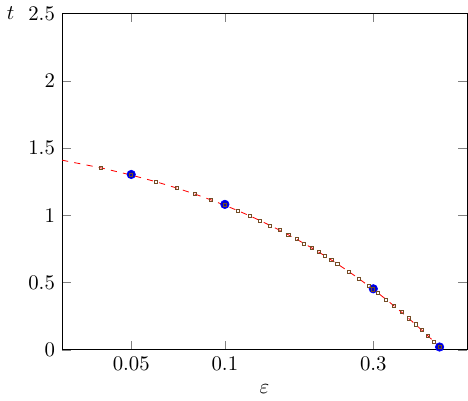}}
    \hfill
    \subfloat[\(\alpha = 1.5\)]{\label{fig:1.5}\includegraphics[scale=0.9]{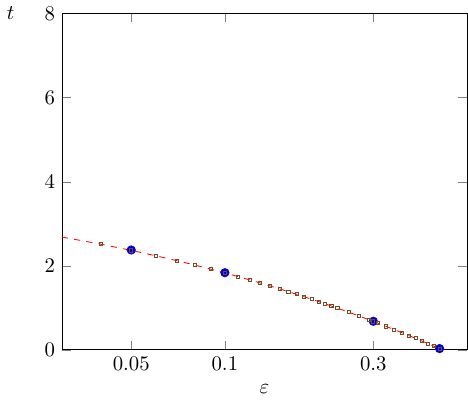}}\\
    \subfloat[\(\alpha = 1.9\)]{\label{fig:1.9}\includegraphics[scale=0.9]{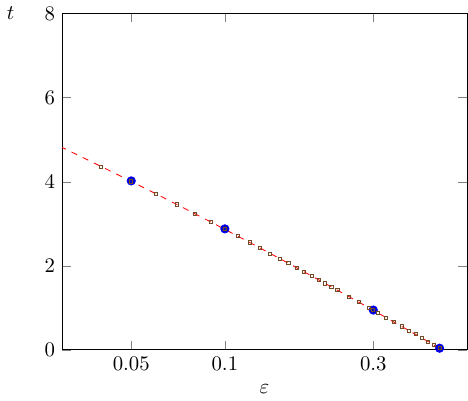}}
    \hfill
    \subfloat[\textit{Asymptotically hyperbolic}]{\label{fig:2.0}\includegraphics[scale=0.9]{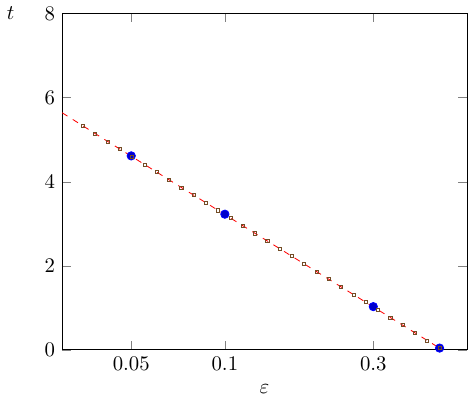}}
    \caption{The mean first escape time of a Monte Carlo simulation of the Brownian motion starting at the origin on the unit disc \(\DD \) with the gas giant metrics from \eqref{eq:gas-giant-metric-disc} of order \(\alpha\) and the Poincaré disc as a function of \(\epsilon \) in blue based on 300 000 simulations.
    The square marks denote the analytically found solutions in Equation \eqref{eq:v-epsilon-def} and in the gas giant case in Equation \eqref{eq:v-epsilon-alpha-def}.
    The dashed red lines denote the value of the numerically found solution to the boundary value problem \cref{eq:bvp}.}
    \label{fig:numerical-hyperbolic-ggg}
\end{figure}

\Cref{fig:numerical-hyperbolic-ggg} shows the mean first escape time of a Monte Carlo simulation of the Brownian motion starting at the origin on the unit disc \(\DD \) with the gas giant metrics from \eqref{eq:gas-giant-metric-disc} of order \(\alpha\) and the Poincaré disc as a function of \(\epsilon \) in blue based on 300 000 simulations.

The square marks denote the analytically found solutions in Equation \eqref{eq:v-epsilon-def} and in the gas giant case in Equation \eqref{eq:v-epsilon-alpha-def}.

Using Wolfram's built-in function \lstinline|ParametricNDSolve|~\cite{WolframR2012}, we solved the boundary value problem \eqref{eq:bvp} on the disc. The value of this numerically found solution at the origin is given in \cref{fig:numerical-hyperbolic-ggg} as the dashed red line.
We can see that both  the Monte Carlo simulations  of the Brownian motion, the analytical solution, and the numerical solutions correspond, and are finite as \(\epsilon\to 0 \) if the metric is gas giant and \( -\log(\epsilon)\) as \(\epsilon\to 0 \) if the metric is asymptotically hyperbolic.

We can also see that indeed as the gas giant parameter \(\alpha \) approaches \(2 \), the mean first escape time of the Brownian motion on the gas giant of order \(\alpha\), approaches the mean first escape time of the Brownian motion on the asymptotically hyperbolic disc.
\appendix
\crefalias{section}{appendix}
\section{Proof of Proposition \ref{lem:normal_integral}}\label{app:prop3.2}
In this part of the appendix we give the details to the proof of \cref{lem:normal_integral}. We start with the proof of \cref{lem:normal_integral_no_singularities}.
\begin{proof}[Proof of \cref{lem:normal_integral_no_singularities}]
    The conclusion for Equation \eqref{eq:del chi - 1} follows from the observation that 
    \begin{equation*}
        \int_{M_\epsilon} \partial_{\nu_z}G_0(z;y) f_1(y)\vol_{\bar g_0}(y)  = \partial_{\nu_z}\int_{M_\epsilon} G_0(z;y) f_1(y)\vol_{\bar g_0}(y)  = \partial_{\nu_z}(\chi(z)-1) = 0
    \end{equation*}
    for \(z\in \partial M_\epsilon\) by \cref{eq: f1}. This follows from the fact that \(f_1(y) \equiv 0 \) for every \(y \) with  \(\rho_0(y) < \delta/2\), the Green's function being smooth away from the diagonal, such that the difference quotient \(|G_0(\rho_0(z) -h , \theta(z);y)f_1(y)- G_0(z;y)f_1(y)|/h \) can be dominated by an absolutely integrable function \(H(z;y)\) on \(M_\epsilon\), see e.g. \cite{klenke2020probability}*{Theorem 6.28}.

    For \(\tilde g_2\) given by \cref{eq: tilde g2}, we notice that the  functions \((1-\chi)/\rho_0^2,\Delta_{\bar{g}_0}\chi \cdot \log \rho_0 \) and \(\inner{d\chi}{d v_\epsilon}_{\bar{g}_0}\in C^\infty(M_\epsilon)\) are supported on a subset of \(M_{\delta/2}\), and independent of \(\epsilon \), hence they are uniformly bounded as \(\epsilon\to 0 \). By construction of \( G_{0,\prin }(x;y) \) in \cref{def:Greens_principal} we have that by \cref{eq:kappa-properties} that \(\supp G_{0,\prin}(z;\cdot)|_{z\in \partial M_\epsilon }\subset \overline M \setminus M_{\delta/3}\) and thus that 
    \begin{equation}\label{eq: tilde g2 1}
        \int_{M_\epsilon}\left|\partial_{\nu_z} G_{0,\prin}(z;y) \left[ \frac{1-\chi}{\rho_0^2}+\Delta_{\bar{g}_0}\chi \cdot \log \rho_0+\inner{d\chi}{d v_\epsilon}_{\bar{g}_0} \right]\right|\vol_{\bar{g}_0}(y) = 0  \ .
    \end{equation}
    By \cref{prop_part:greens decomp 2} of \cref{prop:greens decomposition} and since the functions \((1-\chi)/\rho_0^2,\Delta_{\bar{g}_0}\chi \cdot \log \rho_0 \) and \(\inner{d\chi}{d v_\epsilon}_{\bar{g}_0}\in C^\infty(M_\epsilon)\) are uniformly bounded by some constant \(D \) on \(M_\epsilon \) as \(\epsilon \to 0 \), it follows that there is some constant \(D' \) such that 
    \begin{equation}\label{eq: tilde g2 2}
        \begin{split}
            \sup_{z\in \partial M_\epsilon}&\int_{M_\epsilon\setminus M_{\delta/4}}\left| \partial_{\nu_z} G_{0,\rem}(z;y)\left[ \frac{1-\chi}{\rho_0^2}+\Delta_{\bar{g}_0}\chi \cdot \log \rho_0+\inner{d\chi}{d v_\epsilon}_{\bar{g}_0} \right]\right|\vol_{\bar{g}_0}(y)\\
            &\leq \int_{M_\epsilon\setminus M_\delta} (C_0\rho_0(y) + C_1(\epsilon + \rho_0(y))^2 + C_2(\epsilon + \rho_0(y))^3)\cdot  D \vol_{\bar g_0}(y) \leq D' 
        \end{split}
    \end{equation}
    uniformly as \(\epsilon \to 0  \). Finally, there is a constant \(E \) such that the integral 
    \begin{equation}\label{eq: tilde g2 3}
        \int_{M_{\delta/4}}\left| \partial_{\nu_z} G_{0,\rem}(z;y)\left[ \frac{1-\chi}{\rho_0^2}+\Delta_{\bar{g}_0}\chi \cdot \log \rho_0+\inner{d\chi}{d v_\epsilon}_{\bar{g}_0} \right]\right|\vol_{\bar{g}_0}(y)\leq E 
    \end{equation}
    is uniformly bounded as \(\epsilon \to 0 \). This is because the integral in \cref{eq: tilde g2 3} is defined by integration over the manifold \(M_{\delta/4  } \), and since the functions \(G_{0,\rem }(z;y) \) and the functions \((1-\chi)/\rho_0^2,\Delta_{\bar{g}_0}\chi \cdot \log \rho_0 \) and \(\inner{d\chi}{d v_\epsilon}_{\bar{g}_0}\in C^\infty(M_\epsilon)\) are smooth and uniformly bounded away from the boundary \(\partial M_\epsilon\), it follows by the dominated convergence theorem,  that the integral is uniformly bounded as \(\epsilon\to 0 \) since the integrand is absolutely integrable on \(M_{\delta/4}\).

    To conclude the proof of \cref{lem:normal_integral_no_singularities} we have that 
    \begin{equation*}
        \begin{split}
            \sup_{z\in \partial M_\epsilon}|\tilde g_2 (z) | &\leq \int_{M_\epsilon}\left|\partial_{\nu_z}G_0(z;y) \left[ \frac{1-\chi}{\rho_0^2}+\Delta_{\bar{g}_0}\chi \cdot \log \rho_0+\inner{d\chi}{d v_\epsilon}_{\bar{g}_0} \right]\right|\vol_{\bar{g}_0}(y)\\
            &\leq \int_{M_\epsilon}\left|\partial_{\nu_z}G_{0,\prin}(z;y) \left[ \frac{1-\chi}{\rho_0^2}+\Delta_{\bar{g}_0}\chi \cdot \log \rho_0+\inner{d\chi}{d v_\epsilon}_{\bar{g}_0} \right]\right|\vol_{\bar{g}_0}(y)\\
            &\ + \int_{M_\epsilon\setminus M_{\delta/4 }}\left|\partial_{\nu_z}G_{0,\rem }(z;y) \left[ \frac{1-\chi}{\rho_0^2}+\Delta_{\bar{g}_0}\chi \cdot \log \rho_0+\inner{d\chi}{d v_\epsilon}_{\bar{g}_0} \right]\right|\vol_{\bar{g}_0}(y)\\
            &\ + \int_{M_{\delta/4 }}\left|\partial_{\nu_z}G_{0,\rem }(z;y) \left[ \frac{1-\chi}{\rho_0^2}+\Delta_{\bar{g}_0}\chi \cdot \log \rho_0+\inner{d\chi}{d v_\epsilon}_{\bar{g}_0} \right]\right|\vol_{\bar{g}_0}(y) \ .
        \end{split}
    \end{equation*}
    Each of the three integrals on the right-hand side is uniformly bounded as \(\epsilon \to 0 \) by  \cref{eq: tilde g2 1,eq: tilde g2 2,eq: tilde g2 3}. Hence, the estimate \(\sup_{z\in \partial M_\epsilon}|\tilde g_2(z)| = \mathcal O (1) \) as \(\epsilon \to 0 \) holds.
\end{proof}
We now continue with the claim made in the proof of \cref{lem:normal_integral}.
\begin{proposition}\label{prop:claim_3.2}
    There are families of  functions  \(h_1,h_2: \partial M_\epsilon \to \RR \) such that 
    \[
        \int_{M_\epsilon}\partial_{\nu_z}G_0(z;y) \frac{\beta^{(1)}(\theta(y))\rho_0(y)+ \beta^{(2)}(\rho_0,\theta)}{\rho_0(y)^2}\vol_{\bar{g}_0}(y) = h_1(z) + h_2(z) 
    \]
    where \(\sup_{z\in \partial M_\epsilon}|h_1(z) | = \mathcal{O}(-\log \epsilon)\), and \(\sup_{z\in \partial M_\epsilon}|h_2(z)|= \mathcal O(1)\) as \(\epsilon \to 0 \).
\end{proposition}
We prove this proposition by restricting to different regimes \(M_{\delta/4}\) and \(M_\epsilon\setminus M_{\delta/4}\), and using the decomposition from \cref{prop:greens decomposition} in the regime \(M_{\epsilon}\setminus M_{\delta/4}\). \Cref{lem:normal_error_principal,lem:normal_error_far_away} give the details.
\begin{lemma}\label{lem:normal_error_principal}
    Let \(G_{0,\mathrm{prin}}\) be the principal part of the Green's function from \cref{def:Greens_principal}. Then the function 
    \begin{equation}\label{eq: int Gprin}
        h_1(z) = \int_{M_\epsilon\setminus M_{\delta/4}}\partial_{\nu_z}G_{0,\mathrm{prin}}(z;y) \frac{\beta^{(1)}(\theta(y))}{\rho_0(y)}\vol_{\bar{g}_0}(y) 
    \end{equation}
    satisfies the estimate \(\sup_{z\in \partial M_\epsilon}|h_1(z)| = \mathcal{O}(-\log \epsilon )\) as \(\epsilon\to 0 \).
\end{lemma}
\begin{proof}
    We are going to explicitly compute the integral \eqref{eq: int Gprin} on the annulus \(\DD\setminus \DD_{1-\bar\delta}\), since we can give a global structure of the principal part of the Green's function on the annulus. We have by \cref{rem:Greens disc} on the annulus \(\DD\setminus \DD_{1-\bar\delta}\)
    \begin{equation}\label{eq:g-0-prin-normalderiv}
        \begin{split}
            &\partial_{r_z}G_{0,\mathrm{prin},\mathcal{A}}(\sqrt{1-2\epsilon},\vartheta_z;r_y,\vartheta_y)=\\
            & \frac{\left(r_y^2-1\right) \left(\left(r_y^2+1\right) \sqrt{1-2 \epsilon }+2 r_y (\epsilon -1) \cos (\vartheta_y-\vartheta_z)\right)}{2 \pi  \left(r_y^2-2 r_y \sqrt{1-2 \epsilon } \cos (\vartheta_y-\vartheta_z)-2 \epsilon +1\right) \left(r_y^2 (2 \epsilon -1)+2 r_y \sqrt{1-2 \epsilon } \cos (\vartheta_y-\vartheta_z)-1\right)} 
        \end{split}
    \end{equation}
    in polar coordinates \((r_y,\vartheta_y)\) of the annulus \(\DD\setminus \DD_{1-\bar\delta}\). After a change of variables \(u_y =\vartheta_y- \vartheta_z \) we have 
    \begin{equation*}
        \begin{split}
            h_1(z) &= \int_{M_\epsilon\setminus M_{\delta/4}}\partial_{\nu_z}G_{0,\mathrm{prin}}(z;y) \frac{\beta^{(1)}(\theta(y))}{\rho_0(y)}\vol_{\bar{g}_0}(y) \\
            &= \int_{\sqrt{1-\delta/2}}^{\sqrt{1-2\epsilon}}\int_{-\pi}^\pi\partial_{r_z}G_{0,\mathrm{prin},\mathcal{A}}(\sqrt{1-2\epsilon},0;r_y,u_y)\frac{2r_y \beta^{(1)}(u_y+\vartheta_z)}{1-r_y^2}\dd u\dd r_y 
        \end{split}
    \end{equation*}
    Then, since \(\beta^{(1)}\) is uniformly bounded on \(M\setminus M_\delta\), we can, with the use the triangle inequality for integrals and computer algebra systems such as Wolfram, bound the inner integral (over the angle \(\vartheta_y\)) by 
    \begin{equation*}
        \begin{split}
            |h_1(z)| &=\left|\int_{\sqrt{1-\delta/2}}^{\sqrt{1-2\epsilon}}\int_{-\pi}^\pi\partial_{r_z}G_{0,\mathrm{prin},\mathcal{A}}(\sqrt{1-2\epsilon},0;r_y,u)\frac{2r_y \beta^{(1)}(u+\theta_z)}{1-r_y^2}\dd u\dd r_y\right|  \\
            &\leq \int_{\sqrt{1-\delta/2}}^{\sqrt{1-2\epsilon}}\int_{-\pi}^\pi\left|\partial_{r_z}G_{0,\mathrm{prin},\mathcal{A}}(\sqrt{1-2\epsilon},0;r_y,u)\frac{2r_y \beta^{(1)}(u+\theta_z)}{1-r_y^2}\right|\dd u\dd r_y 
        \end{split}
    \end{equation*}
    which by \cref{eq:g-0-prin-normalderiv} is equal to 
    \begin{equation}\label{eq:to-bound-normal-error}
        {\int\limits_{\sqrt{1-\delta/2}}^{\sqrt{1-2\epsilon} }\int\limits_{-\pi}^\pi \left|\frac{2r_y \beta^{(1)}(u + \vartheta_z)\cdot \left( \left(r_y^2+1\right) \sqrt{1-2 \epsilon }+2 r_y (\epsilon -1) \cos (u) \right)}{2 \pi  \left(r_y^2-2 r_y \sqrt{1-2 \epsilon } \cos (u)-2 \epsilon +1\right) \left(r_y^2 (2 \epsilon -1)+2 r_y \sqrt{1-2 \epsilon } \cos (u)-1\right)} \right|\dd u \dd r_y }\ .
    \end{equation}
    With computer algebra systems such as Wolfram, we find that the integral 
    \begin{equation}\label{eq:cas}
        \begin{split}
            \int\limits_{-\pi}^\pi \bigg|&\frac{2r_y  \left( \left(r_y^2+1\right) \sqrt{1-2 \epsilon }+2 r_y (\epsilon -1) \cos (u) \right)}{2 \pi  \left(r_y^2-2 r_y \sqrt{1-2 \epsilon } \cos (u)-2 \epsilon +1\right) \left(r_y^2 (2 \epsilon -1)+2 r_y \sqrt{1-2 \epsilon } \cos (u)-1\right)} \bigg|\dd u \\
            &\qquad = \frac{ r_y}{\sqrt{1-2\epsilon}(1-r_y^2)} \ .
        \end{split}
    \end{equation}
    Bounding \(\beta^{(1)}(u + \vartheta_z) \) by a uniform constant \(C \), and plugging the result from \cref{eq:cas} into \cref{eq:to-bound-normal-error} we find that 
    \begin{equation*}
        \begin{split}
            \int_{\sqrt{1-\delta/2}}^{\sqrt{1-2\epsilon} }&\int_{-\pi}^\pi \left|\frac{2r_y \beta^{(1)}(u + \theta_z)\cdot \left( \left(r_y^2+1\right) \sqrt{1-2 \epsilon }+2 r_y (\epsilon -1) \cos (u) \right)}{2 \pi  \left(r_y^2-2 r_y \sqrt{1-2 \epsilon } \cos (u)-2 \epsilon +1\right) \left(r_y^2 (2 \epsilon -1)+2 r_y \sqrt{1-2 \epsilon } \cos (u)-1\right)} \right|\dd u \dd r_y\\
            &\leq  \int_{\sqrt{1-\delta/2}}^{\sqrt{1-2\epsilon}}\frac{C r_y}{\sqrt{1-2\epsilon}(1-r_y^2)}\dd r_y   \ .
        \end{split}
    \end{equation*}
    Hence,  we conclude that 
    \begin{equation*}
        |h_1(z)| \leq  \int_{\sqrt{1-\delta/2}}^{\sqrt{1-2\epsilon}}\frac{C r_y}{\sqrt{1-2\epsilon}(1-r_y^2)}\dd r_y =\frac{-C \log\epsilon}{2\sqrt{1-2\epsilon}} + \mathcal{O}(1)
    \end{equation*}
    as \(\epsilon \to 0 \). Then, it follows that \(\sup_{z\in \partial M_\epsilon}|h_1(z)|=\mathcal{O}(-\log \epsilon  )\) as \(\epsilon\to 0 \).
\end{proof}
\begin{lemma}\label{lem:normal_error_far_away}
    The function 
    \begin{equation}\label{eq: int G far}
        \begin{split}
            h_2(z) = &\int_{M_{\delta/4}}\partial_{\nu_z}G_0(z;y) \frac{\beta^{(1)}(\theta(y))}{\rho_0(y)}\vol_{\bar{g}_0}(y) \\
            &+ \int_{M_\epsilon\setminus M_{\delta/4} }\partial_{\nu_z}G_{0,\mathrm{rem}}(z;y)\frac{\beta^{(1)}(\theta(y))}{\rho_0(y)}\vol_{\bar{g}_0}(y)\\
            &+ \int_{M_{\epsilon}}\partial_{\nu_z}G_0(z;y) \frac{\beta^{(2)}(\rho_0(y),\theta(y))}{\rho_0(y)^2}\vol_{\bar{g}_0}(y) 
        \end{split}
    \end{equation}
    satisfies the estimate \(\sup_{z\in \partial M_\epsilon}|h_2(z) | = \mathcal{O}(1)\) as \(\epsilon \to 0 \).
\end{lemma}

\begin{proof}[Proof of \cref{lem:normal_error_far_away}]
    The first integral on the right hand side of \eqref{eq: int G far} is defined by integration over the manifold \(\rho_0(y)> \delta/4\). Since \(G_0(z;y)\) and \({1}/{\rho_0(y) } \)  are smooth and uniformly bounded as \(\rho_0(z) \to 0 \) away from the boundary \(\partial M_\epsilon\). The estimate therefore holds for the first integral on the right hand side of \eqref{eq: int G far} by the dominated convergence theorem since the integrand is absolutely integrable.
    
    For the second integral on the right hand side of \eqref{eq: int G far}, we notice that by \cref{prop_part:greens decomp 2} of \cref{prop:greens decomposition}, and since \(\beta^{(1)}(\theta)\) is  bounded on \(\partial M \). 
    It follows that for \(z\in \partial M_\epsilon \) there are constants \(C_0',C_1',C_2'  \) such that 
    \begin{equation*}
        \sup_{z\in \partial M_\epsilon}\left|\partial_{\nu_z}G_{0,\rem}(z;y) \frac{\beta^{(1)}(\theta(y))}{\rho_0(y)}\right| \leq C_0' + C_1'\frac{(\epsilon+\rho_0(y))^2}{\rho_0(y)} + C_2'\frac{(\epsilon+\rho_0(y))^3}{\rho_0(y)}   \ .
    \end{equation*} 
    Hence, there are constants \(D_0,D_1,D_2 \) such that  
    \begin{equation}
        \begin{split}
            &\sup_{z\in \partial M_\epsilon}\left|\int_{M_{\epsilon}\setminus M_{\delta/4}}\partial_{\nu_z}G_{0,\rem}(z;y) \frac{\beta^{(1)}(\theta(y))}{\rho_0(y)}\vol_{\bar g_0}(y)\right|\\
            &\leq \sup_{z\in \partial M_\epsilon} \int_{M_{\epsilon}\setminus M_{\delta/4}}\left|C_0' + C_1'\frac{(\epsilon+\rho_0(y))^2}{\rho_0(y)} + C_2'\frac{(\epsilon+\rho_0(y))^3}{\rho_0(y)}\right|\vol_{\bar g_0}(y)\\
            &\leq D_0 - D_1 \epsilon\log \epsilon - D_2 \epsilon^2 \log \epsilon \ .
        \end{split}
    \end{equation}
    Hence, the estimate holds as \(\epsilon\to 0 \) for the second integral on the right hand side of \eqref{eq: int G far}.
    
    Finally for the third integral on the right hand side of \eqref{eq: int G far}, we use the same decomposition of the Green's function from \cref{prop:greens decomposition} to split the third integral on the right hand side of \cref{eq: int G far}. Following the same steps as in the proof of \cref{lem:normal_error_principal} and the second integral on the right hand side of \eqref{eq: int G far}, we find that the third integral on the right hand side of \eqref{eq: int G far} satisfies the estimate.

    All together this shows that \(\sup_{z\in \partial M_\epsilon}|h_2(z)| = \mathcal O(1) \) as \(\epsilon \to 0 \).
\end{proof}
\Cref{lem:normal_error_principal,lem:normal_error_far_away} are together are precisely \cref{prop:claim_3.2}.
\section{Proof of Proposition \ref{prop:double-layer-N-sharp}}\label{app:double-layer-N-sharp-pf-general}
For the proof of \cref{prop:double-layer-N-sharp}, we give the Fourier multipliers of \(N_{\epsilon,\mathrm{prin}}^\#\) on the annulus.  By then using the isometry between the annulus \((\DD\setminus \DD_{1-\overline \delta},g_{\Euc})\), and the manifold \((M\setminus M_{\delta},\bar g_0)\), the Fourier multipliers remains identical on each of the manifolds.
On the annulus \((\DD\setminus \DD_{1-\overline \delta},g_{\Euc})\), the operators \(N_{\epsilon}^\# ,N_{\epsilon,\prin }^\# :L^2(\partial \DD_{\sqrt{1-2\epsilon}})\to L^2(\partial \DD_{\sqrt{1-2\epsilon}})\) are equal and defined by  \cref{eq:N-sharp-def}. In particular, for \(z\in \partial \DD_{\sqrt{1-2\epsilon}}\)
\begin{equation}
    (N^\#_\epsilon f)(z) = (N^\#_{\epsilon,\prin} f)(z) = 2 \int_{\partial \DD_{\sqrt{1-2\epsilon}}} \partial_{\nu_z}G_{0,\prin,\A}(z;y) f(y) \vol_{\partial \DD_{1-2\epsilon}}(y) \ ,
\end{equation}
with \(G_{0,\prin,\A }\) defined in \cref{rem:Greens disc}. We now compute the Fourier multipliers of \(N_{\epsilon,\prin}^\# \) on the annulus. 
\begin{proposition}\label{prop:app}
    Let \(f \in L^2 (\partial \DD_{\sqrt{1-2\epsilon}})\), then 
    \begin{equation}
        \begin{split}
            ({N}_{\epsilon,\mathrm{prin}}^\#f&)^\wedge (n) \\
            &=\frac{1}{2\pi}\int_{-\pi}^\pi e^{-in\vartheta_z}\frac{-(1-\epsilon)\epsilon }{\pi\sqrt{1-2\epsilon}}\int_{-\pi}^\pi \frac{\sqrt{1-2\epsilon}f(\vartheta_y)\dd \vartheta_y}{(1-2\epsilon)\cos(\vartheta_z - \vartheta_y)+2(1-\epsilon)\epsilon -1 } \dd \vartheta_z  \\
            &= (1-2\epsilon)^{|n|}\hat{f}(n) \ ,
        \end{split}
        \label{eq:N_sharp_fourier_series}
    \end{equation}
    for sufficiently small \(\epsilon>0 \).
\end{proposition}
\begin{proof}
    The principal part of the Green's function for the annulus  \(G_{0,\mathrm{prin},\A}(x;y) \) is given  by \cref{rem:Greens disc}.
    In polar coordinates \((r_x,\vartheta_x )\) and \((r_y,\vartheta_y) \) for the annulus this is given by 
    \begin{equation}
        \begin{split}
            G&_{0,\prin,\A}(x;y) \\
            &=\frac{\log \left(r_x^2-2 r_x r_y \cos (\vartheta_x-\vartheta_y)+r_y^2\right)-\log \left(r_x^2 r_y^2-2 {r_x} {r_y} \cos (\vartheta_x-\vartheta_y)+1\right)}{4 \pi } \ .
        \end{split}
    \end{equation} 
    Taking \(x \in \DD\setminus \DD_{1-\overline \delta}\) to a point \(z \in \partial \DD_{\sqrt{1-2\epsilon} }\) and taking an outward pointing normal derivative \(\partial_{\nu_z}\) shows that in polar coordinates the principal kernel \(N_{\epsilon,\mathrm{prin}}^\#(\vartheta_z,\vartheta_y)\) 
    is given by 
    \begin{equation}
        \begin{split}
            N_{\epsilon,\mathrm{prin}}^\#(\vartheta_z,\vartheta_y) &= \partial\nu_{z}|_{|z|,|y| = \sqrt{1-2\epsilon}}G_{0,\mathrm{prin}, \A }(\vartheta_z,\vartheta_y) \\
            &= -\frac{(\epsilon -1) \epsilon }{2 \pi  \sqrt{1-2\epsilon} ((2 \epsilon -1) \cos (\vartheta_z-\vartheta_y)+2 (\epsilon -1) \epsilon +1)} \ .
        \end{split}
        \label{eq:kernel_normal_derivative}
    \end{equation}
    Since the circle of radius \(\sqrt{1-2\epsilon } \) has its induced measure given by  \(\sqrt{1-2\epsilon}\dd \vartheta_z \) we find that \((N_{\epsilon,\prin}^\# f)^\wedge (n)\) is given by the integral on the right-hand side of \eqref{eq:N_sharp_fourier_series}.
    By Fubini-Tonelli, we may exchange the order of integration in \eqref{eq:N_sharp_fourier_series}. We find that 
    \begin{equation*}
        (N_\epsilon^\# f)^\wedge (n) = \frac{1}{2\pi}\int_{-\pi}^\pi f(\vartheta_y) \frac{-(1-\epsilon)\epsilon }{\pi\sqrt{1-2\epsilon}} \int_{-\pi}^\pi \frac{e^{-in\vartheta_z}\sqrt{1-2\epsilon}\dd \vartheta_z}{(1-2\epsilon)\cos(\vartheta_z - \vartheta_y)+2(1-\epsilon)\epsilon -1 } \dd \vartheta_y \ .
    \end{equation*}
    The proof to  \cref{prop:app} is complete once we establish that 
    the integral 
    \begin{equation}
        I(n)= \frac{-(1-\epsilon)\epsilon }{\pi\sqrt{1-2\epsilon}} \int_{-\pi}^\pi \frac{e^{-in\vartheta_z}\sqrt{1-2\epsilon}\dd \vartheta_z}{(1-2\epsilon)\cos(\vartheta_z - \vartheta_y)+2(1-\epsilon)\epsilon -1 } = (1-2\epsilon)^{|n|} e^{-in\vartheta_y} \ .
        \label{eq:N_sharp_fourier_series-integral}
    \end{equation}
    Assume first that \(n\geq 0 \). Take a change of variables \(\Theta = \vartheta_z -\theta_y \). So for \eqref{eq:N_sharp_fourier_series-integral}, we are now taking the  integral 
    \begin{equation}
        e^{-in\vartheta_y}\frac{-(1-\epsilon)\epsilon}{\pi\sqrt{1-2\epsilon}}\int_{-\pi}^\pi \frac{e^{-in\Theta}\sqrt{1-2\epsilon}\dd \Theta}{(1-2\epsilon)\cos(\Theta)+2(1-\epsilon)\epsilon -1 } \ ,
    \end{equation}
    We apply the change of variables \(w= e^{-i\Theta},iw\dd  w=  \dd \Theta\), and use the calculus of residues. 
    Then
    \begin{equation}
        \begin{split}
            I(n) &=e^{-in\theta_y}\frac{-(1-\epsilon)\epsilon}{\pi}\int_{|w|=1}\frac{iw^{n-1}\dd w}{\frac{1}{2}(1-2\epsilon)(w+1/w)+2(1-\epsilon)\epsilon-1}\\
            &=-e^{-in\theta_y}{2(1-\epsilon)\epsilon}\cdot \mathrm{Res}_{w= 1-2\epsilon}\frac{w^n }{\frac{1}{2}(1-2\epsilon)(w^2+1)+2(1-\epsilon)\epsilon w-w} \\
            &= e^{-in\theta_y}(1-2\epsilon)^{n}  \ ,
        \end{split} 
    \end{equation}
    due to an orientation change.
    For \(n<0 \), we have the symmetry of the integral \eqref{eq:N_sharp_fourier_series-integral}, such that 
    \[
        I(n)=I(-n)  = e^{-in\vartheta_y} (1-2\epsilon)^{-n}=e^{-in\vartheta_y}(1-2\epsilon)^{|n|} \ ,
    \]
    which proves the proposition.
\end{proof}
\begin{proof}[Proof of \cref{prop_part:double-layer-N-sharp-prin} of \cref{prop:double-layer-N-sharp}]
    Let \(\Phi: (\widetilde M\setminus M_\delta)\to  (\RR^2\setminus \DD_{1-\overline \delta})\) be the isometry from \cref{def:rho0}. Since \(\Phi \) is an isometry from \(\partial M_\epsilon \) to \(\partial \DD_{\sqrt{1-2\epsilon} }\)  when restricted to \(\partial M_\epsilon\), the pullback \(\Phi^* \) induces an isometric isomorphism \(\Phi^*: L^2(\partial \DD_{\sqrt{1-2\epsilon}})\to L^2(\partial M_\epsilon)\). Now the operator \(N^\#_{\epsilon,\prin }:L^2(\partial M_\epsilon)\to L^2(\partial M_\epsilon )\) can abstractly be defined as 
    \begin{equation}\label{eq:N-sharp-abstract}
        N^\#_{\epsilon,\prin} = (\Phi^*)\circ N_{\epsilon}^\# \circ (\Phi^{-1})^* :L^2(\partial M_\epsilon)\to L^2(\partial M_\epsilon)\ .
    \end{equation}
    Since \(G_{0,\prin}\in \mathcal D'(M\times M )\) is defined in terms of \(G_{0,\prin,\A }\) and \(\Phi \), it follows that the definitions from \cref{def:N-sharp} and \cref{eq:N-sharp-abstract} correspond. Since \(\Phi \) is an isometric isomorphism, for \(f\in L^2(\partial M_\epsilon)\)  the Fourier transform of \(f \) is  defined in \cref{def:fourier} as  
    \begin{equation*}
       \hat{f}(n) = \mathcal F \circ (\Phi^{-1})^* f\ ,
    \end{equation*}
    such that 
    \begin{equation*}
        (N_{\epsilon,\prin}^\# f)^\wedge (n) = (\mathcal F\circ  N_{\epsilon}^\#\circ (\Phi^{-1})^*f)(n) =  (1-2\epsilon)^{|n|} (\mathcal{F} \circ (\Phi^{-1})^* f )(n) = (1-2\epsilon)^{|n|}\hat{f}(n) \ ,
    \end{equation*}
    which is \cref{prop_part:double-layer-N-sharp-prin} of \cref{prop:double-layer-N-sharp}.
\end{proof}

The proof  of \cref{prop:double-layer-N-sharp} is complete once we  show \cref{prop_part:double-layer-N-sharp-rem}. This is the content of the following lemma.
\begin{lemma}[\cref{prop_part:double-layer-N-sharp-rem} of \cref{prop:double-layer-N-sharp}]
    The estimate 
    \begin{equation}
        \|N_{\epsilon,\rem}^\#\|_{L^2(\partial M_\epsilon)\to L^2(\partial M_\epsilon)} = \mathcal{O}(\epsilon)
    \end{equation}
    as \(\epsilon\to 0 \) holds.
\end{lemma}
\begin{proof}
    By \cref{prop:greens decomposition}, it follows that there are  constants \(C_0,C_1,C_2 \) such that  for all \(y\in \partial M_\epsilon\) and all \(\epsilon>0 \)
    \begin{equation}
        \begin{split}
            \sup_{z\in \partial M_\epsilon}&\left|\partial_{\nu_z}G_{0,\rem}(z;y)\right|_{{y\in \partial M_\epsilon}} \\
            &\leq C_0 \rho_0(y) + C_1(\rho_0(y)+\epsilon)^2 + C_2(\rho_0(y) + \epsilon)^3\bigg|_{\rho_0(y)\in \partial M_\epsilon} \\
            &= C_0 \epsilon + C_1(\epsilon+\epsilon)^2 + C_2(\epsilon+ \epsilon)^3\ . 
        \end{split}
    \end{equation} 
    Therefore, by the triangle inequality,
    \begin{equation}
        \begin{split}
            \|N_{\epsilon,\rem}^\# f \|^2 =\int_{\partial M_\epsilon}&\left|\int_{\partial M_\epsilon}\partial_{\nu_z}G_{0,\rem}(z;y)f(y)\vol_h(y)\right|^2\vol_h(z) \\
            &\qquad\leq( C_0^2 \epsilon^2 + 4C_1 \epsilon^4 + 4 C_2 \epsilon^6)|\partial M_\epsilon|_h^2\cdot \|f\|_{L^2}^2\ ,
        \end{split}
    \end{equation}
    from which the conclusion follows.
\end{proof}

\bibliography{refs}

@article{nursultanov2021mean,
  author     = {Nursultanov, Medet and Tzou, Justin C. and Tzou, Leo},
  title      = {On the mean first arrival time of {B}rownian particles on
                {R}iemannian manifolds},
  journal    = {J. Math. Pures Appl. (9)},
  fjournal   = {Journal de Math\'ematiques Pures et Appliqu\'ees. Neuvi\`eme
                S\'erie},
  volume     = {150},
  year       = {2021},
  pages      = {202--240},
  issn       = {0021-7824,1776-3371},
  mrclass    = {58J65 (60J65 92C37)},
  mrnumber   = {4248467},
  mrreviewer = {Anton\ Thalmaier},
  doi        = {10.1016/j.matpur.2021.04.006}
}

@book{taylor2011pde2,
  author    = {Taylor, Michael E.},
  title     = {Partial differential equations {II}. {Q}ualitative studies of
               linear equations},
  series    = {Applied Mathematical Sciences},
  volume    = {116},
  edition   = {Third},
  publisher = {Springer, Cham},
  year      = {2023},
  pages     = {xxiii+687},
  isbn      = {978-3-031-33699-7; 978-3-031-33700-0},
  mrclass   = {35-01 (46N20 47F05 47N20)},
  mrnumber  = {4703939},
  doi       = {10.1007/978-3-031-33700-0}
}

@book{lee2018riemann,
  author     = {Lee, John M.},
  title      = {Introduction to {R}iemannian manifolds},
  series     = {Graduate Texts in Mathematics},
  volume     = {176},
  edition    = {Second},
  publisher  = {Springer, Cham},
  year       = {2018},
  pages      = {xiii+437},
  isbn       = {978-3-319-91754-2; 978-3-319-91755-9},
  mrclass    = {53-01 (53B20 53B30 53C20 53C21)},
  mrnumber   = {3887684},
  mrreviewer = {Robert\ J.\ Low}
}

@article{Nursultanov2023,
  author     = {Nursultanov, Medet and Trad, William and Tzou, Justin and
                Tzou, Leo},
  title      = {The narrow capture problem on general {R}iemannian surfaces},
  journal    = {Differential Integral Equations},
  fjournal   = {Differential and Integral Equations. An International Journal
                for Theory \& Applications},
  volume     = {36},
  year       = {2023},
  number     = {11-12},
  pages      = {877--906},
  issn       = {0893-4983},
  mrclass    = {58J65 (35B40 60J65 60J70 92C37)},
  mrnumber   = {4607799},
  mrreviewer = {Vassili\ N.\ Kolokol\cprime tsov},
  doi        = {10.57262/die036-1112-877}
}

@article{Schwarz2023,
  author   = {Schwarz, Simon and Herrmann, Michael and Sturm, Anja and
              Wardetzky, Max},
  title    = {Efficient random walks on {R}iemannian manifolds},
  journal  = {Found. Comput. Math.},
  fjournal = {Foundations of Computational Mathematics. The Journal of the
              Society for the Foundations of Computational Mathematics},
  volume   = {25},
  year     = {2025},
  number   = {1},
  pages    = {145--161},
  issn     = {1615-3375,1615-3383},
  mrclass  = {65C30 (58J65 60H35)},
  mrnumber = {4865121},
  doi      = {10.1007/s10208-023-09635-6}
}

@misc{dehoop2024,
  title           = {Geometric inverse problems on gas giants},
  author          = {Maarten V. de Hoop and Joonas Ilmavirta and Antti Kykkänen and Rafe Mazzeo},
  year            = {2024},
  eprint          = {2403.05475},
  archivepreprint = {arXiv},
  eprintclass     = {math.DG}
}

@book{evans2010pde,
  author     = {Evans, Lawrence C.},
  title      = {Partial differential equations},
  series     = {Graduate Studies in Mathematics},
  volume     = {19},
  edition    = {Second},
  publisher  = {American Mathematical Society, Providence, RI},
  year       = {2010},
  pages      = {xxii+749},
  isbn       = {978-0-8218-4974-3},
  mrclass    = {35-01},
  mrnumber   = {2597943},
  mrreviewer = {Diego\ M.\ Maldonado},
  doi        = {10.1090/gsm/019}
}

@article{holcman2014nep,
  author     = {Holcman, D. and Schuss, Z.},
  title      = {The narrow escape problem},
  journal    = {SIAM Rev.},
  fjournal   = {SIAM Review},
  volume     = {56},
  year       = {2014},
  number     = {2},
  pages      = {213--257},
  issn       = {0036-1445,1095-7200},
  mrclass    = {60J60 (35J25 35R01)},
  mrnumber   = {3201180},
  mrreviewer = {Vassili\ N.\ Kolokol\cprime tsov},
  doi        = {10.1137/120898395}
}

@article{singer2006NE,
  author   = {Singer, Amit and Schuss, Zeev and Holcman, David},
  title    = {Narrow escape. {III}. {N}on-smooth domains and {R}iemann
              surfaces},
  journal  = {J. Stat. Phys.},
  fjournal = {Journal of Statistical Physics},
  volume   = {122},
  year     = {2006},
  number   = {3},
  pages    = {491--509},
  issn     = {0022-4715,1572-9613},
  mrclass  = {82C31 (35J25 60J65 92C20)},
  mrnumber = {2205913},
  doi      = {10.1007/s10955-005-8028-4}
}

@article{holcman2004escape,
  author   = {Holcman, David and Schuss, Zeev},
  title    = {Escape through a small opening: receptor trafficking in a
              synaptic membrane},
  journal  = {J. Statist. Phys.},
  fjournal = {Journal of Statistical Physics},
  volume   = {117},
  year     = {2004},
  number   = {5-6},
  pages    = {975--1014},
  issn     = {0022-4715,1572-9613},
  mrclass  = {82C31 (37H99 37N25 60J65 92C20)},
  mrnumber = {2107903},
  doi      = {10.1007/s10955-004-5712-8}
}

@article{schuss2007NE,
  author  = {Zeev Schuss  and Amit Singer  and David Holcman },
  title   = {The narrow escape problem for diffusion in cellular microdomains},
  journal = {Proceedings of the National Academy of Sciences},
  volume  = {104},
  number  = {41},
  pages   = {16098-16103},
  year    = {2007},
  doi     = {10.1073/pnas.0706599104}
}

@article{Cheviakov2010asymptotic,
  author   = {Cheviakov, Alexei F. and Ward, Michael J. and Straube, Ronny},
  title    = {An asymptotic analysis of the mean first passage time for
              narrow escape problems. {II}. {T}he sphere},
  journal  = {Multiscale Model. Simul.},
  fjournal = {Multiscale Modeling \& Simulation. A SIAM Interdisciplinary
              Journal},
  volume   = {8},
  year     = {2010},
  number   = {3},
  pages    = {836--870},
  issn     = {1540-3459,1540-3467},
  mrclass  = {60J65 (35B25 35R60)},
  mrnumber = {2609641},
  doi      = {10.1137/100782620}
}

@book{holcman2015stoch,
  author    = {Holcman, David and Schuss, Zeev},
  title     = {Stochastic narrow escape in molecular and cellular biology: {A}nalysis and applications},
  publisher = {Springer, New York},
  year      = {2015},
  pages     = {xix+259},
  isbn      = {978-1-4939-3102-6; 978-1-4939-3103-3},
  mrclass   = {92-01 (35R60 60J28 60J70 92C37 92C40 92D20)},
  mrnumber  = {3379932},
  doi       = {10.1007/978-1-4939-3103-3}
}

@book{Bressloff2021a,
  author    = {Bressloff, Paul C.},
  title     = {Stochastic processes in cell biology. {V}ol. {I}},
  series    = {Interdisciplinary Applied Mathematics},
  volume    = {41},
  edition   = {Second},
  publisher = {Springer, Cham},
  year      = {2021},
  pages     = {xxxii+748},
  isbn      = {978-3-030-72514-3; 978-3-030-72515-0},
  mrclass   = {92C37 (60H30 82C31 82C70 92-01 92C40)},
  mrnumber  = {4390035},
  doi       = {10.1007/978-3-030-72515-0}
}

@book{Bressloff2021b,
  author    = {Bressloff, Paul C.},
  title     = {Stochastic processes in cell biology. {V}ol. {II}},
  series    = {Interdisciplinary Applied Mathematics},
  volume    = {41},
  edition   = {Second},
  publisher = {Springer, Cham},
  year      = {2021},
  pages     = {vii--xxxii and 749--1446},
  isbn      = {978-3-030-72518-1; 978-3-030-72519-8},
  mrclass   = {92C37 (60H30 82C31 82C70 92-01 92C40)},
  mrnumber  = {4398624},
  doi       = {10.1007/978-3-030-72519-8}
}

@article{benichou2008narrow,
  title     = {Narrow-Escape Time Problem: {T}ime Needed for a Particle to Exit a Confining Domain through a Small Window},
  author    = {B{\'e}nichou, Olivier and Voituriez, Raphaël},
  journal   = {Physical review letters},
  volume    = {100},
  number    = {16},
  pages     = {168105},
  year      = {2008},
  publisher = {APS},
  doi       = {10.1103/PhysRevLett.100.168105}
}

@article{singer2008NE,
  author   = {Singer, Amit and Schuss, Zeev and Holcman, David},
  title    = {Narrow escape and leakage of {B}rownian particles},
  journal  = {Phys. Rev. E (3)},
  fjournal = {Physical Review E. Statistical, Nonlinear, and Soft Matter
              Physics},
  volume   = {78},
  year     = {2008},
  number   = {5},
  pages    = {051111, 8},
  issn     = {1539-3755,1550-2376},
  mrclass  = {82C41 (82C70 92C37)},
  mrnumber = {2551367},
  doi      = {10.1103/PhysRevE.78.051111}
}

@article{AMMARI201266,
  title    = {Layer potential techniques for the narrow escape problem},
  journal  = {Journal de Mathématiques Pures et Appliquées},
  volume   = {97},
  number   = {1},
  pages    = {66-84},
  year     = {2012},
  issn     = {0021-7824},
  doi      = {10.1016/j.matpur.2011.09.011},
  author   = {Habib Ammari and Kostis Kalimeris and Hyeonbae Kang and Hyundae Lee},
  mrnumber = {2863765}
}

@article{Shiozawa2017escape,
  author     = {Shiozawa, Yuichi},
  title      = {Escape rate of the {B}rownian motions on hyperbolic spaces},
  journal    = {Proc. Japan Acad. Ser. A Math. Sci.},
  fjournal   = {Japan Academy. Proceedings. Series A. Mathematical Sciences},
  volume     = {93},
  year       = {2017},
  number     = {4},
  pages      = {27--29},
  issn       = {0386-2194},
  mrclass    = {60G17 (58J65 60F20)},
  mrnumber   = {3631821},
  mrreviewer = {Jean-Claude\ Gruet},
  doi        = {10.3792/pjaa.93.27}
}

@article{Cammarota2014asymptotic,
  author   = {Cammarota, Valentina and De Gregorio, Alessandro and Macci,
              Claudio},
  title    = {On the asymptotic behavior of the hyperbolic {B}rownian
              motion},
  journal  = {J. Stat. Phys.},
  fjournal = {Journal of Statistical Physics},
  volume   = {154},
  year     = {2014},
  number   = {6},
  pages    = {1550--1568},
  issn     = {0022-4715,1572-9613},
  mrclass  = {60F10 (58J65 60J60)},
  mrnumber = {3176419},
  doi      = {10.1007/s10955-014-0939-5}
}

@article{Cammarota2013hittingprob,
  author     = {Cammarota, Valentina and Orsingher, Enzo},
  title      = {Hitting spheres on hyperbolic spaces},
  journal    = {Theory Probab. Appl.},
  fjournal   = {Theory of Probability and its Applications},
  volume     = {57},
  year       = {2013},
  number     = {3},
  pages      = {419--443},
  issn       = {0040-585X,1095-7219},
  mrclass    = {60J60 (60D05)},
  mrnumber   = {3196780},
  mrreviewer = {Jean-Claude\ Gruet},
  doi        = {10.1137/S0040585X97986114}
}

@article{Gertsenshtein1959waveguides,
  title   = {Waveguides with Random Inhomogeneities and {B}rownian Motion In the {L}obachevsky Plane},
  author  = {Gertsenshtein, M.E. and Vasiliev, V.B},
  journal = {Theory of Probability and Its Applications},
  year    = {1959},
  month   = {Jul},
  day     = {17},
  volume  = {4},
  number  = {4},
  pages   = {391-398},
  issn    = {0040-585X},
  doi     = {10.1137/1104038}
}

@book{Aubin1982,
  author     = {Aubin, Thierry},
  title      = {Nonlinear analysis on manifolds. {M}onge-{A}mp\`ere equations},
  series     = {Grundlehren der mathematischen Wissenschaften [Fundamental
                Principles of Mathematical Sciences]},
  volume     = {252},
  publisher  = {Springer-Verlag, New York},
  year       = {1982},
  pages      = {xii+204},
  isbn       = {0-387-90704-1},
  mrclass    = {58-02 (35J60 35N99 53-02 58G30)},
  mrnumber   = {681859},
  mrreviewer = {A.\ G.\ Ramm},
  doi        = {10.1007/978-1-4612-5734-9}
}

@misc{WolframR2012,
  author       = {{Wolfram Research}},
  title        = {{ParametricNDSolve}, {W}olfram Language Function},
  year         = {2014},
  howpublished = {\url{https://reference.wolfram.com/language/ref/ParametricNDSolve.html}},
  note         = {[Accessed: 29-April-2025]}
}

@article{Gruet1996semigroup,
  author     = {Gruet, J.-C.},
  title      = {Semi-groupe du mouvement brownien hyperbolique},
  journal    = {Stochastics Stochastics Rep.},
  fjournal   = {Stochastics and Stochastics Reports},
  volume     = {56},
  year       = {1996},
  number     = {1-2},
  pages      = {53--61},
  issn       = {1045-1129},
  mrclass    = {60J60 (58G32)},
  mrnumber   = {1396754},
  mrreviewer = {Elton\ Pei\ Hsu},
  doi        = {10.1080/17442509608834035}
}

@article{ovidio2011Bessel,
  author     = {D'Ovidio, Mirko and Orsingher, Enzo},
  title      = {Bessel processes and hyperbolic {B}rownian motions stopped at
                different random times},
  journal    = {Stochastic Process. Appl.},
  fjournal   = {Stochastic Processes and their Applications},
  volume     = {121},
  year       = {2011},
  number     = {3},
  pages      = {441--465},
  issn       = {0304-4149,1879-209X},
  mrclass    = {60J65 (26A33 60J60)},
  mrnumber   = {2763091},
  mrreviewer = {Youngmee\ Kwon},
  doi        = {10.1016/j.spa.2010.11.002}
}

@article{Matsumoto2010limiting,
  author     = {Matsumoto, H.},
  title      = {Limiting behaviors of the {B}rownian motions on hyperbolic
                spaces},
  journal    = {Colloq. Math.},
  fjournal   = {Colloquium Mathematicum},
  volume     = {119},
  year       = {2010},
  number     = {2},
  pages      = {193--215},
  issn       = {0010-1354,1730-6302},
  mrclass    = {58J65 (60J65)},
  mrnumber   = {2609875},
  mrreviewer = {Maria\ Gordina},
  doi        = {10.4064/cm119-2-3}
}

@article{comtet1996diffusion,
  author     = {Comtet, Alain and Monthus, C\'ecile},
  title      = {Diffusion in a one-dimensional random medium and hyperbolic
                {B}rownian motion},
  journal    = {J. Phys. A},
  fjournal   = {Journal of Physics. A. Mathematical and General},
  volume     = {29},
  year       = {1996},
  number     = {7},
  pages      = {1331--1345},
  issn       = {0305-4470,1751-8121},
  mrclass    = {82B41 (60J99 82B44 82D30)},
  mrnumber   = {1395509},
  mrreviewer = {Aernout\ C. D. van Enter},
  doi        = {10.1088/0305-4470/29/7/006}
}

@book{hebey1999sobolev,
  author     = {Hebey, Emmanuel},
  title      = {Nonlinear analysis on manifolds: {S}obolev spaces and
                inequalities},
  series     = {Courant Lecture Notes in Mathematics},
  volume     = {5},
  publisher  = {New York University, Courant Institute of Mathematical
                Sciences, New York; American Mathematical Society, Providence,
                RI},
  year       = {1999},
  pages      = {x+309},
  isbn       = {0-9658703-4-0; 0-8218-2700-6},
  mrclass    = {58D15 (35J60 46E35 53C21 58J60)},
  mrnumber   = {1688256},
  mrreviewer = {Gilles\ Carron}
}

@book{folland1995pde,
  author    = {Folland, Gerald B.},
  title     = {Introduction to partial differential equations},
  edition   = {Second},
  publisher = {Princeton University Press, Princeton, NJ},
  year      = {1995},
  pages     = {xii+324},
  isbn      = {0-691-04361-2},
  mrclass   = {35-01},
  mrnumber  = {1357411}
}

@book{Bergh1976interpolation,
  author    = {Bergh, J\"oran and L\"ofstr\"om, J\"orgen},
  title     = {Interpolation spaces. {A}n introduction},
  series    = {Grundlehren der Mathematischen Wissenschaften},
  volume    = {No. 223},
  publisher = {Springer-Verlag, Berlin-New York},
  year      = {1976},
  pages     = {x+207},
  mrclass   = {46M35},
  mrnumber  = {482275}
}

@incollection{grieser2001basics,
  author     = {Grieser, Daniel},
  title      = {Basics of the {$b$}-calculus},
  booktitle  = {Approaches to singular analysis ({B}erlin, 1999)},
  series     = {Oper. Theory Adv. Appl.},
  volume     = {125},
  pages      = {30--84},
  publisher  = {Birkh\"auser, Basel},
  year       = {2001},
  isbn       = {3-7643-6518-8},
  mrclass    = {58J40 (35B25 35B40 35C20 35S05 58J37)},
  mrnumber   = {1827170},
  mrreviewer = {Robert\ Lauter}
}

@book{MelroseGreen,
  author     = {Melrose, Richard B.},
  title      = {The {A}tiyah-{P}atodi-{S}inger index theorem},
  series     = {Research Notes in Mathematics},
  volume     = {4},
  publisher  = {A K Peters, Ltd., Wellesley, MA},
  year       = {1993},
  pages      = {xiv+377},
  isbn       = {1-56881-002-4},
  mrclass    = {58G10 (58G15 58G25)},
  mrnumber   = {1348401},
  mrreviewer = {Rafe\ Mazzeo},
  doi        = {10.1016/0377-0257(93)80040-i}
}

@book{schulze1991psdo,
  author     = {Schulze, B.-W.},
  title      = {Pseudo-differential operators on manifolds with singularities},
  series     = {Studies in Mathematics and its Applications},
  volume     = {24},
  publisher  = {North-Holland Publishing Co., Amsterdam},
  year       = {1991},
  pages      = {vi+410},
  isbn       = {0-444-88137-9},
  mrclass    = {47G30 (35S05 47-02 58G15)},
  mrnumber   = {1142574},
  mrreviewer = {Luigi\ Rodino}
}

@article{marcus1998hardy,
  author     = {Marcus, Moshe and Mizel, Victor J. and Pinchover, Yehuda},
  title      = {On the best constant for {H}ardy's inequality in {$\mathbf
                R^n$}},
  journal    = {Trans. Amer. Math. Soc.},
  fjournal   = {Transactions of the American Mathematical Society},
  volume     = {350},
  year       = {1998},
  number     = {8},
  pages      = {3237--3255},
  issn       = {0002-9947,1088-6850},
  mrclass    = {26D15 (35J85 46E15)},
  mrnumber   = {1458330},
  mrreviewer = {Ling\ Yau\ Chan},
  doi        = {10.1090/S0002-9947-98-02122-9},
  url        = {https://doi.org/10.1090/S0002-9947-98-02122-9}
}

@book{mclean2000strongly,
  author     = {McLean, William},
  title      = {Strongly elliptic systems and boundary integral equations},
  publisher  = {Cambridge University Press, Cambridge},
  year       = {2000},
  pages      = {xiv+357},
  isbn       = {0-521-66332-6; 0-521-66375-X},
  mrclass    = {35J45 (47F05 47G10 47N20 65N38)},
  mrnumber   = {1742312},
  mrreviewer = {Dorina\ I.\ Mitrea}
}

@unpublished{MelroseDAOMWC,
  author = {Melrose, Richard B.},
  title  = {Differential analysis on manifolds with corners},
  year   = {1996},
  note   = {Unpublished works, Massechusetts Institute of Technology, Available at: \url{https:math.mit.edu/~rbm/book.html}}
}

@book{Abramowitz1964Handbook,
  author     = {Abramowitz, Milton and Stegun, Irene A.},
  title      = {Handbook of mathematical functions with formulas, graphs, and
                mathematical tables},
  series     = {National Bureau of Standards Applied Mathematics Series},
  volume     = {No. 55},
  publisher  = {U. S. Government Printing Office, Washington, DC},
  year       = {1972},
  pages      = {xiv+1046},
  mrclass    = {33.00 (65.05)},
  mrnumber   = {167642},
  mrreviewer = {D.\ H.\ Lehmer},
  edition = {10th}
}

@book {klenke2020probability,
    AUTHOR = {Klenke, Achim},
     TITLE = {Probability theory---a comprehensive course},
    SERIES = {Universitext},
   EDITION = {Third},
 PUBLISHER = {Springer, Cham},
      YEAR = {2020},
     PAGES = {xiv+716},
      ISBN = {978-3-030-56402-5; 978-3-030-56401-8},
   MRCLASS = {60-01 (60Fxx 60Gxx 60H10 60Jxx)},
  MRNUMBER = {4201399},
       DOI = {10.1007/978-3-030-56402-5},
       URL = {https://doi.org/10.1007/978-3-030-56402-5},
}

\end{document}